\documentclass[11pt]{amsart}
\usepackage[top=1in, bottom=1in, left=1in, right=1in]{geometry}

\usepackage{graphicx}
\usepackage[dvipsnames]{xcolor}
\usepackage{amsfonts,amstext,amsmath,amssymb,amsthm,bbm,stmaryrd}
\DeclareMathAlphabet{\mathscr}{OT1}{pzc}{m}{it}
\usepackage{enumerate}
\usepackage[normalem]{ulem}
\usepackage{comment,constants}
\usepackage{hyperref}
\usepackage{mathtools}
\usepackage{multirow}
\newcommand{\stkout}[1]{\ifmmode\text{\sout{\ensuremath{#1}}}\else\sout{#1}\fi}
\usepackage{cancel}

\newtheorem{thm}{Theorem}[section]
\newtheorem{prop}[thm]{Proposition}
\newtheorem{lemma}[thm]{Lemma}
\newtheorem{cor}[thm]{Corollary}

\newtheorem*{asm*}{Assumptions}
\newtheorem{asm}{Assumption}

\theoremstyle{remark}
\newtheorem{rem}[thm]{Remark}
\newtheorem*{rem*}{Remark}

\theoremstyle{definition}
\newtheorem{defn}[thm]{Definition}

\newcommand{\ra}{\rightarrow}

\newcommand{\N}{\mathbb N}     % For Natural numbers
\newcommand{\R}{\mathbb R}     % For Real numbers
\newcommand{\Z}{\mathbb Z}     % For Integers
\newcommand{\cal}{\mathcal}
\renewcommand{\epsilon}{\varepsilon}

\newcommand{\fl}[1]{\lfloor #1 \rfloor}  % Floor function
\newcommand{\ceil}[1]{\lceil #1 \rceil}  % Ceiling function
\newcommand{\bint}[1]{\llbracket{#1}\rrbracket}
\newcommand{\ind}[1]{ \mathbbm{1}_{\{ #1 \}} } % Indicator of a set

\newcommand{\bzmk}{\Bar{Z}^{m,k}}

\renewcommand{\dim}{\mathscr{d}}
\newcommand{\edsrn}{\fl{\delta_1\epsilon\sqrt{n}}}
\newcommand{\esrn}{\fl{\epsilon\sqrt{n}}}
\renewcommand{\kappa}{\varkappa}

\newcommand{\twe}{\tilde{W}^\epsilon}
\newcommand{\tie}{\tilde{I}^\epsilon}
\newcommand{\tse}{\tilde{S}^\epsilon}
\newcommand{\txe}{\tilde{\cal X}^\epsilon}
\newcommand{\twen}{\tilde{W}^{\epsilon,n}}
\newcommand{\tien}{\tilde{I}^{\epsilon,n}}
\newcommand{\tsen}{\tilde{S}^{\epsilon,n}}
\newcommand{\txen}{\tilde{\cal X}^{\epsilon,n}}
\newcommand{\tenk}{T^{\epsilon,n}_k}
\newcommand{\xenk}{X_{\tenk}}
\newcommand{\senk}{S_{\tenk}}
\newcommand{\ienk}{I_{\tenk}}
\newcommand{\tcp}{\tau_{\text{coup}}}
\newcommand{\wo}{\overline{w}}
\newcommand{\wu}{\underline{w}}

\renewcommand{\P}{\mathbb{P}}   
\newcommand{\E}{\mathbb{E}}

\newcommand{\tvmj}{\tilde{V}^+_{m,j}}
\newcommand{\tvmjj}{\tilde{V}^+_{m,j-1}}

\newcommand{\tvmmk}{\tilde{V}^+_{m,\fl{km^{1/4}}}}
\newcommand{\tvmmkk}{\tilde{V}^+_{m,\fl{(k-1)m^{1/4}}}}

\DeclareMathOperator{\Var}{Var}

\newcommand{\br}{\mathbf{r}}

\let\emptyset\varnothing

\title[Recurrent excited random walks]{Convergence of random walks with
  markovian cookie stacks to Brownian motion perturbed at
  extrema}

\author{Elena Kosygina}
\address{Elena Kosygina\\One Bernard Baruch Way \\ Department of Mathematics, Box B6-230 \\ Baruch College \\ New York, NY 10010 \\ USA}
\email{elena.kosygina@baruch.cuny.edu}
\urladdr{http://www.baruch.cuny.edu/math/elenak/}
\thanks{The collaboration of the authors was supported in part by the Simons Foundation through Collaboration Grants for
Mathematicians \#209493 (EK) and \#635064 (JP)}

\author{Thomas Mountford}
\address{Thomas Mountford\\\'Ecole Polytechnique F\'ed\'eral de Lausanne\\Department of Mathematics\\EPFL SB MATH PRST\\
MA B1 517 (Bâtiment MA)
Station 8
CH-1015 Lausanne\\
Switzerland}
\email{thomas.mountford@epfl.ch}
\urladdr{http://people.epfl.ch/thomas.mountford}
\author{Jonathon Peterson}
\address{Jonathon Peterson\\Purdue University\\Department of Mathematics\\150 N University Street\\West Lafayette, IN  47907\\USA}
\email{peterson@purdue.edu}
\urladdr{http://www.math.purdue.edu/~peterson}
\subjclass[2010]{Primary 60K35; Secondary 60F17, 60J55}
\keywords{Excited random walk, markovian cookie stacks, Brownian motion perturbed at its extrema, branching-like processes, generalized Ray-Knight theorems}

\begin{document}

\begin{abstract}
  We consider one-dimensional excited random walks (ERWs) with
  i.i.d.\ markovian cookie stacks in the non-boundary
  recurrent regime.  We prove that under diffusive scaling such
    an ERW converges in the standard Skorokhod topology to a multiple
    of Brownian motion perturbed at its extrema (BMPE). All parameters
    of the limiting process are given explicitly in terms of
    those of the cookie markov chain at a single site.
    While our results extend the results in \cite{dkSLRERW}
    (ERWs with boundedly many cookies per site) and \cite{kpERWPCS}
    (ERWs with periodic cookie stacks), the approach taken is very
    different and involves coarse graining of both the ERW and the
    random environment changed by the walk.  Through a careful
    analysis of the environment left by the walk after each
    ``mesoscopic'' step, we are able to construct a coupling of the
    ERW at this ``mesoscopic'' scale with a suitable discretization of
    the limiting BMPE. The analysis is based on generalized Ray-Knight
    theorems for the directed edge local times of the ERW stopped at
    certain stopping times and evolving in both the original random cookie
    environment and (which is much more challenging) in the
    environment created by the walk after each ``mesoscopic'' step.
\end{abstract}

\maketitle

\section{Introduction and the main result}\label{intro} 

\subsection{Introduction} Over the past several decades, a number of
different one-dimensional self-interacting random walks have been
studied through what may be called a ``Ray-Knight'' approach.  It was
observed that for these walks the joint distributions of edge local
times have the structure of a Markov chain, and by analyzing this
Markov chain one is able to obtain information about the original
self-interacting random walk.  Examples of this approach are numerous
and include
\cite{kksStable,tTSAWGBR,tTSAW,tGRK,bsCRWspeed,bsRGCRW,TV08,kzPNERW,Pin10,kmLLCRW,dkSLRERW,pERWLDP,kzERWsurvey,kzEERW,MPV14,DK14,kosERWPC,CdHPP16,kpERWMCS,PT17,HLSH18,Tra18}.

We refer to this line of thought as a ``Ray-Knight'' approach in
reference to the Ray-Knight theorems for Brownian motion which give a
description of the local time profiles of a standard Brownian motion
stopped when the local time at a fixed site exceeds a fixed level. The
Ray-Knight theorems describe these local time profiles (viewed as
processes in the spatial coordinate) as a gluing together of certain
diffusion processes. In fact, for several models
of self-interacting random walks one can prove that the
Markov chains which correspond to the directed (or undirected) edge
local times of the walk have scaling limits which are diffusion
processes. This was first noticed by T\'oth in
\cite{tTSAWGBR,tTSAW,tGRK} and, more recently, found to be true for
other models, \cite{kzPNERW,kmLLCRW,kpERWMCS,PT17}. Yet the goal had
now become different, namely, to study properties of the original
process from information about its local times and not the other way
around as in the classical Ray-Knight theorems.

Regarding scaling limits of self-interacting random walks, the
Ray-Knight approach is easier to use when the
process is transient, i.e.\ when
with probability one it goes to $+\infty$ (or $-\infty$) as the time
tends to infinity, see  
% In this case, obtaining scaling limits of the random walk is reduced
% to obtaining tail asymptotics (or in the case of Gaussian limits
% just second moment bounds) for two quantities of the Markov chain
% related to the directed edge local times: (1) the time between
% successive visits to zero of the Markov chain, and (2) the sum of
% the Markov chain between successive visits to zero.
\cite{kksStable,bsRGCRW,kzPNERW,kmLLCRW,kpERWMCS,PT17,Tra18}. This is because
the Ray-Knight information on local times can be readily used to
deduce limiting distributions for the hitting times of the random
walk, and if the walk is transient to the right then by inverting the
role of time and space one can deduce a limiting distribution for the
running maximum of the walk.
% (this is analogous to how the renewal CLT is proved from the classical CLT). 
If one can also control the distance between the walk and its
running maximum, then one obtains a limiting
distribution for the walk. On the other
hand, proving the existence of a scaling
limit through the Ray-Knight approach when the
walk is recurrent (in the sense that it returns to the starting point
infinitely often) is a more delicate task.
%For instance, in \cite{tTSAWGBR,tTSAW,tGRK} for a family of recurrent self-interacting random walks T\'oth was able to determine the correct scaling exponent for the walks but he was unable to prove a limiting distribution for the walk. T\'oth was able to characterize the limiting distribution for the endpoint of the walk if one could prove the limiting distribution exists, but his methods were not sufficient to prove the existence of a limiting distribution for the endpoint of the walk (much less for the entire path of the walk). 
In the aforementioned series of papers, T\'oth introduced
  generalized Ray-Knight theorems and showed how to exploit them to
  show the convergence in distribution of the endpoint of a class of
rescaled ``recurrent'' self-interacting random walks along a sequence
of random geometric times independent of the walk.  For one particular
model, Mountford, Pimentel and
Valle \cite{MPV14} were able to obtain additional
estimates that allowed them to prove the
convergence of one dimensional distributions of the walk with
T\'oth's method. Even in this case, however, characterization of
multi-dimensional limiting distributions using this ``roadmap''
seems out of reach.

In this paper, we show how a Ray-Knight approach can be used for a
particular self-interacting random walk model (excited random walks with
markovian cookie stacks) to prove not just the convergence of finite
dimensional distributions but a full functional limit theorem.
%that in the recurrent case the paths of the walks converge to a Brownian motion perturbed at its extrema. 
Our method is completely different from that of T\'oth in that instead
of ``inverting'' the Ray-Knight theorems to get information on the
distribution of the endpoint of the walk, we use information from the
Ray-Knight-type results to construct a coupling of the walk with the
conjectured scaling limit (a Brownian motion perturbed at its
extrema). It is also completely different from methods used in
\cite{dkSLRERW,DK14,kpERWPCS,HLSH18} for variants of this model where
the random walk was decomposed in a natural way into two parts, a
martingale and an accumulated drift, each of which contributed the
corresponding part of a similar decomposition of the limiting process.
We refer to \cite[Section 5]{kpERWPCS} for a discussion as to why the
same kind of decomposition cannot work for the general model
considered in the current paper.  The main approach in this paper is
robust in the sense that it could, in theory, be applied to other
self-interacting random walks as long as one can prove the type of
Ray-Knight theorems for the walk that are needed.  Since there are a
number of self-interacting random walks for which similar (but weaker)
Ray-Knight theorems have been proved but for which full limiting
distributions have not yet been obtained (e.g.,
\cite{tGRK,Tra18}), it may be possible to adapt our techniques to
get functional limit theorems for these random
walks as well.

%The classical Ray-Knight Theorems are a description of the local time profile of a one-dimensional standard Brownian motion, stopped when the local time at a given site first reaches a fixed level. The Ray-Knight Theorems describe the local time profile (viewed as a process in the spatial coordinate) as a gluing together of certain diffusion processes. 

%In this paper we will study the scaling limits of a particular model of self-interacting random walk model, excited random walks in  cookie stacks which were introduced in \cite{kpERWMCS}. The results in \cite{kpERWMCS} extended a number of previously known results for excited random walks to a more general setting. These results included a characterization of recurrence/transience, a characterization of ballisticity (non-zero limiting speed), and limiting distributions for the transient cases. The results in \cite{kpERWMCS} however did not include an extension of the known limiting distributions for the paths of the random walk in the recurrent case, however. 

\subsection{Excited random walks with markovian cookie stacks}

Excited random walks (ERW), sometimes also called cookie random walks, are a model of self-interacting random walks where the transition probabilities of the walk depend on the local time of the walk at the present site. This model was first introduced by Benjamini and Wilson in \cite{bwERW} where the transition probabilities were only different on the first visit to a site (only a single excitation at each site). The model was then generalized in \cite{zMERW} and \cite{kzPNERW} to include multiple excitations at each site and to allow for randomness in the excitation environment.  
%We will consider in this paper a further generalization introduced in \cite{kpERWMCS} which we refer to as excited random walks with markovian cookie stacks. 

For one-dimensional ERW, the model is described as follows. 
A \emph{cookie environment} is an element $\omega = \{ \omega_x(j) \}_{x\in\Z, \, j\geq 1} \in (0,1)^{\Z \times \N}$. Given a fixed cookie environment $\omega$ we can then construct a random walk $\{X_n\}_{n\geq 0}$ as follows. The walk starts at $X_0 = 0$ and then when at the site $x$ for the $j$-th time steps to the right with probability $\omega_x(j)$ or to the left with probability $1-\omega_x(j)$. That is, letting $P_\omega$ denote the law of the process in the cookie environment $\omega$ we have 
\[
 P_\omega( X_{n+1} = X_n+1 \, | \, X_0, X_1,\ldots, X_n )
% = 1 -  P_\omega( X_{n+1} = X_n - 1 \, | \, X_0, X_1,\ldots, X_n )
 = \omega_{X_n}\left( \sum_{i=0}^n \ind{X_i = X_n} \right). 
\]
The distribution $P_\omega$ of the walk in a fixed environment is called the \emph{quenched} law. 
We will assume that the cookie environment $\omega$ is chosen randomly according to some distribution $\P$ on cookie environments so that the annealed law of the walk $P$ is defined by averaging the quenched law with respect to $\P$. That is $P(\cdot) = \E[ P_\omega(\cdot) ]$. 

The ``cookie'' terminology for these walks dates back to \cite{zMERW}
and comes from the following interpretation of the walk. Each site has
a (possibly infinite) stack of cookies initially at that site. The
random walker then always eats the top remaining cookie at his current
location; the cookie induces some excitation/drift to the walker which
determines the law of his next step.  If there is a finite $M<\infty$
for which $\omega_x(j) = 1/2$ for all $x\in \Z$ and $j>M$ then we say
that there are only $M$ cookies per site and the walker takes steps
which are equally likely to the right or left when at a site where all
the cookies are already eaten.  With this cookie terminology we will
refer to $\omega_x(j)$ as the $j$-th cookie at site $x$ and
$\omega_x = \{ \omega_x(j) \}_{j\geq 1}$ as the \emph{cookie stack} at
site $x$.

To give some additional structure to the model we need to describe the
distribution of the cookie environment $\P$. We will assume that a
cookie stack at each site is generated by an independent copy of a
finite state Markov chain. 
\begin{asm}\label{asm:Markov}
  There is a function $p:\{1,2,\ldots,N\} \to (0,1)$ such that
  $\omega_x(j)=p(R^x_j)$, $j\in\N$, where $\{R^x_j\}_{j\geq 1}$,
  $x\in\Z$, are i.i.d.\ Markov chains on $\{1,2,\ldots,N\}$ with
  transition matrix $K$ and initial distribution $\eta$. The Markov
  chain $\{R^x_j\}_{j\geq 1}$ has a unique stationary distribution
  $\mu$ and %the asymptotic average of cookie values within a stack is
  $\bar{p} := \sum_{i=1}^N \mu(i) p(i) = \frac{1}{2}$.
\end{asm}
% We also require the chains to be i.i.d.:
% \begin{asm}\label{asm:iid}
%  The processes $\{R^x_j\}_{j\geq 1}$ are i.i.d. over $x$ in $\Z$.
% \end{asm}
% Due to Assumption \ref{asm:iid}, to describe the measure $\P$ we
% need only to give the distribution of the cookie stack at a single
% site.
% We will assume that the cookie stack at each site is generated
% by a finite state Markov chain.

The assumption of markovian cookie stacks was first made in
\cite{kpERWMCS} where it was shown that a number of asymptotic
behaviors of the walk (such as recurrence/transience, ballistic
behavior, and limiting distributions for the transient cases) can be
explicitly characterized.  If the condition $\bar{p} = 1/2$ is
dropped, then clearly the random walk should have some asymptotic
drift to the right/left. In fact, in \cite{kpERWMCS} it was shown that
if $\bar{p}\neq 1/2$ then the walk has a non-zero limiting speed and
satisfies a CLT for a limiting distribution  under the
  annealed measure $P$. However, if $\bar{p} = 1/2$ then the behavior
can be much more varied.  For instance, the walk can be either
recurrent or transient depending (in a complicated but explicit way) on
the parameters of the model.
\begin{thm}[\cite{kpERWMCS}]\label{th:rtcon}
  There exist two parameters $\theta^+$ and $\theta^-$
%(NB: we have shown that there is a function, say, $\theta(\eta,K,p(\cdot))$ such that the parameters are $\theta^+=\theta(\eta,K,p(\cdot))$ and $\theta^-=\theta(\eta,K,1-p(\cdot))$.) 
which characterize
  the recurrence/transience of the excited random walk as follows.
\begin{enumerate}
 \item If $\theta^+ > 1$ then $P(\lim_{n\to\infty} X_n = +\infty) = 1$. 
 \item If $\theta^- > 1$ then $P( \lim_{n\to\infty} X_n = -\infty) = 1$. 
 \item If $\max\{\theta^+,\theta^-\} \leq 1$ then $P( \liminf_{n\to\infty} X_n = -\infty, \, \limsup_{n\to\infty} X_n = +\infty ) = 1$. 
\end{enumerate}
\end{thm}

\begin{rem}
  It was shown in \cite{kpERWMCS} that the parameter $\theta^+$ can be
  written as an \emph{explicit} function
  $\theta^+ = \Theta(\eta,K,p(\cdot))$ of the parameters $\eta$, $K$
  and $p(\cdot)$. Moreover, $\theta^- = \Theta(\eta,K,1-p(\cdot))$ is
  given by the same function but with $p(\cdot)$ replaced by
  $1-p(\cdot)$.  In the present paper, the parameters $K$ and
  $p(\cdot)$ will always be fixed, but we will at times be
  interested in cookie environments with different initial cookie
  distributions. Thus, for any distribution $\eta'$ on
  $\{1,2,\ldots,N\}$ we will write $\theta^+(\eta')$ and
  $\theta^-(\eta')$ for $\Theta(\eta',K,p(\cdot))$ and
  $\Theta(\eta',K,1-p(\cdot))$, respectively.  In the special case
  where $\eta' = \eta$ from Assumption \ref{asm:Markov} we will just
  write $\theta^\pm$ instead of $\theta^\pm(\eta)$.

\end{rem}

\begin{rem}
It can be shown from the explicit formulas for $\theta^{\pm}$  (see Section~\ref{sec:parameter}) that   
$\theta^+ + \theta^- < 1$ so
 that Theorem \ref{th:rtcon} gives a complete characterization of
 recurrence and transience for excited random walks with markovian
 cookie stacks \cite[Section 4]{kpERWMCS}.
\end{rem}

\begin{rem}[$M$ cookies per stack]
  % While the general formulas for $\theta^+$ and $\theta^-$ are quite complicated,
In a particular case when the Markov chain $\{R^x_j\}_{j\ge 1}$ has an
  absorbing state $a\in\{1,2,\dots,N\}$ (which is unique by Assumption~\ref{asm:Markov}) with $p(a)=1/2$ and reaches it by the $M$-th step with
  probability 1, that is when  %for all $x\in\Z$
  \begin{equation}\label{Mcookies}
   \P( \omega_x(j) = 1/2, \forall j > M ) = 1, 
  \end{equation}
  the formulas for $\theta^\pm$ have a particularly simple
    form, % . In this
  % case we have
  namely, $\theta^+ = -\theta^-=\delta$ where
 \begin{equation}\label{deltadef}
 \delta = \sum_{j=1}^M \E\left[ 2\omega_0(j) - 1 \right]. 
\end{equation}
For additional examples we refer to \cite[Section 1.4]{kpERWMCS}.
\end{rem}

In addition to the criteria for recurrence/transience
stated in Theorem \ref{th:rtcon}, the paper \cite{kpERWMCS}
  also contains characterizations of ballisticity (non-zero limiting
linear speed) and limit laws in the transient cases.  These
results generalized % a number of results
some of those that had been proved earlier in
\cite{zMERW,bsCRWspeed,bsRGCRW,kzPNERW,kmLLCRW} for ERWs with $M$
cookies per stack.  A notable omission, % to the results in \cite{kpERWMCS}
however, was the limiting behavior in the recurrent case when
  $\max\{\theta^+,\theta^-\}<1$. This is the focus of the present
paper.

\subsection{Main results}
In the case when there are $M$ cookies per stack and cookies
  stacks are i.i.d., functional limit theorems for recurrent ERW were
first obtained by Dolgopyat, \cite{dCLTERW}, and Dolgopyat and
Kosygina, \cite{dkSLRERW}.  Before stating their and our results we
need the following definition.
\begin{defn}
  For any $\alpha,\beta<1$, a Brownian motion
    $(\alpha,\beta)$-perturbed at its extrema ($(\alpha,\beta)$-BMPE)
    is a process $\{W(t)\}_{t\geq 0}$ started at $W(0) = 0$,
    continuous in $t$, and solving the functional equation
\begin{equation}\label{BMPEdef}
 W(t) = B(t) + \alpha\, \sup_{s\leq t} W(s) + \beta\, \inf_{s\leq t} W(s),
\end{equation}
where here and throughout the paper $\{B(t)\}_{t\geq 0}$ is a standard one-dimensional Brownian motion. 
\end{defn}

While it is not obvious that the functional equation
\eqref{BMPEdef} has a solution, it was shown in \cite{pwPBM,cdPUPBM}
that for all $\alpha,\beta<1$ there is a pathwise unique continuous
solution and it is adapted to the filtration of $B$.  In the special
case when $\alpha=0$ or $\beta = 0$ the solution can be made
explicit. For instance, if $\beta = 0$ then as shown in
\cite[p.\,242]{cpyBetaPBM} % we can write the BMPE explicitly
% as a linear combination of the Brownian motion and its running
% maximum:
\begin{equation}\label{BMPE-ex}
 W(t) = B(t) + \frac{\alpha}{1-\alpha} B^*(t), \quad \text{where } B^*(t) = \sup_{s\leq t} B(s). 
\end{equation}

In the theorem below and throughout the remainder of the paper % , when
% used in the context of convergence of processes
the symbol
$\overset{J_1}{\Longrightarrow}$ will denote convergence in
distribution with respect to the Skorokhod $J_1$ topology.

\begin{thm}[\cite{dkSLRERW}]\label{th:McookieLimit}
Suppose that $\omega_x,\ x\in\Z$, are i.i.d., \eqref{Mcookies}
    holds, and
    $\P(\omega_x(j)\in(0,1)\ \forall j\in\{1,2,\dots,M\})>0$. Let
    $\{X_n\}_{n\geq 0}$ be an ERW in this cookie environment and 
    $\delta$ be given by \eqref{deltadef}. Then the following
    statements hold with respect to the averaged measure $P$.
%such that $\delta = \E[\sum_{j=1}^M (2\omega_0(j)-1) ] \in [-1,1]$. 
\begin{enumerate}
% the running maximum of a Brownian motion. 
% $\{ \frac{ X_{\fl{nt}} }{ a \sqrt{n} \log n } \}_{t\geq 0}$ converges in distribution to $\{ \sup_{s\leq t} B(s) \}_{t\geq 0}$, the running maximum of a standard Brownian motion. 
  % \item If $\delta=-1$, then there exists a constant $a>0$ such that $\{ \frac{ X_{\fl{nt}} }{ a \sqrt{n} \log n } \}_{t\geq 0} \overset{J_1}{\underset{n\to\infty}{\Longrightarrow}} \{B^\sharp(t)\}_{t\geq 0}$, where $B^\sharp(t) = \inf_{s\leq t} B(t)$ is the running minimum of a Brownian motion. 
%$\{ \frac{ X_{\fl{nt}} }{ a \sqrt{n} \log n } \}_{t\geq 0}$ converges in distribution to $\{ \inf_{s\leq t} B(s) \}_{t\geq 0}$, the running minimum of a standard Brownian motion. 
  \item  If $\delta \in (-1,1)$, then $\left\{\frac{X_{\fl{nt}}}{\sqrt{n}} \right\}_{t\geq 0} \overset{J_1}{\underset{n\to\infty}{\Longrightarrow}} \left\{W(t)\right\}_{t\geq 0}$, where $W$ is a $(\delta,-\delta)$-BMPE. 
% $\{\frac{X_{\fl{nt}}}{\sqrt{n}} \}_{t\geq 0}$ converges in distribution to a $(\delta,-\delta)$-BMPE. 
  \item If $\delta\in\{-1,1\}$, then there exists a constant $a>0$ such that $\left\{ \frac{ \delta X_{\fl{nt}} }{ a \sqrt{n} \log n } \right\}_{t\geq 0} \overset{J_1}{\underset{n\to\infty}{\Longrightarrow}} \left\{B^*(t)\right\}_{t\geq 0}$.
  \end{enumerate}
\end{thm}

\begin{rem}
  Note that the limit in the boundary cases $\delta \in \{-1,1\}$ is
  somewhat surprising since the ERW is recurrent but the
  scaling limit is transient.  In the non-boundary cases, it is not
  hard to see that BMPE is a reasonable scaling limit. Indeed, since
  there are only $M$-cookies per site it is natural to expect that the
  scaling limit should be a process that behaves like a Brownian
  motion when not near the running minimum or maximum and experiences
  some sort of additional drift when at the minimum or maximum.
\end{rem}

%The main results in the present paper are an extension of Theorem \ref{th:McookieLimit} to ERW with markovian cookie stacks. 
In this paper we show that the results of Theorem \ref{th:McookieLimit} can be extended to the case of markovian cookie stacks. Both theorems below hold with respect to the averaged measure $P$.
\begin{thm}\label{main}
If $\max\{ \theta^+, \theta^-\} < 1$, then $\left\{ \frac{ X_{\fl{nt}} }{ a \sqrt{n}} \right\}_{t\geq 0} \overset{J_1}{\underset{n\to\infty}{\Longrightarrow}} \left\{W(t)\right\}_{t\geq 0}$ where $W$ is a $(\theta^+,\theta^-)$-BMPE and the constant $a = \sqrt{1-\theta^+-\theta^-}>0$. 
\end{thm}

\begin{thm}\label{th:BoundaryCase}
 If $\theta^+ = 1$ then there exists a constant $a>0$ such that $\left\{ \frac{ X_{\fl{nt}} }{ a \sqrt{n} \log n } \right\}_{t\geq 0}\overset{J_1}{\underset{n\to\infty}{\Longrightarrow}} \left\{B^*(t)\right\}_{t\geq 0}$. 
 Similarly, if $\theta^- = 1$ then the above statement holds with $-X_{\fl{nt}}$ in place of  $X_{\fl{nt}}$.
% there exists a constant $a>0$ such that $\left\{ \frac{ -X_{\fl{nt}} }{ a \sqrt{n} \log n } \right\}_{t\geq 0}\overset{J_1}{\underset{n\to\infty}{\Longrightarrow}} \left\{B^*(t)\right\}_{t\geq 0}$. 
\end{thm}

We have separated the statements of the scaling limits in the boundary
and non-boundary cases because the proof techniques are completely
different. In fact, the proof of the scaling limits for recurrent ERW
in the boundary case ($\theta^+=1$ or $\theta^-=1$) is exactly
the same as that in \cite{dkSLRERW} for the case of $M$
cookies per stack and depends only on certain tail estimates for the
directed edge local time processes that have already been obtained for
the case of Markov cookie stacks.  See \cite[p. 8]{kpERWPCS} and
\cite[Theorem 2.7]{kpERWMCS} for further details.

The proof of Theorem \ref{main}, on the other hand, is quite different
from previous cases and thus is the focus of the remainder of the
paper.  As we have noted above, BMPE was already
shown to be the scaling limit of ERW with $M$ cookies per stack, but
there have also been a few other self-interacting random walks
which have been shown to converge to
BMPE. We list all cases we are aware of below.
\begin{enumerate}
\item Random walk with partial reflection at extrema \cite{Dav96}. In
  this walk the random walk has a drift when at its running
  maximum/minimum and jumps to the left/right with equal probability
  otherwise. This walk is clearly a discrete analog of the BMPE.
% and was the first instance where a BMPE was obtained as a scaling limit of a self-interacting walk. 
\item ERW with $M$ cookies per stack with $\delta \in (-1,1)$. As noted above this was proved in \cite{dkSLRERW}. 
\item ERW with periodic cookie stacks with $\max\{\theta^+,\theta^-\} < 1$.  This special case of Theorem \ref{main} was proved in \cite{kpERWPCS}. 
\item Broken rotor walk \cite{HLSH18}. This walk, though not described
  as such in the original paper, can be seen as an ERW with markovian
  cookie stacks where the Markov chain is a two state Markov chain
  with transition matrix
  $K = \begin{pmatrix} 1-\alpha & \alpha \\ \alpha &
    1-\alpha \end{pmatrix}$ and where the cookie values are degenerate
  in that $p(1) = 1$ and $p(0) = 0$ (that is, the behavior of the walk
  is deterministic given the realization of the cookie environment).
\end{enumerate}
In all of these previous papers, the proof followed the same general
strategy. First, one proves that the random walk can be approximated
by a martingale plus a linear combination of the running maximum and
minimum of the walk. Next, one proves that the martingale term in this
approximation converges to Brownian motion under diffusive
scaling. Finally, one proves tightness for the random walk process
under diffusive scaling and from this concludes that any scaling limit
must satisfy a functional equation like \eqref{BMPEdef} in the
definition of BMPE. This strategy does not seem to work for the
current model, at least not without involving an intermediate scale
and an additional control on the environment. As mentioned in the
  introduction, a more detailed discussion of the problems arising
when implementing this approach can be found in \cite[Section
5]{kpERWPCS}.

\subsection{Ideas of the proof}

%We discuss briefly here some of the ideas of the proof of Theorem \ref{main}.

The main idea of our proof is to use information on the local time
processes to determine the movement of the ERW on a macroscopic scale.
For a BMPE this is understood through the Ray-Knight type theorems
proved in \cite{cpyBetaPBM}.  For a $(\theta^+,\theta^-)$-BMPE $W$ let
$\{\ell_{x,t}^W\}_{x\in\R, \, t\geq 0}$ be the local time process of
$W$, and if $\tau_x^W = \inf\{ t\geq 0: \, W(t) = x\}$ is the hitting
time of $x \in \R$ then it was shown in \cite[Theorem 3.4]{cpyBetaPBM}
that $\{\ell_{x,\tau_{-1}^W}^W \}_{x\geq -1}$ is a gluing together of
two Bessel squared processes; that is,
$\{\ell_{x,\tau_{-1}^W}^W \}_{x \in [-1,0]}$ is a Bessel squared
process of dimension $2(1-\theta^-)$ started at 0 and
$\{\ell_{x,\tau_{-1}^W}^W \}_{x \geq 0}$ is a Bessel squared process
of dimension $2\theta^+$ which is killed when reaching zero.  See
Figure \ref{fig:RayKnight}.  From this Ray-Knight theorem for BMPE we
can deduce some information about macroscopic behavior of $W$. For
instance, the event that $W$ exits the interval $(-1,1)$ to the left
is equal to the event that the local time process
$\{\ell_{x,\tau_{-1}^W}^W\}_{x \geq -1}$ dies out somewhere in
$(0,1)$. Moreover, when this event happens, the location where the
local time process dies out is equal to the running maximum of $W$ by
time $\tau_{-1}$ and the area under the curve of
$x\mapsto \ell_{x,\tau_{-1}^W}^W$ is equal to the time for $W$ to exit
the interval $(-1,1)$. A similar analysis of the local time profile at
time $\tau_1^W $ can be used to determine the distribution of the
exit time and the running minimum of $W$ when the process exits
$(-1,1)$ to the right.

The above explains how one can describe the \emph{initial} macroscopic
behavior of a BMPE using the Ray-Knight theorems for BMPE. However,
understanding the macroscopic behavior of the BMPE at later times is
a little more complicated because the BMPE $W$ is
not a Markov process. Nevertheless, if we define
\begin{equation}
  \label{IS}
  I(t)=\inf_{s\le t}W(s)\quad \text{and}\quad S(t)=\sup_{s\le t}W(s),\quad t\ge 0, 
\end{equation}
to be the running minimum and maximum of $W$ respectively, then
$\{(I(t),W(t),S(t))\}_{t\ge 0}$ is a Markov process. Suppose that at
time $t$ we have $(I(t),W(t),S(t)) = (w+\wu,w,w+\wo)$ for some
$w \in \R$ and $\wu\leq 0 \leq \wo$ and we want to know the
probability that $W$ will subsequently exit the interval
$(w-1,w+1)$ to the left. By the Markov property  and
  translation invariance of Brownian motion we can then consider the
process started at $(I(0),W(0),S(0)) = (\wu,0,\wo)$ (that is, started
with artificial non-zero minimum and maximum) and use the local time
profiles stopped at times $\tau_{-1}^W$ or $\tau_1^W$ as
before. However, in this case since the minimum and maximum are not
initially zero the distributions of the local time profiles are
different.
In this case (see, for example, \cite[Proposition 2.1]{cdhULL}) if  we
start from $(I(0),W(0),S(0))= (\wu,0,\wo)$ then
$\{\ell_{x,\tau_{-1}^W}^W\}_{x \geq -1}$ is a gluing together of (up
to) 4 squared Bessel processes of (1) dimension $2(1-\theta^-)$ on the
interval $[-1, \wu \vee -1]$, (2) dimension 2 on the interval
$[\wu \vee -1 ,0]$, (3) dimension 0 on the interval $[0,\wo]$, and (4)
dimension $2\theta^+$ on the interval $[\wo,\infty)$. See Figure \ref{fig:RayKnight}.
\begin{figure}[ht]
 \includegraphics[width=0.45\textwidth,page=1]{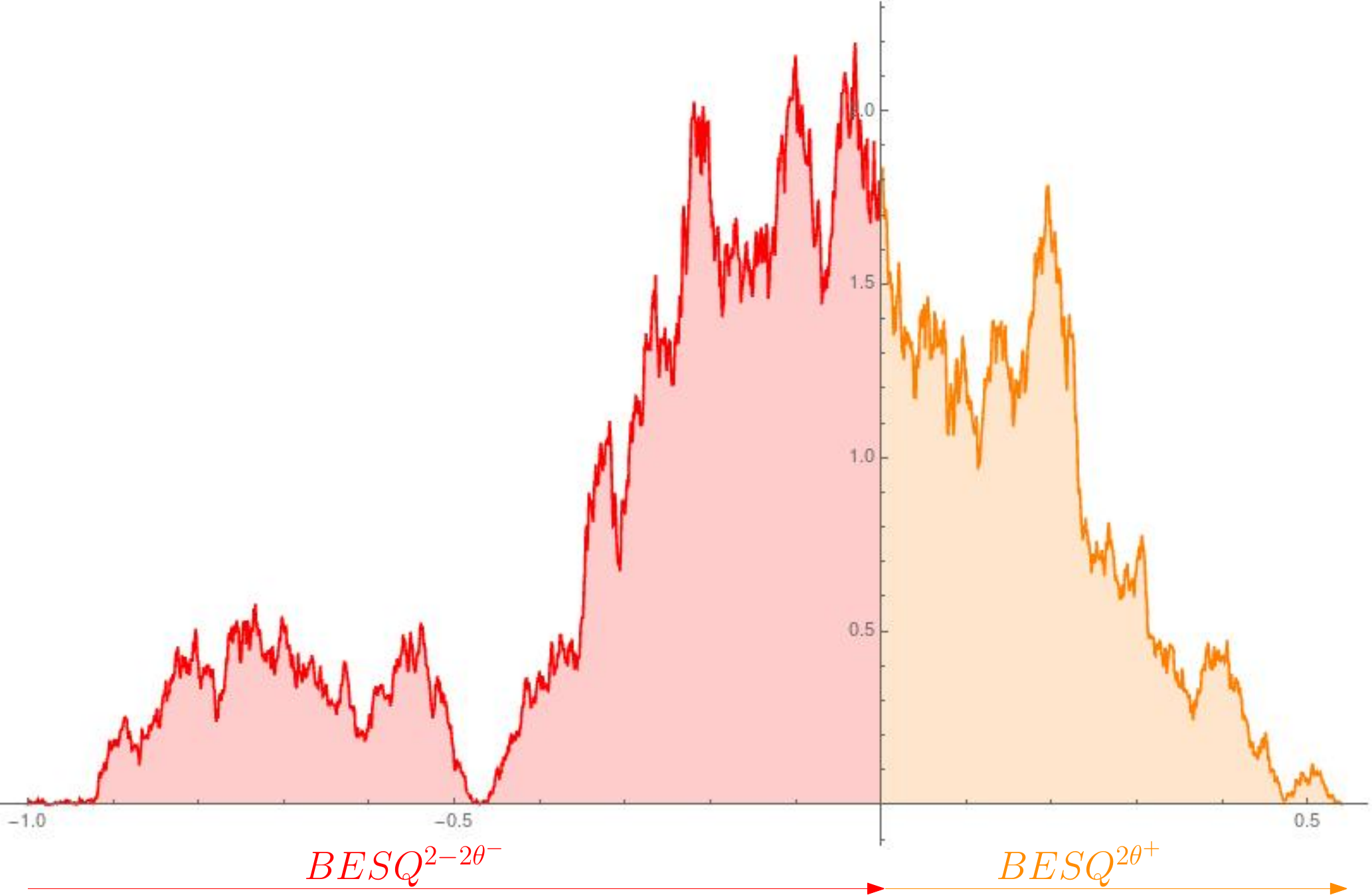}
 \includegraphics[width=0.45\textwidth,page=2]{RayKnightERWgraphs}
 \caption{On the left is a graphical representation of the Ray-Knight
   theorem for a standard $(\theta^+,\theta^-)$-BMPE stopped when the
   process first reaches $-1$.  On the right is a graphical
   representation of the Ray-Knight theorem for a
   $(\theta^+,\theta^-)$-BMPE started with initial condition
   $(\wu,0,\wo) = (-0.5,0,0.5)$ and stopped when the process first
   reaches
   $-1$. % Here, and throughout the paper, we will use BESQ$^\dim$ to denote a squared Bessel process of generalized dimension $\dim \in \R$.
 }
  \label{fig:RayKnight}
\end{figure}

One of the key results  of the present
work is a set of generalized Ray-Knight theorems for the ERW
on a ``mesoscopic'' scale.
More precisely, we first fix an  $\epsilon\in(0,1)$ and define stopping times $\{T_k^{\epsilon,n}\}_{k\geq 0}$ for the ERW by 
\begin{equation}\label{tenk}
  T^{\epsilon,n}_0=0,\ \tenk=T^{\epsilon,n}_{k,+}\wedge
  T^{\epsilon,n}_{k,-},\ \mbox{ where } T^{\epsilon,n}_{k,\pm}=\inf\{j>T^{\epsilon,n}_{k-1}:\,X_j-X_{T^{\epsilon,n}_{k-1}}=\pm\fl{\epsilon\sqrt{n}}\},\
  k\in\N.
\end{equation}
(We refer to $\sqrt{n}$ as the macroscopic scale for the ERW and $\epsilon \sqrt{n}$ as the mesoscopic scale since we will later take $\epsilon \to 0$.)
First of all, we show that the local time profile of the ERW when it first reaches $\fl{-\epsilon\sqrt{n}}$, converges when scaled by $\fl{\epsilon\sqrt{n}}$ to a concatenation of Bessel squared processes of generalized dimension $2(1-\theta^-)$ and $2\theta^+$ just as in the Ray-Knight Theorems for BMPE.
This then allows us to couple the first step of the induced mesoscopic walk $X_{T_1^{\epsilon,n}}$ with the first macroscopic step of a BMPE.  

Yet the most challenging and technical part of the paper
is in an extension of  this
coupling via a Ray-Knight approach to subsequent steps of the induced
mesoscopic walk.  In order to do this, we need rather strong control
on the distribution of the remaining cookie environment at the
stopping times $T_k^{\epsilon,n}$. That is, while initially the
distribution of first cookies was independent with marginal
$\eta$ at each site, after the walk has run for
a long time the distribution of the next cookie to be used at sites
within the range of the walk is no longer $\eta$ and no longer
necessarily independent for different sites.  However, we are able to
approximate the
distribution of next cookies in a convenient
way. There are two distributions $\pi^+$ and $\pi^-$, which we can
explicitly identify (see Section~\ref{sec:parameter} and \cite[Lemma 3.2 and (37)]{kpERWMCS}), such that the next cookie distribution is
approximately i.i.d.\ $\pi^-$ between the running minimum and the
current location, approximately i.i.d.\ $\pi^+$ between the current
location and the running maximum, and i.i.d.\ $\eta$ outside of the
range of the walk.
%While our proof doesn't directly rely on making this approximation of the next cookie environment precise, understanding this approximation lies behind many of the main ideas of our proof. 
Moreover, recalling that the parameters $\theta^+= \theta^+(\eta)$ and
$\theta^-= \theta^-(\eta)$ depend on the initial distribution $\eta$
of first cookies and since it follows from \cite[Corollary 3.5 and
equation (38)]{kpERWMCS} that $\theta^+(\pi^+) = 0$ and
$\theta^-(\pi^-) = 0$, from this we are able to show that the local
time process of the ERW after time $T_{k-1}^{\epsilon,n}$ and up
until time $T_{k,-}^{\epsilon,n}$ can be approximated by a
concatenation of Bessel squared processes of dimensions
$2(1-\theta^-)$, $2$, $0$, and $2\theta^+$ just as in the case of
the BMPE shown on the
  right in Figure~\ref{fig:RayKnight}.  A similar result can be
obtained for the local time process  between times $T_{k-1}^{\epsilon,n}$ and $T_{k,+}^{\epsilon,n}$.

%Our goal is to couple the two processes so that for any $T<\infty$ with probability at least $1-C\epsilon$ we have 
%\[
% \frac{X_{T^{\epsilon,n}_k}}{\fl{\epsilon\sqrt{n}}} = W(\tau_k), \quad \forall k \leq \epsilon^{-2}T. 
%\]
%In fact, the coupling we produce will not only give this but also give control on the running minimums/maximums of the processes at these stopping times as well. 

\begin{rem}
  One can, in fact, check using the definitions of the distributions
  $\pi^+$ and $\pi^-$ that the initial distribution of cookies that
  are independent and distributed according to $\pi^-$ on
  $(-\infty,-1] \cap \Z$, $\pi^+$ on $[1,\infty) \cap \Z$, and
  $\frac{1}{2} \pi^- + \frac{1}{2} \pi^+ =  \mu$ at $0$ is
  stationary for the
  cookie environment seen from the walker. That is, if this is the
  distribution of the initial first cookies then at any later time the
  remaining next cookies, shifted so that the current location of the
  random walk is taken to the origin, has the same distribution. We
  did not use this fact in our proof, nor are we able to even see how
  it could be used to prove convergence of the ERW to a BMPE. However,
  it may be possible use Kipnis-Varadhan techniques to prove that the
  path of an ERW with this stationary initial configuration of cookies
  converges in distribution to a Brownian motion. Again, since
  $\theta^+(\pi^+) = 0$ and
  $\theta^-(\pi^-) = 0$ this is consistent with
  what would be expected from our main results since a $(0,0)$-BMPE is
  just a standard Brownian motion.
\end{rem}

\begin{rem}
  It is interesting to note, and somewhat surprising, that while there
  is an asymmetry to the cookie environment in the interior of the
  range of the walk (approximately distribution $\pi^-$ to the left
  and $\pi^+$ to the right) this asymmetry is not seen in the
    scaling limit which behaves
%  that in the scaling limit the walk behaves
    like a Brownian motion in the interior of its range. We note,
    however, that the steps of the walk in the interior of the range
    are still highly correlated and this is reflected in the presence
    of the scaling parameter $a$ in the statement of Theorem
    \ref{main} which in general is not equal to 1.
\end{rem}

\subsection{Outline of the paper.}
The paper is organized as follows.  In Section~\ref{sec:BLP} we define
branching-like processes (BLPs) and recall from \cite{kpERWMCS} their
fundamental properties.  These processes are essential to describe the
behavior of local times of ERWs and to apply a Ray-Knight approach.
Section~\ref{sec:BMPE} discusses some basic properties of BMPEs,
including the Ray-Knight Theorems and couplings of BMPEs started
from slightly  different initial conditions.
In Section~\ref{sec:disc} we construct various discretizations of
BMPEs which will be used in Section~\ref{sec:coup} for coupling with
our ERW.

  Section~\ref{sec:tb}, for the most part, discusses diffusion approximations
  for the local times which are needed to relate exiting probabilities
  of ERWs to those of BMPEs.  It establishes ``classical'' results
  where the Markov chains that generate the cookie stacks initially
  have product distribution. This is then extended to the case when
  the initial values of the Markov chains are regular in a scale that
  is small compared to the macroscopic scale.  Section~\ref{sec:ei}
  concerns the regularity of cookies environments in two ways.
  Firstly, we prove that throughout time scale of order $n$ the
  next states of the cookie Markov chains are to scale $n^{1/4}$
  distributed like (in a crude averaging sense) $\pi^+$-product measure
  between the current position of the ERW and the current
  maximum and like $\pi^-$-product measure  between the current minimum
  and the current position of the ERW.  Secondly, we show that at each time $\tenk$ % after
  % each time the ERW moved by $\fl{\epsilon \sqrt n}$
  the distribution of the next states of the cookie Markov chains
  around points
  $(\fl{\epsilon \sqrt n}\Z)\setminus \{X_{\tenk}\}$ is very
  close to appropriate product measures in a total variation sense.
  These two results permit us to argue that the past does not play too
  big a role in the future at macroscopic level.

  Thereafter the paper works to implement the argument that
  $\{n^{-1/2}X_{\tenk}\}_{k\ge 1}$ evolves like a
    discretized BMPE and that the times $\tenk$ are
    well-controlled.  More precisely, Section~\ref{sec:coup}, drawing
    on diffusion approximations and the ``environmental'' results of
    Section~\ref{sec:ei}, constructs a coupling of our ERW and
    discretized BMPE% shows that the motion is indeed as claimed
    , while the final Section~\ref{sec:fin} establishes a law of large
    numbers for $\{\tenk\}_{k\ge 1}$, which enables us to
    pass from the discretized process to the general renormalized
    process and complete the proof of Theorem~\ref{main}.

  The proofs of many results that are of a technical nature and that
  are easy to believe are placed in an appendix, since the reader may
  wish to omit them on a first pass.

  % From this point on we will no longer refer to {\it cookie
  % environments} but rather to {\it token } environments which refer
  % to the values of $\{R^x_j\}_{x \in \Z, j \ge 1}$.  In particular
  % we at time $k$, the first token environment for the ERW will refer
  % to the variables $\{R^x_{\cal{L} (x,k) +1}\}_{x \in \Z}$.  Of
  % course the distinction between a cookie environment and a token
  % environment is moot if the function
  % $p: \{1,2 \ldots N \} \rightarrow (0,1)$ is injective but while
  % the first token environment shifted by the current location of the
  % ERW, is always a Markov chain, in general the first cookie
  % environment need not be.  Given the first token at site $x$ at
  % time $k$, the evolution of $\{R^x_{\cal{L} (x,n) +1}\}_{n \ge k}$
  % is independent
  % of the past history of the ERW, which is not in general so for the
  % corresponding cookie process.

\subsection{Notation}
For the convenience of the reader, we collect here some notation that will be used throughout the  paper.
\begin{enumerate}[-]
\item We  write $x_+$ for  $x\vee 0$ and $\Z_+$ for $\N\cup\{0\}$.
\item For any $a<b$ we will let $\bint{a,b} = [a,b] \cap \Z$. 
Similarly, we will use $\llbracket a,\infty)$  for $[a,\infty) \cap \Z$. 
% Similarly, we will use $\llbracket a,\infty)$ and $(-\infty,b\rrbracket$ for $[a,\infty) \cap \Z$ and $(-\infty,b]\cap \Z$, respectively.

\item We write $\|\mu_1-\mu_2\|_{TV}$ for the total variation
    distance between two measures $\mu_1$ and $\mu_2$. For two random
  variables $V$ and $U$,
  $d_{TV}(V,U)$ will denote the total variation distance between their
  distributions.

\item   We denote by $P_\gamma$ the averaged probability measure when
  the first cookies are i.i.d.\ with marginal distribution
  $\gamma$. We shall typically drop the subscript and write $P$
  instead of $P_\eta$ if $\gamma=\eta$, the original initial
  distribution of the first cookies.

\item The local time of the ERW at $x$ by time $n$
is given by
\[{\cal L}(0,x)=0,\quad {\cal L}(n,x)=\sum_{j=0}^{n-1}\ind{X_j=x},\ \
  n\in\N,\ \ x\in\Z.\]

\item For a stochastic process $Z=(Z_n)_{n\ge 0}$ and $a\in\mathbb{R}$ we
define the hitting times
\[\tau^Z_a=\inf\{n\ge 0:\ Z_n\ge a\},\quad \sigma^Z_a=\inf\{n\ge 0:\
  Z_n\le a\},\quad \sigma_{m,a}^Z = \inf\{ n\geq m: \, Z_n \leq a\},\]
with $\inf \emptyset=\infty$.  For instance
$\sigma_{-m}^X \wedge \tau_m^X$ will denote the exit time of the
excited random walk from the interval $(-m,m)$.  A similar definition
will apply to hitting times of processes in continuous time. We shall
occasionally drop the superscript whenever there is no ambiguity about
which process we are talking about.
\item With mild abuse of terminology we shall refer to $\{R^x_1\}_{x\in\Z}$
as ``the first cookies''. The expression ``the first cookies at time
$\tau$'' will refer to the collection
$\{R^x_{L(x,\tau)+1}\}_{x\in \Z}$ for a stopping time $\tau$ and will
denote the next states of the cookie Markov chains at time $\tau$.
\end{enumerate}

\section{The branching-like processes (BLPs)}\label{sec:BLP}

In this section we introduce four Markov chains $U^+,U^-,V^+,$ and
$V^-$ taking values on $\Z_+$ which are useful in analyzing excited
random walks. We will refer to these Markov chains as the
\emph{branching-like processes} (BLPs) due to a similarity in
structure to Galton-Watson branching process (or branching processes
with migration).  We will first describe the transition probabilities
of the four BLPs and then give a brief description of their relation
to the directed edge local times of excited random walks.

From this point on we will shift the meaning of {\it cookie} and {\it
  cookie environments}.  Henceforth, the cookie at site $x$ at time $k$
will refer to $R^x_{\cal{L} (x,k) +1}$.  In particular, given time
$k$, the first cookie environment for the ERW will refer to the
variables $\{R^x_{\cal{L} (x,k) +1}\}_{x \in \Z}$.  Of course, the
distinction between the former usage of {\it cookie} and the present
and future usage is moot if $p: \{1,2 \ldots N \} \rightarrow (0,1)$
is injective. We note that while the (present sense) cookie
environment shifted by the current location of the ERW, is always a
Markov chain, in general the former cookie environment need not have
this property.  Given the first cookie at site $x$ at time $k$, the
evolution of $\{R^x_{\cal{L} (x,n) +1}\}_{n \ge k}$ is independent of
the past history of $X$ (again unlike the cookie evolution in the
previous sense).

We will describe the distribution of the four BLPs given the
distribution of the first cookies $\{ R^x_1\}_{x \in \bint{\ell,r}}$
on an interval $\bint{\ell,r} \subset \Z$.  The distribution of the
first cookies can either be deterministic or random with
  independence over the sites (e.g., $\{R^x_1\}_{x\in \Z}$ can be
i.i.d.\ with distribution $\eta$).  Given the distribution of the
first cookies on $\bint{\ell,r}$, we can construct the BLP $U^+$ as
follows.  First, we generate the remainder of the environment
$\omega_x(j)=p(R^x_j), j\geq 1$, at each site $x\in \bint{\ell,r}$ by
letting $\{R^x_j\}_{j\geq 1}$ be a realization of the Markov chain in
Assumption \ref{asm:Markov} but with $R^x_1$ having the prescribed
initial distribution. The realizations of the Markov chains at
different sites are independent.  Next, given the entire cookie
environment on $\bint{\ell,r}$, we let
$\{\xi^x_j \}_{x\in \bint{\ell,r}, \, j\geq 1}$ be a family of
independent Bernoulli random variables with
$\xi^x_j \sim \text{Ber}(\omega_x(j))$.  Finally, we let the BLP $U^+$
started with initial value $U^+_0 = m \in\Z_+$ be defined as follows.
 \[U^+_0=m,\quad
 U^+_{i} = \inf \left\{ k\geq 0: \, \sum_{j=1}^{k+U^+_{i-1}} (1-\xi^{\ell+i}_j) = U^+_{i-1} \right\}\ \ \text{for }i\in\{1,2,\dots,r-\ell\}. 
 \]
That is, $U_{i}^+$ is the number of ``successes''  before the $U_{i-1}^+$-th ``failure''  in the sequence of Bernoulli trials $\{\xi^{\ell+i}_j\}_{j\geq 1}$. 
The BLP $V^+$ is defined similarly, but instead we have 
\[
 V^+_0=m,\quad V^+_{i} = \inf \left\{ k\geq 0: \, \sum_{j=1}^{k+V^+_{i-1}+1} (1-\xi^{\ell+i}_j) = V^+_{i-1} +1 \right\}\ \ \text{for }i\in\{1,2,\dots,r-\ell\}, 
\]
so that $V_{i}^+$ is the number of successes before the $(V_{i-1}^+ + 1)$-th failure in the sequence $\{\xi^{\ell+i}_j\}_{j\geq 1}$.
The BLPs $U^-$ and $V^-$  are constructed similarly but reversing the role of ``successes'' and ``failures'' and using the cookie stacks from right to left instead. That is, given the initial values of $U^-_0$ or $V^-_0$ we let 
\[
 U^-_{i} = \inf \left\{ k\geq 0: \, \sum_{j=1}^{k+U^+_{i-1}} \xi^{r-i}_j = U^-_{i-1} \right\}\ \ \text{for }i\in\{1,2,\dots,r-\ell\},
\]
and
\[
 V^-_{i} = \inf \left\{ k\geq 0: \, \sum_{j=1}^{k+V^+_{i-1}+1} \xi^{r-i}_j = V^-_{i-1} + 1 \right\}\ \ \text{for }i\in\{1,2,\dots,r-\ell\}. 
\]

Before giving the connection of the BLPs with the excited random walk,
we first mention some properties of the BLPs that we will use
throughout the paper.
\begin{enumerate}
\item Because the Markov chains $\{R^x_j\}_{j\geq 1}$ are independent
  at different sites, it follows that all four of the BLPs $U^\pm$ and
  $V^\pm$ are Markov chains.  For general first cookie conditions the
  BLPs are time inhomogeneous Markov chains with the transition
  probabilities at different times depending on the distribution of
  the first cookies at different sites, but if the initial
  distribution $\{R_1^x\}_x$ is i.i.d.\ then the BLPs are time
  homogeneous Markov chain. 
 \item The processes $U^+$ and $U^-$ have $0$ as an absorbing state. In contrast, $V^+$ and $V^-$ are irreducible Markov chains on $\Z_+$. 
 \item The BLPs all have a natural monotonicity property with respect to the initial condition. If $Z$ and $Z'$ are two instances of the same BLP started from the same first cookie environments but with different initial conditions $Z_0 = k < k' = Z_0'$, then our construction above provides a coupling  such that $Z_i \leq Z_i'$ for all $i$ (as long as  both processes use the same Bernoulli random variables $\{\xi_x^j\}_{x,j}$). 
 \end{enumerate}

 We now explain the connection of the BLPs to the study of excited
 random walks.  For any $\ell \geq 1$ and $x\geq -\ell$ let
\begin{equation}\label{downsteps}
 \mathcal{E}^{(-\ell)}_x = \sum_{i=0}^{\sigma^X_{-\ell}-1} \ind{X_i = x, \, X_{i+1} = x+1}
\end{equation}
be the number of steps the ERW takes from $x$ to $x+1$ prior to the first visit to $-\ell$. Then, it can be seen that the sequence $\{ \mathcal{E}_x^{(-\ell)} \}_{x\geq -\ell}$ has the same distribution as a  concatenation of a $V^+$ and $U^+$ process.\footnote{Implicitly we are using here that under the assumptions of this paper the walk is recurrent. Thus, for $\P$-a.e.\ cookie environment $\omega$ we have that $P_\omega( \sigma_{-\ell}^X < \infty) = 1$.} More precisely, 
\begin{itemize}
 \item $(\mathcal{E}^{(-\ell)}_{-\ell},  \mathcal{E}^{(-\ell)}_{-\ell+1}, \cdots, \mathcal{E}^{(-\ell)}_{-1}, \mathcal{E}^{(-\ell)}_{0} )$ has the same distribution as $(V^+_0,V^+_1,\ldots,V^+_\ell)$ started with $V_0^+ = 0$ and using the cookie environment on the interval $\bint{-\ell,0}$. 
 \item Given $\mathcal{E}^{(-\ell)}_{0} = m$ the sequence $(\mathcal{E}^{(-\ell)}_0,\mathcal{E}^{(-\ell)}_1,\mathcal{E}^{(-\ell)}_2,\ldots )$ has the same distribution as \linebreak $(U^+_0,U^+_1,U^+_2,\ldots)$ started with $U^+_0 = m$ and using the cookie environment on the interval $\llbracket 0,\infty)$. 
\end{itemize}
(See \cite{kzPNERW} or \cite{kpERWMCS} for more details.)  Let
$\{Z_i\}_{i\ge 0}$ denote the concatenation of the above $V^+$ and
$U^+$ processes.  This connection of the ERW with the BLPs allows us
to restate an exit distribution problem for the ERW as a question
about the process $Z$. Indeed, one sees that the random walk exits the
interval $(-\ell,\ell)$ at $-\ell$ if and only if the process
$\{ \mathcal{E}_x^{(-\ell)}\}_{x\geq \ell}$ dies out before $x=\ell$.
Therefore, letting
$\sigma^Z_{\ell,0} = \inf\{i \geq \ell:\, Z_i = 0 \}$ we have that
\[
 P( \sigma_{-\ell}^X < \tau_\ell^X ) = P(\sigma_{\ell,0}^Z < 2\ell ). 
\]
Moreover, since $\sigma_{-\ell}^X = \ell + 2 \sum_{x\geq -\ell} \mathcal{E}^{(-\ell)}_x$,  it follows that conditioned on the event $\{ \sigma_{-\ell}^X < \tau_\ell^X \}$ the exit time $\sigma_{-\ell}^X \wedge \tau_\ell^X$ for the ERW has the same distribution as $\ell + 2 \sum_{i=0}^{\sigma_{\ell,0}^Z-1} Z_i $ conditioned on the event $\{ \sigma_{\ell,0}^Z < 2\ell \}$. 

In this paper we will often be interested in similar exit distribution
and exit time problems for the ERW but conditioned on some knowledge
of the walk up to a certain time.  For instance, suppose the random
walk has already evolved for some amount of time $T$ (either a
deterministic time or a stopping time for the walk) and that we know
by this time the maximum and minimum are $S= \max_{k\leq T} X_k$ and
$I = \min_{k\leq T} X_k$, the current position of the walk is
$X_T = z \in \bint{I,S}$, and we also know the values of the next
cookies to be used at all sites $x \in \bint{I,S}$ that have been
visited thus far. Given all this information, we wish to know after
time $T$ whether the walk will reach $z-\ell$ or $z+\ell$ first. This
can be translated to a problem about concatenated BLPs as follows.
Let $\{Z_k\}_{k=0}^{2\ell}$ be a concatenation of BLPs such that
\begin{itemize}
 \item $(Z_0,Z_1,\ldots,Z_\ell)$ is a $V^+$ process started from $V^+_0 = 0$ and using the remaining first cookie environment on $\bint{z-\ell,z}$.
 \item Given $Z_\ell=m$, the process $(Z_\ell,Z_{\ell+1},\ldots,Z_{2\ell})$ is a $U^+$ process started from $U^+_0 = m$ and using the remaining first cookie environment on $\bint{z,z+\ell}$. 
\end{itemize}
As above, questions about the exit distribution and exit time of the walk after time $T$ and until hitting $z-\ell$ or $z+\ell$ can be translated to questions about the concatenated BLP $Z$. 
%The probability that the random walk reaches $z-m$ before $z+m$ after time $T$ given the remaining first cookie environment at the sites visited up to time $T$ is then equal to the probability that $\sigma^Z_{m,0} < 2m$, and conditioned on the walk reaching $z-m$ before $z+m$ the time to exit $(z-m,z+m)$ has the same distribution as $m + 2 \sum_{i=0}^{\sigma_{m,0}^Z-1} Z_i $ conditioned on the event $\{ \sigma_{m,0}^Z < 2m \}$. 

This illustrates how the BLPs $U^+$ and $V^+$ arise in connection with
the study of ERW. The BLPs $U^-$ and $V^-$ arise in a somewhat similar
manner. For instance, let
$\mathcal{D}^{(\ell)}_x = \sum_{i=0}^{\tau^X_{\ell}-1} \ind{X_i = x,
  \, X_{i+1} = x-1}$ be the number of steps from $x$ to $x-1$ prior to
the walk first reaching $\ell$. Then, one can see that the process
$(\mathcal{D}^\ell_\ell,
\mathcal{D}^\ell_{\ell-1},\cdots,\mathcal{D}^\ell_0)$ has the same
distribution as the concatenation of a $V^-$ process and a $U^-$
process.  One can use this concatenated BLP process to study the
probability the ERW exits an interval to the right and the
distribution of the time it takes the walk to exit an interval on the
event that it exits to the right.

Because arguments involving the BLPs $U^-$ and $V^-$ are symmetric to
those involving $U^+$ and $V^+$, we will give all proofs only for the
processes $U^+$ and $V^+$.

\subsection{Parameters associated to the BLPs}\label{sec:parameter}

The parameters $\theta^+$ and $\theta^-$ which appear in the statement
of the main results are defined in terms of the BLPs. We close this
section by giving the description of these parameters along with
several other related parameters that will be used throughout the
paper.  While explicit formulas for all  parameters
discussed below can be found in \cite[equation (37)]{kpERWMCS}, we
restrict our attention here to the probabilistic definition of these
parameters in terms of the BLPs.

Let $\mathbf{r}^+ = (r^+(i))_{1\leq i\leq N}$ and $\mathbf{r}^- = (r^-(i))_{1\leq i\leq N}$ be the vectors with entries  
\begin{equation}\label{rpmdef}
 r^+(i) = \lim_{n\to \infty} E[ U^+_1 - n \mid U_0^+ = n, \, R^1_1 = i ], 
\quad \text{and}\quad 
 r^-(i) = \lim_{n\to \infty} E[ U^-_1 - n \mid U_0^- = n, \, R^{-1}_1 = i ].
\end{equation}
That is, $r^\pm(i)$ gives the limit of the expected ``drift'' of the first step of the process $U^{\pm}$ when the first cookie to be used is of type $i$ and the  BLP is started from a very large initial value $U^\pm_0 = n$. 
Next let 
\begin{equation}\label{nudef}
 \nu = \lim_{n\to\infty} \frac{\Var(U^+_1 \mid U^+_0 = n, \, R^1_1 = i)}{n} = \lim_{n\to\infty} \frac{\Var(U^-_1 \mid U^-_0 = n, \, R^{-1}_1 = i)}{n}. 
\end{equation}
The proof that the limits in \eqref{rpmdef} and \ref{nudef} exist and that the limits in \eqref{nudef} are equal and do not depend on the distribution of the first cookie can be found in \cite{kpERWMCS}. Moreover, it was shown in \cite[Proposition 4.3 and Lemma 4.4]{kpERWMCS} that these parameters have the following relation: 
\begin{equation}\label{vecr}
 r^+(i) + r^-(i) =\frac{\nu}{2} -1 , \quad \forall i\in \{1,2,\ldots,N \}. 
\end{equation}
Finally, the parameters $\theta^+$ and $\theta^-$ are defined by 
\begin{equation}
  \label{thedef}
  \theta^+ = \theta^+(\eta)=\frac{2 \eta \cdot \mathbf{r}^+}{\nu}
\quad \text{and}\quad 
\theta^- = \theta^-(\eta)= \frac{2 \eta \cdot \mathbf{r}^-}{\nu}. 
\end{equation}
Note that the equations \eqref{vecr} and \eqref{thedef} imply that
$\theta^+ + \theta^- = 1 - \frac{2}{\nu}$ and, thus,
$\theta^+ + \theta^- < 1$.  The relevance of the parameters $\theta^+$
and $\theta^-$ is that the BLPs have scaling limits which are Bessel
squared processes, and the parameters $\theta^\pm$ identify the
generalized ``dimension'' of these Bessel squared processes (this will
be detailed further in Sections \ref{daiid} and \ref{0.25}).

It should be noted that $\mathbf{r}^+$, $\mathbf{r^-}$ and $\nu$
depend only on the transition matrix $K$ and the function $p(\cdot)$
which appear in the description of markovian cookie stacks in
Assumption \ref{asm:Markov}. The parameters $\theta^+$ and $\theta^-$,
however, depend not only on $K$ and $p(\cdot)$ but also on the initial
distribution $\eta$ of the first cookies.  

  Finally, we will introduce two distributions $\pi^+$ and
  $\pi^-$  which depend only on $K$ and $p(\cdot)$ and play an
  important role in this paper.  Let
$\pi^\pm = (\pi^\pm(i))_{1\leq i\leq N}$ be defined by\footnote{The
  existence of these limits in the definition of $\pi^\pm$ and the
  fact that the limits do not depend on the first cookie
  distribution can be found in \cite[Section 3.1 and (37)]{kpERWMCS}.}
\begin{equation*}
 \pi^\pm(i) = \lim_{n\to\infty} P\left( R^{\pm1}_{U^\pm_1 + n + 1} = i \mid U_0^\pm = n, \, R^{\pm1}_1 = i' \right)% \\
% \text{and}\quad \pi^-(i) &= \lim_{n\to\infty} P\left( R^{-1}_{U^-_1 + n + 1} = i \mid U_0^- = n, \, R^{-1}_1 = i' \right)
.
\end{equation*}
In words, using the sequence of Bernoulli random variables
$\{\xi^{\pm 1}_j\}_{j\geq 1}$, the distributions $\pi^+$ and $\pi^-$
give the limiting distribution of the next value in the underlying
Markov chain $\{R^{\pm 1}_j\}_{j\geq 1}$ immediately following the
$n$-th ``failure'' or ``success,'' respectively, as $n\to \infty$.
%$\pi^+(i)$ is the probability that after the $n$-th ``failure'' in the sequence of Bernoulli trials the next Bernoulli random variable will have distribution Ber($p(i)$) and $\pi^-(i)$ is the probability that after the $n$-th ``success'' the next Bernoulli random variable will have distribution Ber($p(i)$). 

The relevance of the distributions $\pi^\pm$ is that they are good
approximations for the distribution of the remaining first cookie
environment at sites with a large local time (which is most sites in
the range). Indeed, for a site $x$ within the range of the walk but to
the right of the current location, since the last step from that site
was to the left (corresponding to a ``failure'' in a sequence
$\{\xi^x_j\}_{j\geq 1}$), the probability that the remaining first
cookie at $x$ is of ``type $i$'' can be approximated by $\pi^+(i)$.

Since we will at times be using the BLPs in cookie environments which
have first cookie distributions which are approximately $\pi^+$ or
$\pi^-$, it is important to note what the parameters $\theta^+$ and
$\theta^-$ are with these distributions on the first
cookies. It was shown in \cite[Corollary
  3.5]{kpERWMCS} that $\pi^+ \cdot \mathbf{r}^+ = 0$ and
  $\pi^- \cdot \mathbf{r}^- = 0$. Substituting these equations into
  \eqref{thedef} we get $\theta^+(\pi^+)= 0$ and
$\theta^-(\pi^-) = 0$.

\section{Brownian motion perturbed at extrema: preliminaries}\label{sec:BMPE}
Recall the notation \eqref{IS}.  Though in the introduction a
  BMPE had initial value $0$, henceforth the process triple
  $(I,W,S)\coloneqq \{(I(t),W(t),S(t))\}_{t\ge 0}$ is a Markov process
  which can be considered starting from any initial state
  $(\wu,w,\wo)$, $\wu\le w\le \wo$.  When BMPE starts from
  $(0,0,0)$ we shall call it {\em a standard BMPE}. We remark
    that a standard $(\alpha,\beta)$-BMPE \eqref{BMPEdef} inherits the
    scaling property of the Brownian motion: for every $c>0$ the
    process $\left\{c W\left(c^{-2}t\right)\right\}_{t\ge 0}$ is a
    standard BMPE with the same parameters.

For reference convenience we shall use $(\theta^+,\theta^-)$-BMPE in
place of $(\alpha,\beta)$-BMPE and assume throughout this section that
$\theta^+$ and $\theta^-$ are arbitrary real numbers strictly less
than 1. In later sections $\theta^+$ and $\theta^-$ will be fixed as
parameters of our ERW but in this section we only require that
$\theta^+,\theta^-<1$.

%Recall the notation \eqref{IS} and the fact that the process triple
%$(I,W,S)\coloneqq \{(I(t),W(t),S(t))\}_{t\ge 0}$ is a Markov process which
%can be considered starting from any initial state $(\wu,w,\wo)$,
%$\wu\le w\le \wo$.

\subsection{Exit probabilities.} We shall define for $a\in\R$, $\wu\le w\le \wo$,
\begin{equation}\label{hitp}
  \tau_a(\wu,w,\wo)=\inf\{t>0:\ W(t)=a\mid (
  I(0),W(0),S(0))=(\wu,w,\wo)\}
\end{equation}
and for $\epsilon>0$
\begin{equation}
  \label{exit}
  \tau(\epsilon,\wu,w,\wo)\coloneqq \tau_{w-\epsilon}(\wu,w,\wo)\wedge\tau_{w+\epsilon}(\wu,w,\wo).
\end{equation}
These are respectively the first time $W$ hits $a$ and the first time
$W$ exits $(w-\epsilon,w+\epsilon)$ given that it started at
$(\wu,w,\wo)$. We also drop the arguments $\wu,w,\wo$ whenever
$(\wu,w,\wo)=(0,0,0)$.  Clearly, when
$\wu+\epsilon\le w\le \wo-\epsilon$,
$P(W(\tau(\epsilon,\wu,w,\wo))=w\pm\epsilon)=1/2$. In many other
cases, these probabilities can also be computed explicitly. The
following lemma can be found in \cite[Proposition 4(iii)]{pwPBM} (see
also Proposition 3 in \cite{pwPBM}).
\begin{lemma} \label{mel1}
  Let $W$ be a standard $(\theta^+,\theta^-)$-BMPE. Then for $a<0<b$
  \[P(\tau_a<\tau_b)=\frac{1}{B(1-\theta^+,1-\theta^-)}\int_0^{\frac{b}{b-a}}t^{-\theta^-}(1-t)^{-\theta^+}\,dt,\]
  where $B(\cdot,\cdot)$ is the beta function.
\end{lemma}

\begin{cor} \label{roc1} Let $W$ be an $(\theta^+,\theta^-)$-BMPE starting
  from $(\wu,0,\wo)$. \\If $\wu\le a< 0\le \wo\le b$,
  then
  \[P(\tau_b(\wu,0,\wo)<\tau_a(\wu,0,\wo))=\frac{-a}{b-a}\left(\frac{b-a}{\wo-a}\right)^{\theta^+}.\]
  If $a\le \wu\le 0< b\le \wo$
  then
  \[P(\tau_a(\wu,0,\wo)<\tau_b(\wu,0,\wo))=\frac{b}{b-a}\left(\frac{b-a}{b-\wu}\right)^{\theta^-}.\] % If $-a\le m\le 0\le \wo\le b$, then
  % \begin{multline*}
  %   P(\tau_{-a}(m,0,M)<\tau_b(m,0,M))\\=\frac{(M-a)^{1-\theta^-}((b-m)^{1-\theta^+}+m(M-m)^{-\theta^+})}{(M-a)^{1-\theta^-}(b-m)^{1-\theta^+}+(M-m)^{1-\theta^+}(M+a)^{1-\theta^-}-(M-a)^{1-\theta^-}(M-m)^{1-\theta^+}}.
%  \end{multline*}
\end{cor}
\begin{proof}
  We shall prove the second statement, the first one is obtained in a
  symmetric way.

  Observe that to reach $a$ before $b$ the process $W$ has to reach
  $\wu$ before $b$ and that between $\wu$ and $b$ the process $W$ behaves
  simply as a standard Brownian motion. Using this observation and the
  Markov property of the triple $(I,W,S)$ we get
  \[P(\tau_a(\wu,0,\wo)<\tau_b(\wu,0,\wo))=\frac{b}{b-\wu}\,P(\tau_a(\wu,\wu,\wo)<\tau_b(\wu,\wu,\wo)).\]
  Next, note that the last probability is equal to the probability that
  a standard $(0,\theta^-)$-BMPE reaches $a-\wu$ before $b-\wu$. Applying
  Lemma~\ref{mel1} we obtain the desired result.
\end{proof}

\subsection{Ray-Knight theorems for BMPE}

As we described in the introduction, to approximate exit probabilities
and the exit time of our ERW from an interval we shall use an approach
based on edge local time BLPs. In this subsection we discuss the
continuous counterpart of these results in more detail.

As noted in the Introduction (see, for example, \cite{cpyBetaPBM}) the
local times of $(\theta^+,\theta^-)$-BMPE satisfy analogs of the first
and second Ray-Knight theorems. These theorems involve squared Bessel
processes of generalized dimensions $\dim\in\R$ which we shall denote
by BESQ$^\dim$. BESQ$^\dim$ process starting at $y\ge 0$ is a unique
strong solution of the SDE (see, for example, \cite[Chapter
XI]{ryCMBM} for $\dim\ge 0$ and \cite[Section 3]{gySGBP} for $\dim<0$)
\begin{equation}
  \label{besq0}
  y(t)=y+\dim t+2\int_0^t\sqrt{|y(s)|}\,dB(s).
\end{equation}
When $\dim\ge 0$ and $y(0)\ge 0$, the solution $y(s)$ of the above equation
is always non-negative, and the absolute value in \eqref{besq0} can
be simply dropped. We recall that when $\dim\ge 2$ the process
$\{y(t)\}_{t\ge 0}$ with $y(0)\ge 0$ is strictly positive for all
$t>0$ with probability 1. When $\dim <2$ then with probability 1 the
process hits zero in finite time. Up to this time, $\tau^y_0$, we
also can drop the absolute value even if $\dim<0$.

In this paper we will start with $y(0)\ge 0$ and stop the process with
$\dim\le 0$ at time $\tau^y_0$. This means that we are always in
the setting when $|y(s)|=y(s)$ in \eqref{besq0}. However, for
convenience we often use $(y(s))_+$ instead. With this choice, when
$\dim\le 0$, after time $\tau^y_0$ the process continues
degenerately as $y(s+\tau^y_0)=\dim s\le 0$ for all $s\ge 0$. We
continue to refer to solutions of
\[y(t)=y+\dim t +2\int_0^t\sqrt{(y(s))_+}\,dB(s),\quad y(0)=y\ge 0,\]
for any $\dim\in\R$ as BESQ$^\dim$. This definition coincides with
\eqref{besq0} for all $\dim\ge 0$ and $y\ge 0$ and for all $\dim<0$
and $y\ge 0$ up to $\tau^y_0$. 

Denote by $(\ell^W_{x,t})_{x\in\R,t\ge 0}$ the jointly continuous
family of local times of $(\theta^+,\theta^-)$-BMPE $W$. The starting
triple for $W$ will not be reflected in the notation. It will be given
explicitly in each case.

The following proposition states Ray-Knight theorems for BMPE in the
most convenient form for our purposes.
\begin{prop}[\cite{cdhULL}, Proposition 2.1]\label{RNbmpe}
  Let $W$ be a $(\theta^+,\theta^-)$-BMPE starting from $(\wu,w,\wo)$,
  $0\le \wu\le w\le \wo$, and $\tau^W_0=\inf\{t\ge 0: W(t)=0\}$. Then
  the local time process $\{\ell^W_{x,\tau^W_0}\}_{x\ge 0}$ has the same
  law as $\{y(x\wedge \sigma^y_{w,0})\}_{t\ge 0}$ where $\{y(x)\}_{x\ge 0}$ is
  the unique strong solution of the
  equation
  \[y(x)=2\int_0^x\sqrt{(y(s))_+}\,dB(s)+\int_0^x(2(1-\theta^-)\ind{0\le s\le
      \wu}+2\ind{\wu\le s\le w}+2\theta^+\ind{s\ge \wo})\,ds.\] 
%and $\tau^y_{w,0}=\inf\{x> w:\,y(x)=0\}$. 
In words, the process
  $\{ \ell^W_{x,\tau^W_0}\}_{x\ge 0}$ is an inhomogeneous Markov process
  which is a BESQ$^{2(1-\theta^-)}$ on $[0,\wu]$, a BESQ$^{2}$ on
  $[\wu,w]$, a BESQ$^{0}$ on $(w,\wo]$ and a BESQ$^{2\theta^+}$ on
  $[\wo,\infty)$, absorbed at its first zero after $w$.
\end{prop}

This proposition immediately implies the following statement. Its
discrete version, Lemma~\ref{hitpr} (``concatenation lemma''), 
is one of the main tools of this paper.
\begin{cor}\label{exitc}
  Let $W$ be a $(\theta^+,\theta^-)$-BMPE starting from $(\wu,w,\wo)$,
  $0\le \wu\le w\le \wo$ and $\{y(x)\}_{x\ge 0}$ be the process defined
  in Proposition~\ref{RNbmpe}. Then
  \begin{enumerate}[(i)]
  \item for any $b> w$,
    $P(\tau_0(\wu,w,\wo)<\tau_b(\wu,w,\wo))=P(\sigma^y_{w,0}<b)$;
  \item for any $b>\wo$ and any interval $J\subseteq (\wo,b)$
    \[P(\tau_0(\wu,w,\wo)<\tau_b(\wu,w,\wo),\ S(\tau_0(\wu,w,\wo))\in
      J)=P(\sigma^y_{w,0}\in J).\]
  \end{enumerate}
\end{cor}
\begin{rem}
  Note that (ii) and (i) imply that \[P(\tau_0(\wu,w,\wo)<\tau_b(\wu,w,\wo),\ S(\tau_0(\wu,w,\wo))=\wo)=P(\sigma^y_{w,0}\le \wo).\]
\end{rem}
\begin{rem}
  Part (i) of this corollary and known facts about the distribution of BESQ processes can be used to derive Lemma~\ref{mel1}.
\end{rem}

\subsection{Coupling of BMPEs with different initial data.} We shall
need the following coupling result about Brownian motions perturbed
only at one extremum. For definiteness, we assume that the
perturbation is at the maximum, i.e.\ we shall consider BMPEs with
$\theta^-=0$. In this case, the process $\{(W(t),S(t))\}_{t\geq 0}$, is
markovian.
\begin{lemma}\label{BMPEcoup}
  Let $\{(W_i(t),S_i(t))\}_{t\geq 0}$, $i=1,2$, be Brownian motions
  perturbed at the maximum, i.e.\ BMPE with parameters $\theta^-=0$
  and $\theta^+\in(-\infty,1)$. Suppose that
  $(W_i(0),S_i(0))\in\{(w,\wo)\in[-1,1]^2:\,w\le \wo\},\ i=1,2$. There
  exists a coupling such that for some $\Cl{0}>0$
  uniformly over all initial conditions in
  $\{(w,\wo)\in[-1,1]^2:\,w\le \wo\}$
  \[P\left((W_1(1),S_1(1))=(W_2(1),S_2(1))\ \text{and }W_i(t)\in[-4,4]\
      \forall t\in[0,1],\ i=1,2\right)\ge \Cr{0}.\]
\end{lemma}
The proof of this lemma is given in the Appendix.

\begin{cor}\label{BMPEcoup1}
  Let $\{(W_i(t),S_i(t))\}_{t\geq 0}$, $i=1,2$, be as in Lemma
  \ref{BMPEcoup}. There exists a coupling and nontrivial constants
  $\Cl{one} $ and $K_0$ (which do not depend on the initial conditions
  $(W_i(0),S_i(0)) = (w_i,\wo _i)  \in [-1,1]^2$) such that for all
    $R\ge 4$ outside probability $\Cr{one}R^{- \frac{1}{K_0 }} $
  \[
  (W_1(t),S_1(t))=(W_2(t),S_2(t)) \quad \forall t> \rho_R,
  \]
  where $\rho_R=\inf \{t \ge 0 : \max\{ |W_1(t)|, \,  |W_2(t)| \} \ge R\}$.
  \end{cor}    

\begin{proof}
Let $\tcp = \inf\{t\geq 0: \,  (W_1(t),S_1(t))=(W_2(t),S_2(t))  \}$ be the coupling time of the BMPEs,
 and for any choice of $w_i \leq \wo_i$, $i=1,2$ and $R>0$ let
\[
 a_R(w_1,\wo_1,w_2,\wo_2) = P\left( \tcp > \rho_R \mid (W_i(0),S_i(0)) = (w_i,\wo_i), \, i=1,2 \right). 
\]
%Then since we can clearly couple the two BMPEs for all times after $\tcp$ we need to show that 
%\[
% a_R(w_1,\wo_1,w_2,\wo_2) \leq \Cr{one} R^{- \frac{1}{K_0 }}, \quad \forall R\geq 4, \text{ and } (w_i,\wo_i) \in [-1,1]^2, \, w_i \leq \wo_i, \, i=1,2. 
%\]
If for any $r>0$ we denote $\Delta_r = \{(w,\wo) \in [-r,r]^2: \, w \leq \wo \}$, then it follows easily from Lemma \ref{BMPEcoup} that 
\[
 \sup_{(w_1,\wo_1), (w_2,\wo_2) \in \Delta_1 } a_4(w_1,\wo_1,w_2,\wo_2) \leq 1-\Cr{0}. 
\]
%That is, if the initial conditions are chosen in $[-1,1]$ then the BMPEs have coupled with probability at least $1-\Cr{0}$ prior to either of them exiting $[-4,4]$.
Note that since a rescaled BMPE is again a (time changed) BMPE it follows that $a_R(w_1,\wo_1,w_2,\wo_2) = a_{cR}(c w_1, c \wo_1, c w_2,c \wo_2)$ for any $c>0$. 
In particular, this implies that 
\[
 \sup_{(w_1,\wo_1), (w_2,\wo_2) \in \Delta_r } a_{4r}(w_1,\wo_1,w_2,\wo_2) \leq 1-\Cr{0}. 
\]
The coupling in Lemma \ref{BMPEcoup} gives a coupling of the two processes up until time $\rho_4$.  If the coupling constructed in Lemma \ref{BMPEcoup} doesn't succeed by this time (i.e., if $\tcp > \rho_4$) then from time $\rho_4$ to $\rho_{16}$ we can use a rescaled version of the coupling from Lemma \ref{BMPEcoup} to obtain
%Then, it follows from the {\color{blue}strong Markov property} that for any $(w_1,\wo_1),(w_2,\wo_2) \in \Delta_1$ that by conditioning on the processes
% up to time $\rho_{4}$
 \begin{align*}
 &a_{16}(w_1,\wo_1,w_2,\wo_2) \\
 &= E\left[ a_{16}( W_1(\rho_4),  S_1(\rho_4) ,  W_2(\rho_4) ,  S_2(\rho_4) ) \ind{ \tcp > \rho_4 } \mid  (W_i(0),S_i(0)) = (w_i,\wo_i), \, i=1,2 \right] \\
 &\leq \left( \sup_{(w'_1,\wo'_1), (w'_2,\wo'_2) \in \Delta_4} a_{16}(w'_1,\wo'_1,w'_2,\wo'_2) \right) a_4(w_1,\wo_1,w_2,\wo_2)\\
 &\leq (1-\Cr{0})^2. 
\end{align*}
Similarly, we can show that 
\[
\sup_{(w_1,\wo_1), (w_2,\wo_2) \in \Delta_1 } a_{4^k} (w_1,\wo_1,w_2,\wo_2) \leq (1-\Cr{0})^k, \qquad \forall k\geq 1.  
\]
%Since $a_R(\cdot)$ is non-increasing in $R$ it follows that if
Therefore, if $4^k \leq R < 4^{k+1}$ we have 
\[
\sup_{(w_1,\wo_1), (w_2,\wo_2) \in \Delta_1 } a_{R} (w_1,\wo_1,w_2,\wo_2) \leq (1-\Cr{0})^k 
\leq(1-\Cr{0})^{\log_4(R)-1} 
= (1-\Cr{0})^{-1} R^{-\log_4(\frac{1}{1-\Cr{0}})}. 
\]
This completes the proof of the corollary with $\Cr{one} = (1-\Cr{0})^{-1}$ and
% $K_0 = \frac{1}{ \log_4(\frac{1}{1-\Cr{0}}) }$. 
$\frac{1}{K_0}= \log_4(\frac{1}{1-\Cr{0}})$. 
\end{proof}

% Unless stated otherwise, from now on we shall write
% $W=(W(t))_{t\ge 0}$ for $(\theta^+,\theta^-)$-BMPE. If its starting
% triple $(\wu,w,\wo)$ is not specified then it is assumed to be $(0,0,0)$.

\section{Discretizations of BMPE}\label{sec:disc}

Recall that the standard BMPE has Brownian scaling, that is for every
$\epsilon>0$
\[\{(I(t),W(t),S(t))\}_{t\ge 0}\overset{\text{Law}}{=}\{ (\epsilon
  I(\epsilon^{-2}t),\epsilon W(\epsilon^{-2}t),\epsilon
  S(\epsilon^{-2}t))\}_{t\ge 0}.\]

\subsection{Basic BMPE-walk.}\label{walk} Our first step will be to define a natural
sequence of random walks $\{(I_k, W_k, S_k)\}_{k\ge 0}$ which after
rescaling converges to BMPE. Set $\tau_0=0$ and let
\[(I_k, W_k,S_k)=(I(\tau_k),W(\tau_k),S(\tau_k)),\quad
  \tau_{k+1}\coloneqq \inf\{t>\tau_k:\ |W(t)-W(\tau_k)|=1\},\ \ k\in\N_0.\]
The walk $(I_k,W_k,S_k)$, $k\ge 0$, is markovian and can be also be
constructed directly by specifying its transition probabilities. Set
$(I_0,W_0,S_0)=(0,0,0)$ and define the
transition probabilities as follows.
\begin{itemize}
\item ``In the bulk'', i.e.\ on the set $\{I_k+1\le W_k\le S_k-1\}$, \[W_{k+1}=W_k\pm1\ \text{with equal probabilities,}\ I_{k+1}=I_k,\ S_{k+1}=S_k.\]
\item ``At the extrema'', i.e.\ on $\{S_k-W_k<1\}\cup\{W_k-I_k<1\}$,
  \begin{align*}
  P(W_{k+1}=W_k+1\mid I_k,W_k,S_k)&=1-P(W_{k+1}=W_k-1 \mid I_k,W_k,S_k)
\\&=P(W(\tau(1, I_k,W_k,S_k))=W_k+1).
\end{align*}
\end{itemize}
Now we need to see what happens to the extrema. This is easy when
$k>0$ and the end point of the walk, $W_k$, lands outside of
$\bint{I_k,S_k}$. But if $k=0$ or if the walk is, say, close to the max but
``decides'' to jump to the left, the new max should be chosen
according to the distribution of the BMPE. More precisely, for every
Borel set $B$
\begin{align*}
  &P(I_{k+1}\in B\mid I_k,W_k,S_k)=P(I(\tau(1,I_k,W_k,S_k)\in B);\\
  &P(S_{k+1}\in B\mid I_k,W_k,S_k)=P(S(\tau(1,I_k,W_k,S_k)\in B).
\end{align*}

\begin{prop}\label{DiscBMPE}
  Let 
  $(I_0,W_0,S_0)=(0,0,0)=(I(0),W(0),S(0))$.
  Then for each $ T>0 $
  \begin{equation}
\label{labe1}
 \sup_{0 \leq s \leq T} \left\vert \epsilon W(\epsilon^{-2}s)- \epsilon W_{\fl{\epsilon^{-2}s}} \right\vert\overset{\text{P}}{\longrightarrow} 0\quad\text{as }\ \epsilon\to 0.
\end{equation}
Here $\overset{\text{P}}{\longrightarrow}$ denotes the convergence in
probability.
% as $\epsilon\downarrow 0$
%   \[(\epsilon W_{\fl{\epsilon^{-2}t}})_{t\ge 0}\ \overset{J_1}{\underset{n\to\infty}{\Longrightarrow}}\ (W(t))_{t\ge 0}.\]
\end{prop}
The proof of this proposition is given in the Appendix.

\subsection{Modified BMPE-walk.}\label{mwalk} For our coupling it will be convenient to have a slightly modified
discretization which allows for small shifts of the running
extrema. For each $\epsilon>0$ we shall need a process
$\{(\tie_k,\twe_k,\tse_k)\}_{k\ge 0}$ adapted to some filtration
$({\cal F}^\epsilon_k)_{k\ge 0}$ and satisfying properties listed
below. Since we shall be coupling this process with a rescaled ERW,
the filtration will contain information about both processes. At this
time we shall make only the necessary specifications and shall not
describe the filtration. Our goal is to show that if we have a family
of processes indexed by $\epsilon$ which satisfies the properties
listed below then after rescaling this family converges weakly to a
BMPE.

We wish that the evolution of $\{(\tie_k,\twe_k,\tse_k)\}_{k\ge 0}$ be
close to that of the walk $\{(I_k,W_k,S_k)\}_{k \geq 0}$ when $\epsilon$
is small.  To describe this ``closeness'' we
divide up the interval $(-1, 1)$
into $2L = 2 \fl{\epsilon^{-3(K_0+1)}}$  equal intervals with disjoint
interior of length $L^{-1}$,
\begin{equation}
  \label{Jell}
  % J_\ell=L^{-1}[\ell-1,\ell),\quad \ell\in\{1,2,\dots,L\}.
  (-1,1)=\left(\bigcup_{\ell=-L}^{-1}\left(\frac{\ell}{L},\frac{\ell+1}{L}\right]\right)\bigcup\left(\bigcup_{\ell=1}^L\left[\frac{\ell-1}{L},\frac{\ell}{L}\right)\right)=:\bigcup_{0<|\ell|\le L}
  J_\ell.
\end{equation}
Here $K_0 $ is the constant from Lemma \ref{BMPEcoup1}. Then
$[x,x+1)=\left(\bigcup_{\ell=1}^L J_\ell\right)+x$ where $J+x$ is simply the
translation of set $J$ by $x$.

\medskip

{\em Properties of $\{(\tie_k,\twe_k,\tse_k)\}_{k\ge 0}$.} We shall
require that the process $\{(\tie_k,\twe_k,\tse_k)\}_{k\ge 0}$ satisfy the
following three conditions.
\begin{enumerate}[(i)]
\item (Starting point) $\tie_0=\twe_0=\tse_0=0$.
\item (Steps of the walk)
  $\twe_{k+1}\in\{\twe_k-1,\twe_k+1\}$ with probability
  1 for all $k\ge 0$ and
  \[
    P(\twe_{k+1}=\twe_k+1\mid {\cal F}^\epsilon_k)=
    1-P(\twe_{k+1}=\twe_k-1\mid {\cal
      F}^\epsilon_k)=P(W(\tau(1, \tie_k,\twe_k,\tse_k))=\twe_k+1).
  \]
  where the right hand side probabilities are those for the BMPE (see
  notation \eqref{hitp}, \eqref{exit}).
\item (``Choice'' of extrema) If
  $\tie_k+1\le \twe_k\le \tse_k-1$ then
  $\tie_{k+1}=\tie_k$ and $\tse_{k+1}=\tse_k$.  
  The definitions when $\twe_k$ is close to its minimum or
  maximum are  more complicated.
  Set
  \[x_k=\twe_k,\quad a_k=\tie_k-\twe_k,\quad b_k=\tse_k-\twe_k.\] Note
  that the first step is special as both $|a_0|<1$ and $b_0<1$ so that
  we need to choose a new minimum and a new maximum. For all other
  steps only one extremum might need to be changed. Suppose for
  definiteness that $b_k<1$ and we need to determine a new maximum. If
  $b_k<1$ and the walk moves to the right then the new maximum is
  simply $x_k+1$ and as in (ii)
\[
  P(\tse_{k+1}=x_k+1\mid {\cal
    F}^\epsilon_k)=P(\twe_{k+1}=\twe_k+1\mid {\cal
    F}^\epsilon_k)=P(W(\tau(1,a_k,0,b_k))=1).\] If the walk moves to
the left then we first choose an intermediate index $\ell_k$ according
to probabilities
    \[
    P( \ell_k  = \ell)  = 
    P(S(\tau(1,a_k,0,b_k))\in J_\ell \mid W(\tau(1,a_k,0,b_k))=-1 ), \quad\ell\in\{1,2,\dots,L\}.
\]
Given ${\cal F}^\epsilon_k$ and the index $ \ell_k$ we ``pick''
the value of $\tse_{k+1}$ within
$((J_{\ell_k }\cup J_{\ell_k+1})\cap[0,1)) +x_k
$ arbitrarily, provided that $\tse_{k+1}$ is
${\cal F}^\epsilon_{k+1}$-measurable, i.e.\ the process remains
adapted to the filtration ${\cal F}^\epsilon_k,\ k\ge 0$.
\end{enumerate}

When we define the coupling of $\twe$ with our rescaled ERW we shall construct
a particular version of $\twe$ explicitly. 
  \begin{thm}\label{dfc}
There exists a constant $C>0$ such that for all sufficiently small $\epsilon>0$ the following holds: 
If $\twe$ satisfies (i)-(iii) above then we can couple $\{\twe_k\}_{k\geq 1}$ with a basic BMPE-walk $\{W_k\}_{k\geq 1}$ so that
for any $K>0$ we have 
\begin{equation}\label{twecoup}
 P\left( \twe_k = W_k, \, \forall k \leq K \epsilon^{-2} \right) \geq 1-CK\epsilon. 
\end{equation}
\end{thm}
Since rescaling implies that $\{ \epsilon W(\epsilon^{-2}t) \}_{t\geq 0} \overset{\text{Law}}{=} \{ W(t) \}_{t\geq 0}$, then Proposition \ref{DiscBMPE} and Theorem \ref{dfc} immediately imply the following corollary. 
\begin{cor}\label{dfccor}
If $\twe$ satisfies (i)-(iii)
    above then for every $\epsilon>0$ there exists a BMPE $W^ \epsilon$ such that for all $\delta,\,T>0$
\[
 \lim_{\epsilon \to 0} P\left( \sup_{0 \leq s \leq T} \left\vert W^{\epsilon}(s)-  \twe_{\fl{\epsilon^{-2}s}} \right\vert > \delta \right) = 0.
\]
%and therefore $(\epsilon \twe_{\fl{\epsilon^{-2}t}})_{t\ge 0}\ \overset{J_1}{\underset{\e\to 0}{\Longrightarrow}}\ (W(t))_{t\ge 0}$. 
\end{cor}

\begin{proof}[Proof of Theorem \ref{dfc}]
Our proof will in fact prove the following stronger statement than \eqref{twecoup}. We will show that with probability at least $1-C K \epsilon$ we have for all  $k\le K \epsilon^{-2}$
\begin{equation}\label{WSIcoup}
\twe_k = W_k, , \quad \fl{S_k} = \fl{\tse_k}, \quad \ceil{I_k} = \ceil{\tie_k}, \quad  | \tse_k - S_k | \leq 2L^{-1},\, \quad\text{and}\quad  | \tie_k - I_k | \leq 2 L^{-1}. 
\end{equation}

For the first step of the coupling, we use a single BMPE stopped when
exiting $(-1,1)$ to generate both $(I_1,W_1,S_1)$ and
$(\tie_1,\twe_1,\tse_1)$. If the BMPE exits to the left so that
$I_1=W_1 = -1$ and $S_1 \in J_\ell$ for some $0<\ell \leq L$ then we
let $\tie_1 = \twe_1 = -1$ and choose
$\tse_1 \in \left( J_\ell \cup J_{\ell+1} \right) \cap [0,1)$ in some
way that is $\mathcal{F}^{\epsilon}_1$-measurable.  Thus, we can give
a coupling so that \eqref{WSIcoup} holds for $k=1$ with probability 1.

For later steps, we suppose \eqref{WSIcoup} holds for some $k\geq 1$. If we are ``in the bulk'' (i.e., if $I_k +1 \leq W_k \leq S_k-1$), then in the next step of the walk the minimums and maximums remain unchanged while the walks both move to the right or left with equal probabilities. That is, if \eqref{WSIcoup} holds at time $k$ when the walk is in the bulk then \eqref{WSIcoup} will also hold at time $k+1$. It remains to show how we can couple the walks when \eqref{WSIcoup} holds for some $k\geq 1$ and we are ``at the extrema.'' Without loss of generality we may assume that $I_k+1 \leq W_k$ and $W_k > S_k -1$ so that we are near the maximum. Let $x = S_k-W_k$ and $\tilde{x} = \tse_k - \twe_k$, and note that our assumptions are that $x, \tilde{x} \in [0,1)$ and $|x-\tilde{x}| < 2L^{-1}$, and assume without loss of generality that $x \leq \tilde{x}$. 
We now consider a coupling of two Brownian motions perturbed at the maximum $W$ and $\tilde{W}$ with initial conditions $(W(0),S(0)) =(0,x)$ and $(\tilde{W}(0),\tilde{S}(0)) = (0,\tilde{x})$ until they reach either $-1$ or $1$. 
We let $W(t) = \tilde{W}(t)$ for all $t \leq \tau_{-1,x} =  \inf\{s: \, W(s) \in \{-1,x\} \}$. If $W(\tau_{-1,x}) = -1$ then our coupling is complete. On the other hand, if $W(\tau_{-1,x}) = x$ then we still need to describe the remainder of the coupling. We consider two cases. \\
\noindent\emph{Case I: $x > 1-\epsilon^3$}. Given that the processes $W$ and $\tilde{W}$ reached $x \geq 1-\epsilon^3$ before $-1$, it follows from Corollary \ref{roc1} that the probability they will reach $1$ before $-1$ is at least $1-C \epsilon^3$ for some $C>0$ depending only on $\theta^+$. 
%Thus, in this case we can give a coupling so that outside of probability $C \epsilon^3$ either both maximums are unchanged and they exit $(-1,1)$ to the left or they both exit $(-1,1)$ to the right (and thus the maximums both become $1$). 
\\
\noindent\emph{Case II: $x < 1-\epsilon^3$}. If $\tilde{x} > x$ then the processes $W$ and $\tilde{W}$ may no longer be exactly coupled after reaching $x$. However, since $|\tilde{x}-x| \leq 2L^{-1}$ it follows from Corollary \ref{BMPEcoup1} that we can couple them so that both processes and the maximums join together again before exiting the interval $[x-\epsilon^3,x+\epsilon^3]$ with probability at least $1-\Cl{1} (\frac{\epsilon^3 L}{2} )^{-\frac{1}{K_0}} \geq 1-c\epsilon^3$. \\
In either case, we have shown that we can create a coupling so that outside of probability $C \epsilon^3$ the processes exit out the same side of the interval $(-1,1)$ and that at this time the maximums are either unchanged or both changed to the same value. We then apply this to the BMPE-walks by using the process $W$ to generate $W_{k+1}$ and $S_{k+1}$ and $\tilde{W}$ (plus additional randomness which is $\mathcal{F}^\epsilon_{k+1}$-measurable) to generate $\twe_{k+1}$ and $\tse_{k+1}$ so that with probability at least $1-C\epsilon^3$ we have $W_{k+1} = \twe_{k+1}$ and $|S_{k+1}-\tse_{k+1}| \leq 2L^{-1}$.  

We have therefore shown that if \eqref{WSIcoup} holds for some $k\geq 1$ then with probability at least $1-C \epsilon^3$ it again holds for $k+1$. This is enough to show that $\eqref{WSIcoup}$ holds for all $k\le K\epsilon^{-2}$ with probability at least $1-CK\epsilon$, and this finishes the proof of the theorem. 
\end{proof}

\section{Toolbox}\label{sec:tb}

We would like to argue that our ERW $X=\{X_n\}_{n\ge 0}$ considered
only at the stopping times $\{T_k^{\epsilon,n}\}_{k\geq 0}$ defined in \eqref{tenk}
%only at the times when it moves by $\fl{\epsilon\sqrt{n}}$ from its current location 
and scaled down by $\fl{\epsilon \sqrt{n}}$
behaves essentially as a modified BMPE-walk described in
Section~\ref{mwalk}. The important issue here is that the ERW moves in
a random environment and the environment is modified by the walk.  In
this section we shall collect 
%our main tools 
a number of results concerning BLPs in random environments which will
be helpful as long as we know that the cookie environment is ``good''
in some way.
%Our tools are several results about BLP in random environments.
We begin this section with several definitions which will be used to quantify exactly what we mean by ``good.''

%We write $\br^{\pm}=(r^\pm(1),\dots, r^\pm(N))$, where
%\begin{align}
 % \label{vecr}
%  r^+(i)&=\lim_{n\to\infty}E_i[U^+_1-U^+_0\mid U_0=n]=\lim_{n\to\infty}E_i[V^+_1-V^+_0\mid V^+_0=n]-1;\nonumber\\\nonumber r^-(i)&=\lim_{n\to\infty}E_i[U^-_1-U^-_0\mid U^-_0=n]=\lim_{n\to\infty}E_i[V^-_1-V^-_0\mid V^-_0=n]-1;\\ \frac{\nu}{2}&=r^+(i)+r^-(i)+1,\ \ \forall i\in\{1,2,\dots,N\}.
%\end{align}

Recall the parameters $\mathbf{r}^\pm = (r^\pm(1),\dots, r^\pm(N))$ introduced in Section \ref{sec:parameter}
and that that $R^x_j,\ j\in\N$ is the cookie
Markov chain at site $x$ with values in $\{1,2,\dots,N\}$ so that
$\omega_x(j)=p(R^x_j)\in(0,1)$ is the probability that the ERW jumps to the
right after the $j$-th visit to $x$. 
\begin{defn}
  Let $\alpha\in(0,1)$, $m\in\N$, and $\rho\in\R$. The first cookies
  $(R^z_1)_{z\in\Z}$ are said to be
  $(m^\alpha,\rho)$-good on a discrete interval $I$ if for every
  discrete subinterval $J\subset I$ of length $\fl{m^\alpha}$
  \begin{align*}
    &\left|\frac{1}{m^\alpha}\sum_{z\in
    J}r^+(R^z_1)-\rho\right|\le\frac{1}{\ln{m}}, \quad \text{or, equivalently (by \eqref{vecr}),}\\
    &\left|\frac{1}{m^\alpha}\sum_{z\in
      J}r^-(R^z_1)-\left(\frac{\nu}{2}-1-\rho\right)\right|\le\frac{1}{\ln{m}}.
  \end{align*}
  We shall say that the family of first cookie environments is
  $m^\alpha$-good on some interval $I$ if there is a constant $\rho$ for
  which it is $(m^\alpha,\rho)$-good.
\end{defn}

% \begin{rem}
%  Note that the equivalence of the two definitions above follows from the parameter relation in \eqref{vecr}. 
% \end{rem}

The relevance of the above definition is that if the first cookie environment in $I$ is (approximately) i.i.d.\ with marginal $\eta'$, then we expect the interval to be $(m^\alpha,\rho)$-good with $\rho = \eta' \cdot \mathbf{r}^+$. 
In particular, we expect intervals in the initial cookie environment (which is i.i.d.\ $\eta$) to be $(m^\alpha, \frac{\nu \theta^+}{2})$-good, whereas if an interval has first cookie environments which are approximately i.i.d. $\pi^+$ or $\pi^-$ then we expect the interval to be $(m^\alpha,0)$-good or $(m^\alpha,\frac{\nu}{2}-1)$-good, respectively.

\begin{defn}\label{lifting}
  Given $x>0$, $\epsilon>0$, and $a,b\in\R$, $a<b$, a first cookie
  environment on the interval
  $\bint{a,b}$ is said to be
  {\em $x$-lifting from the left (resp.\ right)} if for a $V^+$
  (resp.\ $V^-$) process which uses the environment with these first
  cookies for generations $1,2,\dots,\fl{b}-\ceil{a}+1$ (resp.\ $\fl{b}-\ceil{a}+1,\dots,2,1$)
  and starts with $0$ particles in generation $0$
  \[P(\tau^{V^+}_x\le
    b-a)\ge 1-\epsilon^3\ \ (\text{resp. }P(\tau^{V^-}_x\le
    b-a)\ge 1-\epsilon^3).\]
\end{defn}

\begin{defn}\label{grounding}
  Given $x>0$, $\epsilon>0$, and $a,b\in\R$, $a<b$, a first cookie
  environment on the interval $\bint{a,b}$ is said to be {\em
    $x$-grounding from the left (resp.\ right)} if for a $U^+$ (resp.\
  $U^-$) process which uses the environment with these first cookies
  for generations $1,2,\dots,\fl{b}-\ceil{a}+1$ (resp.\
  $\fl{b}-\ceil{a}+1,\dots,2,1$) and starts with $\fl{x}$ particles in
  generation $0$
  \[P(\sigma^{U^+}_0\le b-a)\ge 1-\epsilon^3\ \ (\text{resp. }P(\sigma^{U^-}_0\le b-a)\ge 1-\epsilon^3).\]
\end{defn}

\subsection{The ``full'' diffusion approximation  in product
  environments }\label{daiid}

In this subsection we extend the results of \cite[Lemma 6.1]{kpERWMCS}
to either the convergence on $D([0,\infty))$ (for $V^+$ processes with
positive drifts) or the convergence up to the first hitting time of
$0$. 
The diffusion approximations of the BLPs here and throughout the paper will generally be of the form 
\begin{equation} \label{daa}
 dY(t)= D(t) \, dt+\sqrt{\nu (Y(t))_+}dB(t)% ,\quad Y(0)=y
 , 
\end{equation}
where the constant $\nu>0$ is the parameter which was defined earlier
in \eqref{nudef} and the drift $D(t)$
  is a nonrandom piecewise constant function of time depending on the
  particular BLP being considered ($U^\pm$, $V^\pm$, or concatenation
  of those) and the distribution of the first cookies.

We note that if $Y$ is defined as in \eqref{daa} and $D(t)\equiv D$
then the process $\frac{4 Y(t)}{\nu}$ is a BESQ process of generalized
dimension $4D/\nu$.  Weaker versions of diffusion approximation for
BLPs with initial cookie distributions i.i.d.\ $\eta$ were proved
earlier in \cite{kpERWMCS} where the drift
$D(t)\equiv \eta \cdot \br^\pm$ in the case of $U^\pm$ and
$D(t)\equiv 1+\eta\cdot \br^\pm $ in the case of $V^\pm$.

\begin{defn}\label{admit}
  We shall say that a family of stochastic processes
  $Z^m=\{Z_k^m\}_{k\ge 0}$, $m\in\N$,  admits an approximation by a BESQ process of
  generalized dimension $\dim\in\R$ if $\forall \delta>0$ and
  $\forall y>\delta$ the rescaled processes
  $m^{-1}Z^m_{\fl{mt}\wedge \sigma_{n\delta}}$ with $m^{-1}Z^m_0\to y$
  converge weakly in the standard ($J_1$) Skorokhod topology to a
  positive multiple of BESQ$^{\dim}$ process
  $Y(t\wedge \sigma_{\delta})$ with $Y(0)=y$.
\end{defn}
In terms of the above definition, and recalling the relations \eqref{vecr} and \eqref{thedef}, the arguments in \cite{kpERWMCS} show that if the initial cookie distribution is i.i.d.\ $\eta$ then the 
BLPs $U^+$, $U^-$, $V^+$, and $V^-$ admit approximation by a BESQ processes of generalized dimensions
$2\theta^+$, $2\theta^-$, $2(1-\theta^-)$ and
$2(1-\theta^+$), respectively.  

%Recalling the relations \eqref{vecr} and \eqref{thedef},
% \[\theta^\pm=\frac{2\eta\cdot\br^\pm}{\nu},\quad
%   \theta^++\theta^-=1-\frac2\nu,\]
%we see that when the initial environment is i.i.d.\ $\eta$, the BESQ
%process associated to the BLPs $U^+$, $U^-$, $V^+$, and $V^-$ are of
%generalized dimension $2\theta^+$, $2\theta^-$, $2(1-\theta^-)$ and
%$2(1-\theta^+$), respectively.  
Since we are assuming in this paper
that $\max\{\theta^+,\theta^-\} < 1$, the dimensions of the BESQ
processes associated to $V^\pm$ are strictly positive and the
dimensions of the BESQ processes associated to $U^\pm$ are strictly
less than $2$.

\begin{thm}[Diffusion approximation in i.i.d.\ 
  environments]\label{DA0}
  Assume that the cookie environment is i.i.d.\ with marginal $\eta$.
  \begin{enumerate}
  \item Suppose that $4\nu^{-1}(1+\eta\cdot\br^+)=2(1-\theta^-)> 0$
    and consider a sequence of rescaled BLPs $Y_m(t)\coloneqq m^{-1}V^+_{m,\fl{mt}}$, $t\ge 0$, with
    initial distributions $\kappa_m$, $Y_m(0)\sim\kappa_m$. If
    $\kappa_m\underset{m\to\infty}{\Longrightarrow} \kappa$ then
  \[\{Y_m(t)\}_{t\ge
      0}\overset{J_1}{\underset{m\to\infty}{\Longrightarrow}}
    \{Y(t)\}_{t\ge 0},\] where $(Y(t))_{t\ge 0}$ is the solution of
\eqref{daa} with $D(t)\equiv 1+\eta\cdot \br^+$ and $Y(0)\sim \kappa$.
\item Suppose that $4\nu^{-1}(\eta\cdot\br^+)=2\theta^+<2$ and
  consider a sequence of rescaled BLPs $Y_m(t)\coloneqq m^{-1}U^+_{m,\fl{mt}},\ t\ge 0$, with
  initial distributions $\kappa_m$, $Y_m(0)\sim\kappa_m$. If 
    $\kappa_m\underset{m\to\infty}{\Longrightarrow} \kappa$ then
  \[\{Y_m(t)\}_{t\ge
      0}\overset{J_1}{\underset{m\to\infty}{\Longrightarrow}}
    \{ Y(t\wedge \sigma^Y_0) \}_{t\ge 0},\] where $\{Y(t)\}_{t\ge 0}$ is the solution of 
    \eqref{daa} with $D(t)\equiv \eta\cdot \br^+$ and $Y(0)\sim \kappa$.
      Moreover,       
\begin{equation}
  \label{ht}
  \sigma^{Y_m}_0\ \Rightarrow \ \sigma^Y_0\ \text{ as }m\to \infty. 
\end{equation}
    \end{enumerate}
  \end{thm}
  \begin{rem}
    Part (2) also holds for the process $V^+$ if
    $4\nu^{-1}(1+\eta\cdot\br^+)=2(1-\theta^-)<2$ provided that we
    replace the drift of the $Y$ process in that part with
    $1+\eta\cdot\br^+$ and $Y_m(t)$ with
    $Y_m(t\wedge \sigma^{Y_m}_0)$. The proof needs practically no
    changes.
  \end{rem}
  The proof of this theorem is standard and, for convenience of the reader, is given in the Appendix.

\subsection{Lifting from 0 and driving to extinction in 
  i.i.d.\ environments}\label{updown}

The next two lemmas are stated for the $V^+$ process with parameter
$\theta ^-<1$. Similar statements with identical proofs hold for the
$V ^-$ process with parameter $\theta ^+<1$.
  \begin{lemma}\label{liftIID}
    If the environment is i.i.d. $\eta' $ where $\theta^- (\eta') < 1$
    and $V^+$ is the BLP with $V^+_0=0$ then
\[
 \lim_{\delta\to 0} \limsup_{m\to\infty} P\left( V^+_m \leq \delta m \right) = 0. 
\]
\end{lemma}

\begin{proof}[Proof of Lemma~\ref{liftIID}]
  The proof is based on the Dynkin-Lamperti theorem for renewal
  processes with infinite expectation, \cite[XIV.3,
  p.\,472]{Feller2}. 

  Assume first that $\theta^-\in(0,1)$ so that the diffusion
  approximation $Y=\{Y(t)\}_{t\ge 0}$ is a multiple of a
  BESQ$^{2(1-\theta^-)}$ of dimension strictly between 0 and 2. Let
  $N_m=\sum_{k=1}^m\ind{V^+_k=0}$ and $\sigma_i$ be the
  end of the $i$-th lifetime of $V^+$. Random variables
  $\sigma_1,\sigma_2,\dots$ are i.i.d.\ finite random variables with
  infinite expectation.\footnote{The lifetime $\sigma_1$ has infinite
    expectation for $\theta^-\le 1$. For $\theta^-<0$ also the probability
    that $\sigma_1=\infty$ is positive.} Then Dynkin-Lamperti theorem
  states that $(m-\sigma_{N_m})/m$ and $(\sigma_{N_m+1}-m)/m$ converge
  in distribution to random variables with explicit densities
  supported on $(0,1)$ and $(0,\infty)$ respectively.

  Given $\epsilon>0$, we can find an $s>0$ such that  $P(\sigma_{N_m+1}-m\le sm)<\epsilon$ for all sufficiently large $m$. Then
  \begin{align*}
    P(V^+_m\le \delta m)&\le P(V^+_m\le \delta m, \sigma_{N_m+1}-m> sm)+P(\sigma_{N_m+1}-m\le  sm)\\ &\le P(\sigma_{N_m+1}-m> sm\mid V^+_m\le \delta m)+P(\sigma_{N_m+1}-m\le  sm)\\ &\le P(\sigma_1> sm\mid V^+_0=\delta m)+P(\sigma_{N_m+1}-m\le  sm).
  \end{align*}
  Going from the second line to the third we used monotonicity of the
  BLP in the initial number of particles and Markov property. Taking a
  limit as $m\to\infty$  we see that
  \[\limsup_{m\to\infty}P(V^+_m\le \delta m)\le
    P^Y(\tau_0> s\mid Y(0)=\delta)+\epsilon=P^Y(\tau_0> \delta^{-2}s\mid Y(0)=1)+\epsilon.\]
  Finally, letting $\delta\to 0$ and using the fact that
  $P^Y(\tau_0=\infty\mid Y(0)=1)=0$ for $\theta^-\in(0,1)$ we get
  \[\lim_{\delta\to 0}\limsup_{m\to\infty}P(V^+_m\le \delta m)\le
    \epsilon.\] Now we can let $\epsilon\to 0$ and get the result for the case when $\theta^-\in(0,1)$.

  If $\theta^-\le 0$ then the process $V^+$ can be coupled with a
  ``smaller'' process (corresponding to $\theta^-\in (0,1)$). The
  coupling can be done by adding one or more cookies of strength
  $\min_{i\leq N} (p_i)$ before the first cookie in each stack (or a
  geometric number of these with an appropriate success probability).
  Details of how such a coupling can be constructed can be found in
  \cite[Section 5.1]{kpERWMCS}.  This will complete the proof of the
  lemma by comparison.
\end{proof}

Next we show that BLPs which evolve in environments close to
i.i.d.\ and which admit an approximation by a BESQ process of dimension
less than 2 will become extinct very soon after becoming
macroscopically small.

\begin{lemma}\label{sm0}
  Let $Z$ be a BLP in an i.i.d.\ cookie environment. Assume that it
  admits an approximation by a BESQ process of dimension strictly less
  than 2. Then for all $\delta,\epsilon>0$ there is a $\delta'>0$ such
  that for all sufficiently large $m$
  \[P(\sigma_0>\delta m\mid Z_0\le \delta' m)<\epsilon.\]
\end{lemma}
  \begin{proof}
    The proof  is the same as that of (5.5) in
    [KM11]. 
  \end{proof}

\subsection{BLPs in $m^{1/4}$-good environments}\label{0.25}

In this section we extend the diffusion approximation of BLPs (and
some of the resulting consequences) from the case where the first
cookie environments are i.i.d.\ to the weaker condition of
$(m^{1/4},\rho)$-good. The cost of this relaxation is that we will not
be able to get convergence of the hitting time of 0 as in
\eqref{ht}. Nevertheless, we will be able to get enough control on
this hitting time (Lemma \ref{thetruth}) for our applications
later.

A number of the results in this section hold for more than one of the
four different BLPs ($U^\pm$ or $V^\pm$). Thus, if a result holds for
one or more of these BLPs we will state the result in terms of a
generic BLP $Z$ and will state which of the four BLPs $Z$ can be (if
no restrictions are made it is assumed that $Z$ can be any of the four
BLPs).  Also, if the result concerns one of the BLPs using the cookie
environment on the interval $\bint{a,b}$, we will always assume that
if the BLP $Z$ is either $U^+$ or $V^+$ then the cookie stacks are
used to generate successive generations of the BLP from left to right
whereas if the BLP is either $U^-$ or $V^-$ then the cookie stacks are
used from right to left.

\begin{thm}[Diffusion approximation in $m^{1/4}$-good environments]\label{da0.25}
  Suppose that for some $\rho\in\R,\ T>0$ the first cookies are
  $(m^{1/4},\rho)$-good on intervals $\bint{0,mT}$ for all sufficiently
  large $m$.  Fix an arbitrary $\delta>0$ and consider a
  sequence of rescaled BLPs
  $Y_m(t)\coloneqq m^{-1}Z^m_{\fl{mt}\wedge \sigma_{\delta m}},\
  t\in[0,T]$, with initial distributions $Y_m(0)\sim \kappa_m$.  If
  $\kappa_m\underset{m\to\infty}{\Longrightarrow} \kappa$ then
  \[\{Y_m(t)\}_{0\le t\le
      T}\overset{J_1}{\underset{n\to\infty}{\Longrightarrow}}
    \{ Y(t\wedge \sigma_{\delta})\}_{0\le t\le T},\] where
  $\{Y(t)\}_{t\ge 0}$ is the solution of \eqref{daa} with
  $Y(0)\sim \kappa$ and where $D(t)$ is a constant equal to $\rho$ for
  $U^+$, $\rho+1$ for $V^+$, $\nu/2-1-\rho$ for $U^-$, and
  $\nu/2-\rho$ for $V^-$.
 \end{thm}
 The proof of this theorem is given in the Appendix.  

 The diffusion approximation in Theorem \ref{da0.25} guarantees the
 convergence as long as the processes
 stay macroscopically away from zero. Nevertheless, when the limiting
 diffusion process is a BESQ$^0$ process, the diffusion approximation
 can be extended to all times (see Corollary~\ref{da0.25dead}).  This
 will follow from Theorem \ref{da0.25} together with the following
 lemma which says that when the BLP becomes ``macroscopically small''
 and then it cannot become ``macroscopically much larger'' during a
 fixed macroscopic time period.

  \begin{lemma}\label{dead}
    Let $Z^m$ be a BLP using the cookies on the interval $\bint{0,m}$,
    and suppose that the first cookies on intervals $\bint{0,m}$ are
    $(m^{1/4},\rho)$-good where the parameter $\rho$ is such that the
    family $Z^m$, $m\in \N$, admits an approximation by a BESQ process
    of dimension $0$. Then $\forall \epsilon>0,\ \forall\delta>0$
    there is a $\delta' \in(0,\delta)$ such that for all sufficiently
    large $m$
  \[P(\max_{j\le m}Z^m_j\ge \delta m\mid Z^m_0\le \delta ' m)<\epsilon.\]
  \end{lemma}
  
 \begin{rem}
We will only apply Lemma \ref{dead} in the case of the BLP $U^+$ or $U^-$. Due to Theorem \ref{da0.25}, we see that the condition that the approximating BESQ process is of dimension 0 if $\rho=0$ in the case of $U^+$ or $\rho=\frac{\nu}{2}-1$ in the case of $U^-$. 
 \end{rem}
 
  \begin{proof}
  Recall that $\tau^Z_x$ is the first entrance time of the process $Z$
  to the interval $[x,\infty)$. It is notationally convenient to prove
  an equivalent statement, namely, that
  $\forall\epsilon>0,\  \forall L>1$ there is a $k\in\N$ such that for
  all sufficiently large $m$
   \begin{equation*}
     P(\tau^{Z^m}_{2^km}\le 2^kLm\mid Z^m_0=m)<\epsilon.
   \end{equation*}
   The equivalence can be easily seen from the following relabeling
   (from the last expression to the original):
   $m\to \delta' m,\ 2^k\to \delta/\delta',\ L\to 1/\delta$.

   Our proof is based on comparison of $Z^m_j$, $j\ge 0$, with a
   modified process $\bzmk_j$, $j\ge 0$, and a diffusion
   approximation. The process $\bzmk$ coincides with $Z^m$ up until
   $\sigma^{Z^m}_{m/2}\wedge \tau^{Z^m}_{2^km}$ at which it
   resets to $m$. After the reset it continues as a ``fresh copy'' of $Z^m$
   but in the environment shifted by
   $\sigma^{Z^m}_{m/2}\wedge \tau^{Z^m}_{2^km}$, and so on.  Let
   $T_{k,i}^m$, $i\in\N$, be the sequence of waiting times between
   consecutive resets of $\bzmk$ and $N^m_k$ be the total number of resets
   until the first reset from the upper boundary inclusively. Then by
   construction and monotonicity of BLPs $\bzmk_j\ge Z^m_j$ for
   all $j<\sum_{i=1}^{N^m_k}T^m_{k,i}$ and, therefore,
   \[P(\tau^{Z^m}_{2^km}\le 2^kLm\mid Z^m_0=m)\le
     P\left(\sum_{i=1}^{N^m_k}T^m_{k,i}\le 2^kLm\mid
       Z^m_{k,0}=m\right).\] We conclude that it is enough to show
   that $\forall \epsilon>0$ and $\forall L>0$ there is a $k\in\N$
   such that for all sufficiently large $m$
   \begin{equation}
     \label{ref}
     P\left(\sum_{i=1}^{N^m_k}T^m_{k,i}\le 2^kLm\mid
       Z^m_{k,0}=m\right)<\epsilon.
   \end{equation}

By Theorem~\ref{da0.25} the process $m^{-1}Z^m$ admits an approximation
   by the zero dimensional BESQ process $Y$ with the starting point
   $Y(0)=1$ up to the time $\sigma^Y_{1/2}$. Note
   that
   \[P(\sigma^Y_{1/2}<\tau^Y_{2^k}\mid Y(0)=1)=1-(2^{k+1}-1)^{-1},\quad\forall
     k\in\N.\] We shall also consider a right-continuous process
   $\Bar{Y}_k(t),\ t\ge 0$, which coincides with $Y$ up to the time
   $\sigma^Y_{1/2}\wedge \tau^Y_{2^k}$, jumps to $1$ at time
   $\sigma^Y_{1/2}\wedge \tau^Y_{2^k}$ and continues to follow a
   ``fresh copy'' of $Y$ until it again hits the boundary of $[1/2,2^k]$ at
   which time $\Bar{Y}_k$ resets to $1$, and so on.  Let
   $T_{k,i},\ i\in\N$, be a sequence of waiting times between
   consecutive jumps of $\Bar{Y}_k$. Random variables
   $T_{k,i},\ i\in\N$, are i.i.d.\ and have the same distribution as
   $\sigma^Y_{1/2}\wedge \tau^Y_{2^k}$. Denote by $N_k$ be the number
   of jumps of $\Bar{Y}$ until the first jump down from $2^k$ to $1$
   inclusively. By construction, $N_k$ has a geometric distribution on
   $\N$ with parameter $(2^{k+1}-1)^{-1}$. Given an arbitrary
   $\epsilon>0$ and $L>0$ we shall first show that there is a $k\in\N$
   such that
   \begin{equation}
     \label{ref1}
     P\left(\sum_{i=1}^{N_k}T_{k,i}\le 2^kL\right)\le
    \epsilon/2
   \end{equation}
   and then argue that \eqref{ref} holds by Theorem~\ref{da0.25}.

   For any fixed $\alpha\in(0,\epsilon/4)$, $k_0\in\N$, and
   $k\ge k_0\vee \log_2\frac1\alpha$ we have
   \begin{align}
     P\left(\sum_{i=1}^{N_k}T_{k,i}\le 2^kL\right)&\le P\left(\sum_{i=1}^{\fl{\alpha 2^k}}T_{k,i}\le 2^kL,\ N_k\ge\fl{\alpha 2^k}\right)+P(N_k< \fl{\alpha 2^k})\label{fin}\\ &\le P\left(\sum_{i=1}^{\fl{\alpha 2^k}}T_{k_0,i}\le 2^kL\right)+1-(1-(2^{k+1}-1)^{-1})^{\fl{\alpha 2^k}}.\nonumber
   \end{align}
Centering, we get that
\begin{align*}
  P\left(\sum_{i=1}^{N_k}T_{k,i}\le 2^kL\right) &\le P\left(\frac{1}{\fl{\alpha 2^k}}\sum_{i=1}^{\fl{\alpha 2^k}}(T_{k_0,i}-E(T_{k_0,i}))\le -\left(E(T_{k_0,1})-\frac{L2^k}{\fl{\alpha 2^k}}\right)\right)+\frac{\alpha 2^k}{2^{k+1}-1}\\ &\le P\left(\frac{1}{\fl{\alpha 2^k}}\left|\sum_{i=1}^{\fl{\alpha 2^k}}(T_{k_0,i}-E(T_{k_0,i}))\right|\ge \left(E(T_{k_0,1})-\frac{2L}{\alpha}\right)\right)+\frac{\alpha 2^k}{2^{k+1}-1}.
\end{align*}
   Applying the optional stopping theorem to the local martingale
   $Y(t)\ln Y(t)-t$ we obtain
   \[E(T_{k_0,1})=\frac{\ln 2(2^{k_0}k_0-2^{k_0}+1)}{2^{k_0+1}-1}\ \
     \text{and}\ \ \lim_{k_0\to\infty}\frac{E(T_{k_0,1})}{(2^{-1}\ln
       2) k_0}=1.\] Thus, we can choose $k_0$ so that
   $2L/\alpha<E(T_{k_0,1})/2$ and conclude by the weak law of large
   numbers that \eqref{ref1} holds for all sufficiently large $k$.

   Return now to the process $\bzmk$. By Theorem~\ref{da0.25} and the
   continuous mapping
   theorem
   \[\frac{T^m_{k,i}}{m}\underset{m\to\infty}{\Rightarrow} T_{k,i},\
     \forall i,\,k\in\N.\] Since the cookie stacks, given the first
   cookies, are independent and $\Bar{Z}^m$ has the strong Markov property,
   $\{ T^m_{k,i} \}_{i\in\N}$ is a sequence of independent random
   variables while $\{ T_{k,i}\}_{i\in\N}$ is an i.i.d.\
   sequence. Therefore, for each fixed $n\in\N$ we also have that
   \begin{equation}
     \label{wc1}
     \sum_{i=1}^n\frac{T^m_{k,i}}{m}\underset{m\to\infty}{\Rightarrow}\sum_{i=1}^nT_{k,i}.
   \end{equation}
  Next, we claim that $N^m_k\underset{m\to\infty}{\Rightarrow}N_k$.
   Indeed, denoting by $p^m_{k,i}$ the probability that the $i$-th
   reset of $\bzmk$ is from the upper boundary we have, again by
   Theorem~\ref{da0.25}, that $\forall k,n\in\N$
   \begin{equation}
     \label{wc2}
     P(N^m_k>n)=\prod_{i=1}^n(1-p^m_{k,i})\underset{m\to\infty}\to
     \left(1-\frac{1}{2^{k+1}-1}\right)^n=P(N_k>n).
   \end{equation}
   Repeating the same steps for $T^m_{k,i}$ and $N^m_k$ as in \eqref{fin}
   and using the weak convergence results \eqref{wc1} and \eqref{wc2}
   we conclude that with the same choice of $k_0$ and $k$ as above the
   inequality \eqref{ref} holds for all sufficiently large $m$.
 \end{proof}

\begin{cor}\label{da0.25dead}
  Assume the conditions of Lemma~\ref{dead} and consider a
    sequence of rescaled BLPs
    $Y_m(t)\coloneqq m^{-1}U^+_{m,\fl{mt}},\ t\in[0,T]$, with $Y_m(0)\sim \kappa_m$. If $\kappa_m\underset{m\to\infty}{\Longrightarrow}\kappa$ then
  \[\{Y_m(t)\}_{0\le t\le T}\overset{J_1}{\underset{m\to\infty}{\Longrightarrow}}
    \{Y(t)\}_{0\le t\le T},\] where $\{Y(t)\}_{t\ge 0}$ solves \eqref{daa} with $D(t)\equiv 0$ and $Y(0)\sim \kappa$.
\end{cor}
For the proof of Corollary~\ref{da0.25dead} we refer to  Remark~\ref{corproof}.

 Note that
 Corollary~\ref{da0.25dead} does not imply that the stopping times
 $\sigma^{Y_m} _0$ converge in distribution to $\sigma^Y_0$.  However,
 our environments have additional properties which will allow us to
 get more information.  
 
\begin{asm}\label{asmG}
  Let $c_m\in\Z_+/m$ and $c_m\to c\in[0,1]$. Assume that the first
  cookie environments are $(m^{1/4}, 0)$-good on $\bint{0, c_m m}$ and
  are i.i.d.\ with marginal $\eta$ on $\bint{c_m m, m}$.  Suppose also that if
  $c = 1$, then the interval $\bint{(1-\delta_1) m,m}$ is
  $\delta_2 m$-grounding for some $\delta_1,\delta_2>0$.
\end{asm}

\begin{lemma} \label{thetruth} Fix $ \epsilon > 0$ and
    an arbitrary sequence $y_m\in\Z_+/m$, $m\in\N$.  Let
    $\{Z^m_k\}_{k\le m}$, $m\in\N$, be $U^+$ processes starting at
    $y_m m$ in first cookie environments on $\bint{0,m}$ satisfying
    Assumption \ref{asmG}. Let $Y_m(t),\ t \in [0,1]$ be a solution
  of \eqref{daa} with $D(t)=(\eta\cdot \br^+)\ind{t>c}$ and $Y_m(0)=y_m$.
 If $\delta_1,\delta_2>0$ are sufficiently small then there is an
 $m_0=m_0(\epsilon,c,\delta_1,\delta_2)$ which is independent of the
 choice of sequence $y_m$, $m\in\N$, such that for all $m\ge m_0$
\begin{equation}\label{thetruth-bound}
  \left| P\left(\sigma^{Z^m}_0 < m \right) - P\left(\sigma^{Y_m}_0<1\right)
  \right| < 4 \epsilon^3.
\end{equation}
\end{lemma}

\begin{proof}
 % We fix $\epsilon > 0$.  
 % We first show that it is enough prove a bound like \eqref{thetruth-bound} for any fixed sequence $\{y_m\}_{m\geq 1}$ since the uniformity of $m_0$ will then follow from the monotonicity (in $y_m$) of $P\left(\sigma^{Z^m}_0 < m \right)$
% and $P\left(\sigma^{Y_m}_0<1\right)$.
  To reduce notation we write
  $g(x) = P(\sigma^Y_0 < 1 \mid Y(0) = x)$ where $Y$ solves the same equation as all the $Y_m$.  The function $g$ is
  decreasing and it is easily seen that $g(x)$ is continuous and tends
  to zero as $x\to\infty$.  Thus we can set $g(\infty)=0$ and choose
  $0 < y^1 < y^2 < \ldots< y^R<y^{R+1}=\infty $ so that
  $g(y^i) - g(y^{i+1}) < \epsilon^3$ for all
  $i\in\{1,2,\dots,R\}$. For each $y^i$, $i\in\{1,2,\dots,R\}$, we
  define a sequence
  $y^i_m = \frac{\ceil{my^i}}{m}\in\left[y^i,y^i+\frac1m\right)$,
  $m\in\N$.  We claim that it is sufficient to show that there is an
  $m_0=m_0(\epsilon,c,\delta_1,\delta_2)$ such that for all $m \ge m_0$
  \begin{equation} \label{monman} \left\vert P(\sigma^{Z^m}
    _0 <m \mid Z^m_0 = y^i _m m )-g(y^i) \right\vert < 3 \epsilon^3 \quad
    \text{for all }i\in\{1,2,\dots,R\}.
\end{equation}
Indeed, suppose \eqref{monman} holds and $y_m \in \Z_+/m$. Then
$y^i_m\le y_m< y^{i+1}_m$ for some $i\in\{1,2,\dots,R\}$. Since for
our choice of the sequence $y^i_m$ we have
\[y^i\le y^i _m \le y_m \le y^{i+1}_m-\frac1m<y^{i+1}\le
  y^{i+1}_m,\] by monotonicity we get 
\begin{align*}
  P(\sigma^{Z^m} _0 < m\mid Z^m_0 = y_m m) &\le P(\sigma^{Z^m} _0 < m \mid  Z^m_0 = y^i_m  m ) < g(y^i) +3\epsilon ^3 
< g(y_m) +4 \epsilon ^3, \\
P(\sigma^{Z^m} _0 < m \mid  Z^m_0 = y_m m) &\ge P(\sigma^{Z^m} _0 < m \mid Z^m_0 = y^{i+1}_m m) > g(y^{i+1}) -3 \epsilon ^3 
> g(y_m) -4 \epsilon ^3, 
\end{align*}
and
\eqref{thetruth-bound} follows. Thus, we need only to prove \eqref{monman}.

Fix an $i\in\{1,2,\dots,R\}$ and let $Y$ be a solution of \eqref{daa}
with $Y(0) = y^i$ and $Z^m $ be the $U^+ $ process with
$Z^m_0 = y^i_m m$ evolving in the cookie environment on $\bint{0,m}$ .

{\it Case $c<1$}.  
For any $\delta_3>0$ we have
\begin{align}
  \left| P\left(\sigma^{Z^m}_0 < m \right) - P\left(\sigma^Y_0<1\right) \right|
   &=\left| P\left(\sigma^{Z^m}_0 \ge m \right) - P\left(\sigma^Y_0>1\right) \right| \nonumber \\
   &\le P\left(\sigma^{Z^m}_0 \ge m,\, \sigma^{Z^m}_{\delta_3 m}\le c_mm\right)+P\left(\sigma_0^Y>1,\ \sigma_{\delta_3}^Y<c\right) \label{spl} \\ 
   &\qquad    +\left| P\left(\sigma^{Z^m}_0 \ge m,\ \sigma^{Z^m}_{\delta_3 m}> c_mm \right) - P\left(\sigma^Y_0>1,\ \sigma_{\delta_3}^Y>c\right) \right|. \nonumber
\end{align}
We claim that we can choose $\delta_3$ small enough so that all three
terms in the right hand side of \eqref{spl} are small when $m$ is
large. We shall treat the first and the last terms since the second
term can be dealt with like the first but is simpler.

That the third term tends to zero is an immediate consequence of
Theorem \ref{da0.25} and then Theorem \ref{DA0} and the fact that the
boundary of set $\left\{ \sigma_0^Y>1,\ \sigma_{\delta_3}^Y<c\right\}$
has probability zero.

For the first term in \eqref{spl}, we choose $\delta_4$ so that
$P\left( \sigma_0^Y \ge 1\mid Y(c)=\delta_4 \right) <\epsilon^3/4$.
Then we fix $ \delta_3$ so that
$( \delta_3, \delta_4, \epsilon ^3/4 )$ are as
$(\delta ', \delta, \epsilon )$ for Lemma \ref{dead}.  Next we note
that
\[\{  \sigma^{Z^m}_{0} \ge m,\ \sigma^{Z^m}_{\delta_3 m} \le c_m m
  \}\subseteq \{ \sigma^{Z^m}_{\delta_3 m} \le c_m m, \, Z^m_{c_m m} \ge \delta_4 m\}\cup\{Z^m_{c_m m} \le
  \delta_4 m,\ \sigma^{Z^m}_{0} \ge m \}.\] By Lemma \ref{dead} the
first event has probability less than $\epsilon^3/4$ by our choice of
$\delta_3$ and $\delta_4$ for $m$ large, while the probability of the
second event is similarly bounded by Theorem \ref{DA0} and our choice
of $\delta_4$. 

{\it Case $c = 1$}.  For any $\delta_3 > 0$ 
 \[
   \left| P\left( \sigma^Y_0 < 1 \right) - P\left(
       \sigma^Y_{\delta_3} < 1- \delta_1\right) \right| \le P\left(
     \sigma^Y_0 \ge 1, \ \sigma^Y_{\delta_3} < 1- \delta_1 \right) +
   P\left( \sigma^Y_{\delta_3} \in (1 - \delta_1,1) \right)
 \]
 (and similarly with $Y$ replaced by $m^{-1}Z^m$).  Since by Theorem
 \ref{da0.25}
 \[\lim_{m\to\infty}P\left( \sigma^{Z^m}_{m\delta_3}< m(1- \delta_1)\right)=
   P\left( \sigma^Y_{\delta_3} < 1- \delta_1\right),\] it will suffice
 to show that the two terms on the right are bounded appropriately for
 process $Y$ and for $Z^m$ for $m$ large provided $\delta_1 $ and then
 $\delta_3 $ are well chosen.  For the second term for $Y$ we first
 choose $\delta_1 $ so small that
 $P\left(\sigma^Y_0 \in (1 - \delta_1,1+ \delta_1) \right)<
 \epsilon^3/ 10$ and then choose $\delta_3 $ so small that
 $P(\sigma^Y_0 < \delta_1 \mid Y(0) = \delta_3) \ge 1/2 $. The strong
 Markov property then gives the bound
 \begin{multline}
   \label{half}
   P\left(\sigma^Y_0 \in (1 - \delta_1,1+ \delta_1) \right)\ge 
   P\left(\sigma^Y_0 \in (1 - \delta_1,1+
     \delta_1)\mid\sigma^Y_{\delta_3} \in(1- \delta_1,1) \right)
   P\left( \sigma^Y_{\delta_3} \in (1 - \delta_1,1) \right)\\\ge
   P\left(\sigma^Y_0<\delta_1\mid Y(0) = \delta_3\right) P\left(
     \sigma^Y_{\delta_3} \in (1 - \delta_1,1) \right)\ge \frac12
   P\left( \sigma^Y_{\delta_3} \in (1 - \delta_1,1)\right),
 \end{multline}
 from which we conclude that
$P\left( \sigma^Y_{\delta_3} \in (1 - \delta_1,1)
\right)<\epsilon^3/5$. This bound applies for $Z^m$ when $m$ is large
by Theorem \ref{da0.25}.

For the first term we treat $Z^m $ as the argument for $Y$ is similar
but simpler.  Decreasing $ \delta_3$ if necessary we ensure 
that $( \delta_3, \delta_2, \epsilon ^3/5 )$ are as
$(\delta ', \delta, \epsilon )$ for Lemma \ref{dead}. Then
 \begin{multline*}
   \{  \sigma^{Z^m}_{0} \ge m,\ \sigma^{Z^m}_{\delta_3 m} \le (1- \delta_1)m
  \}\\\subseteq \{ \sigma^{Z^m}_{\delta_3 m} \le (1- \delta_1)m, \, Z^m_{\ceil{(1- \delta_1)m}} \ge \delta_2 m\}\cup\{Z^m_{\ceil{(1- \delta_1)m}}  \le
  \delta_2 m,\ \sigma^{Z^m}_{0} \ge m \}.
 \end{multline*}
 The probability of the first set is less than $\epsilon ^3 / 5$
 for $m$ large by Theorem \ref{dead} while of the last is less than $\epsilon ^3$ by our grounding hypothesis
 for the cookie environments.  This latter bound can be reduced
 arbitrarily for $Y$.
\end{proof}
The following lemma expresses a simple coupling result which
leads to Corollary~\ref{isBESQ2} below.
\begin{lemma} \label{cupBESQ} For every $\epsilon > 0$ there exists a
  $\delta ' > 0 $ such that a BESQ$^2$ process $Y$ beginning at
  space-time point $(y,t) \in [0, \delta']^2$ and a BESQ$^2$ process
  $Z$ beginning at $(0,0)$ con be coupled
  together so that with probability at least
  $1 - \epsilon^ 3 /5$ there exists a $\sigma \in(t,\epsilon^5)$
  such that
  \begin{enumerate}[(i)]
  \item $Y(s) = Z(s)$ for all $s\in[\sigma,\infty)$;
  \item $(\sup _{t\le s \le \sigma} Y(s))\vee(\sup _{s \le \sigma} Z(s)) \le \epsilon^5$.
  \end{enumerate}

\end{lemma} 
The next lemma provides a basic coupling of two BESQ$^2$
processes. Lemma~\ref{cupBESQ} follows from it by a simple scaling
argument in the same way that Corollary \ref{BMPEcoup1} follows from
Lemma \ref{BMPEcoup}.  For the coupling we will make the processes
independent until the first time that they meet and then equal
thereafter.
\begin{lemma} \label{cupBESQ1} Let $Y$ and $Z$ be independent
    BESQ$^2$ processes beginning at space-time points
    $(y_0,t_0)\in [0,1]^2$ and $(z_0,s_0)\in [0,1]^2$ respectively and
    $\sigma = \inf \{s > s_0 \vee t_0: Z(s) = Y(s) \mbox{ or } Z(s) \vee
    Y(s) = 2\} \wedge
    2$. Then
    \[c:=\inf_{(y_0,t_0,z_0,s_0)\in[0,1]^4}P(\sigma=\inf \{ s > s_0
      \vee t_0: Z(s) = Y(s) \})>0.\] 
% There exists a constant $c > 0$ such
%   that given two BESQ$^2$ processes $Z$ and $Y$, beginning at
%   space-time points $(z_0,s_0)\in [0,1]^2$ and $(y_0,t_0)\in [0,1]^2$
%   respectively, they can be coupled together so that
%   $\sigma = \inf \{ s > s_0 \vee t_0: Z(s) = Y(s) \mbox{ or } Z(s) \vee
%   Y(s) = 2\} \wedge 2$ is a stopping time for the coupled process and
%   \blu{\sout{such that}} with probability at least $c$ we have
%   $ \sigma = \inf \{ s > s_0 \vee t_0: Z(s) = Y(s) \} $.
\end{lemma}  
\begin{proof}
  We suppose without loss of generality that $s_0 < t_0$ and let
    \begin{align*}
      A_1&=\left\{\max_{s_0\le u\le 3/2}Z(u)<3/2,\,Z(3/2)>4/3\right\};\\
      A_2&=\left\{\max_{t_0\le u\le 2}Z(u)<3/2,\, \max_{t_0\le u\le
          3/2}Y(u)<4/3,\,\max_{3/2\le u\le
          2}Y(u)<2,\,Y(2)>5/3\right\}.
    \end{align*}
    It is clear that $P(A_1)$ and $P(A_2\mid A_1)$ are bounded away
    from zero uniformly over $(y_0,t_0,z_0,s_0)\in[0,1]^4$.  Noticing
    that
    $A_1\cap A_2\subseteq\{\sigma=\inf \{ s > s_0 \vee t_0: Z(s) =
    Y(s) \}\} $ completes the proof.
%     For the coupling, we will make the processes independent.  Given
%     this it is enough that independent processes touch with
%     probability $c$ before time $2$ and before attaining value $2$.
%     We first suppose without loss of generality that $s_0 < t_0 $ and
%     that event $A_1 $ occurs: process $\{ Z(u) \}_{s_0 \le u \le 5/4}$
%     is always below $5/4$ and $Z_{5/4 } > 11/10$. It is clear the
%     probability of $A_1$ is uniformly bounded away from zero.  It is
%     immediate that conditional on $A_1$ event $A_2$ also has
%     probability bounded away from $0$ uniformly in $(y_0,t_0)$ where
%     $A_2$ is the intersection of
% \begin{itemize}
% \item
% On $[t_0,2], Z(s) < 3/2$
% \item
% On $[t_0,5/4], Y(s) < 11/10$ and on $[5/4,2], Y(s) <2$ and  $Y_2 > 3/2$.
% \end{itemize}
% %Clearly on the event $A_2$ the processes $Y_\cdot$ and $Z_\cdot$ intersect prior to time $2$ and before attaining value $2$. 
% This is enough to complete the proof of the lemma.
  \end{proof}

The next statement is a consequence of \cite[Lemma 6.3]{kpERWMCS}.
\begin{lemma} \label{overshoot} Let $V^+_0=0$. Then for every fixed
    $\delta>0$ uniformly over first cookie environments
\[
\lim_{m\to\infty}P(\tau^{V^+}_{\delta m} \le \delta m,\ V_{\tau^{V^+}_{\delta m} } \ge \delta m + m^{2/3}) =0.
\]
\end{lemma} 
%\begin{proof}
%The event $\{T \le \delta_1 m, V_T \ge \delta m + m ^ {2/3}\}$ is contained in $\cup_{k=1}^{\fl {\delta_1 m} } \{ V_k \ge \delta m + m ^ {2/3}, V_{k-1} < \delta m\}$.  So $P(T \le \delta_1 m, V_T \ge \delta m + m ^ {2/3}) \le \sum_{k=1}^{\fl {\delta_1 m} } P( V_k \ge \delta m + m ^ {2/3}, V_{k-1} \le {\fl {\delta_1 m} })$
%which by monotonicity of the $V^+ $ process is bounded above by $ \delta_1 m P(V_1\ge  \delta m + m ^ {2/3  } \vert V_0 ={\fl {\delta_1 m} })$.
%By Lemma A.1 of \cite{KP} this is bounded by $ \delta_1 m C e^{-c' m^{1/3}/ \delta }$.
%\end{proof}
The following corollary is immediate given Theorem \ref{da0.25},
Lemma~\ref{cupBESQ}, Lemma \ref{overshoot} and the fact that BESQ$^2$ processes do not return
to zero.
    
\begin{cor} \label{isBESQ2}
Given $ \epsilon > 0$, parameters $\delta_1,\delta_2 > 0$ can be chosen so small that for any sequence of first cookie environments on $\bint{0,m}$ which are 
\begin{enumerate}[(i)]
\item $(m^{1/4},\nu/2 - 1)$-good,
\item $\delta_2 m $-lifting from the left on $\bint{0, \delta_1 m}$
\end{enumerate}
and any fixed $s \in [\epsilon ^ 5, 1]$, every distributional limit
point of $\{m^{-1}V^+_{\fl{sm}}\}_{m\ge 1}$  with $V^+_0=0$ evolving in this
environment must be within $3\epsilon^3/2 $ (in total
variation distance) of the law of $Y(s)$ where $Y$ is a solution of \eqref{daa} with $D\equiv \nu/2$ and $Y(0)=0$, i.e.\ a $\frac{\nu}{4}$BESQ$^2$ process.
\end{cor}

\begin{prop} \label{BESQ2int} Denote by $\{Z^m_k\}_{0 \le k \le 2m} $
  a concatenation of a $V^+$ process on $\bint{0,m}$ starting at $0$
  and a $U^+$ process on $\bint{m,2m}$.  Given
  $\epsilon\in(0,\epsilon_0)$ for some $\epsilon_0$ fixed small,
  parameters $ \delta_1,\delta_2 > 0$ can be chosen so small
  that for any sequence of first cookie environments on $\bint{0,2m}$
  which are
\begin{enumerate}[(i)]
\item  $(m^{1/4},\nu/2 - 1)$-good on $\bint{0,m} $ and $(m^{1/4},0)$-good on $\bint{m,2m}$;
\item $\delta_2 m $-lifting from the left on $\bint{0, \delta_1 m}$ and
  $\delta_2 m $-grounding from the left on $\bint{2m- \delta_1 m,2m}$,
\end{enumerate}
every limit point of
$ m^{-2} \sum_{k=0}^{2m-1} Z^m_k \ind{Z^m_{2m-1}= 0}$ must be of the form
\[
\int K(z,\cdot) \lambda_1 (dz)
\]
where $K(\cdot,\cdot)$ is a probability kernel satisfying
$K(z, [z- \epsilon ^8, z+\epsilon ^ 8]^c ) =0$ for all $z$ and
$\lambda_1$ is a probability measure within $ 8\epsilon^3$ (in total
variation distance) from the law  that is one half
$\delta_0$ plus one half the law of $\frac{\nu}{4}$ times the time
for the standard Brownian motion to exit $(-1,1)$.
\end{prop}

\begin{proof} 
  Throughout the proof $Y$ will be a solution of \eqref{daa} with
  $D(t) = \frac{\nu}{2} \ind{t<1}$ with initial condition $Y(0) = 0$.
  The previous results in this section will allow us to approximate
  certain probabilities for the process $Z^m$ in terms of
  corresponding probabilities involving the process $Y$. Also, since
  $Y$ is $\frac{\nu}{4}$ times a concatenation of a standard BESQ$^2$
  process and BESQ$^0$ process, it follows easily from the Ray-Knight
  Theorems that $\int_0^2 Y(s) \, ds \ind{\sigma_{1,0} < 2}$ has the
  law that is one half $\delta_0$ plus one half the law of
  $\frac{\nu}{4}$ times the time for the standard Brownian motion to
  exit $(-1,1)$.

  We fix a small $\delta_3>0$ and write
  $\Sigma^m:=\left( \sum_{k=0}^{2m-1} Z^m_k \right) \ind{Z^m_{2m-1}=0}
  $ as $\Sigma^m_a + \Sigma^m_b + \Sigma^m_c$ where
  \[\Sigma^m_a = \sum_{k=0}^{\epsilon^5 m} Z^m_k \ind{
    \sigma_{m,0} < 2m},\ \ 
  \Sigma^m_b = \sum_{k=\epsilon^5 m + 1}^{\sigma_{m, m \delta_3 }}
    Z^m_k \ind{ \sigma_{m,0} < 2m},\ \ 
  \Sigma^m_c = \sum_{k=\sigma_{m, m \delta_3 }+1}^{2m-1} Z^m_k
  \ind{ \sigma_{m,0} < 2m}
  .\]  

{\em Step 1.} We claim that if 
    $\delta_3 , \delta_1$ are fixed sufficiently
    small, then outside of probability $ \epsilon^3/4
    $, \begin{enumerate}[(i)]
    \item $\Sigma^m_a/m^2$ and
      $\Sigma^m_c/m^2$ are less than $\epsilon^8/4$ for all $m$
      large.
    \item
      $ \int_0^{\epsilon^5} Y(s) ds \ind{\sigma_{1,0} < 2}$ and
      $ \int_{\sigma_{1, \delta_3}}^{2} Y(s) ds \ind{\sigma_{1,0} <
        2}$ are less than $\epsilon^8/4$ for all $m$ large.
 \end{enumerate} We first consider (i).  The bound for
    $\Sigma^m_a $ is easily seen, since by monotonicity in the initial number of particles we have for all sufficiently small $\epsilon$ that
\[
 P(\Sigma^m_a > (\epsilon^8/4) m^2) \leq P( \tau^{Z^m}_{\epsilon^4 m} \le \epsilon^5 m\mid Z^m_0=0) \leq P( \tau^{Z^m}_{\epsilon ^ 4 m} \le \epsilon ^5 m\mid Z^m_0=\fl{\epsilon^5m}), 
\]
and it follows from Theorem \ref{da0.25} that the last probability is
at most $\epsilon^3/4$ for $m$ large enough.  Similarly, if the the
event $\{\Sigma^m_c\ge (\epsilon^8/4)m^2 \}$ occurs then after time
$\sigma_{m,\delta_3 m}$ the process $Z^m$ goes above $(\epsilon^8/4)m$
before time $2m$, and by Lemma \ref{dead} the probability of this is
less than $\epsilon^3/4$ if $\delta_3$ is chosen sufficiently small
(depending on $\epsilon$).  This finishes the proof of (i). The proof
of the bounds for (ii) is similar (but simpler since we do not need to
go through the diffusion approximation steps).

{\em Step 2.} We claim that the total variation distance between any
limit point of $\Sigma^m_b/m^2$ and
$ \int_{\epsilon^5}^{\sigma_{1, \delta_3}} Y(s) ds \ind{\sigma_{1,0} <
  2}$ is less than $7\epsilon^3/2$.  We begin by introducing some
notation. Let $\Sigma_a, \Sigma_b$, and $\Sigma_c$ be the analogs of
$\Sigma^m_a$, $\Sigma^m_b$, and $\Sigma^m_c$, respectively for the
process $Y$ in place of $Z^m$, that is
\[
 \Sigma_a = \int_0^{\epsilon^5} Y(s) \, ds \ind{\sigma_{1,0}^Y < 2}, 
\quad 
\Sigma_b = \int_{\epsilon^5}^{\sigma_{1,\delta_3}^Y} Y(s) \, ds \ind{\sigma_{1,0}^Y < 2}, 
\quad\text{and}\quad
\Sigma_c = \int_{\sigma_{1,\delta_3}^Y}^2 Y(s) \, ds \ind{\sigma_{1,0}^Y < 2}. 
\]
Also, let
$\tilde{\Sigma}_b^m = \left( \sum_{k=\epsilon^5 m + 1}^{\sigma_{m, m
      \delta_3 }} Z^m_k \right) \ind{ \sigma_{m,\delta_3 m} < 2m}$ and
$\tilde{\Sigma}_b = % \tilde{\Sigma}_b(Y) =
\int_{\epsilon^5}^{\sigma_{1,\delta_3}^Y} Y(s) \, ds
\ind{\sigma_{1,\delta_3}^Y < 2}$.

Let $V$ be a weak subsequential limit of $\Sigma^m_b/m^2$. We can then
take a further subsequence (which, for ease of notation, we will keep denoting by $m\in\N$)  along which
\begin{itemize}
 \item $(\Sigma^m_b/m^2, \tilde{\Sigma}^m_b/m^2)$ converges in distribution to a random vector $(V,\tilde{V})$,
 \item $Z_{\epsilon^5 m}^m/m$ converges in distribution to a random
   variable $\zeta_1$,
 \item and $Z_m^m/m$ converges in distribution to a random
   variable $\zeta_2$.
\end{itemize}
Note that it follows from Corollary \ref{isBESQ2} that
$d_{TV}(\zeta_1,\, Y(\epsilon^5)) < 3\epsilon^3/2$ and
$d_{TV}(\zeta_2, \, Y(1)) < 3 \epsilon^3/2$. %  (To make the notation
  % easier, for the time being we will just write $m$ instead of $m_k$
  % though all limits below will be along this subsequence.)
Then, using Theorem \ref{da0.25} and the fact that
$Z_{\epsilon^5 m}^m/m \Rightarrow \zeta_1$ %along the same subsequence
we can conclude that the distribution of $\tilde{V}$ is given by
\[
 P( \tilde{V} \in A ) = \int P\left( \tilde{\Sigma}_b \in A \mid Y(\epsilon^5) = z \right) P(\zeta_1 \in dz).
\]
From this it follows that
$d_{TV}\left( \tilde{V},\, \tilde{\Sigma}_b \right) \leq d_{TV}(
\zeta_1,\, Y(\epsilon^5) ) \leq 3 \epsilon^3/2$.  Next,
  note that
  \[d_{TV}\left( \Sigma_b, \, \tilde{\Sigma}_b \right) \leq P(\Sigma_b
    \neq \tilde{\Sigma}_b) \le P(\sigma^Y_{1,\delta_3} < 2-\delta_1,
    \, \sigma_{1,0}^Y > 2) + P(\sigma^Y_{1,\delta_3} \in
    (2-\delta_1,2) ).\] We first choose $\delta_1$ so small that
    $\sup_{y>0} P( \sigma_{1,0}^Y \in (2-\delta_1,2+\delta_1) \mid
    Y(1) = y) < \epsilon^3/10$.\footnote{Indeed, the function
        $f(y,\delta):=P( \sigma_{1,0}^Y \in (2-\delta,2+\delta) \mid
        Y(1) = y)$ is continuous on $[0,\infty)\times[0,1/2]$ and
        $f(y,\delta)\le f(y,1/2)\to 0$ as $y\to \infty$. This implies
        that there is an $L>0$ such that
        $\sup_{y>0}f(y,\delta)\le(\epsilon^3/15)\vee
        \sup_{y\in[0,L]}f(y,\delta)$ for all $\delta\in[0,1/2]$. The
        claimed bound now follows from the uniform continuity of $f$
        on a compact set $[0,L]\times[0,1/2]$ and the fact that
        $f(y,0)\equiv 0$.}  Next we choose $\delta_3$ depending on
    $\delta_1$ so that
    \[P(\sigma^Y_{1,\delta_3} < 2-\delta_1,
    \, \sigma_{1,0}^Y > 2)\le P(\sigma^Y_{1,0}>\delta_1\mid
    Y(1)=\delta_3)<\epsilon^3/10.\] Then by the same
    calculation as in \eqref{half} for this choice of
    $\delta_1,\delta_3$ we get that
\begin{equation}\label{unif-sd3}
 \sup_{y>0} P(\sigma^Y_{1,\delta_3} \in (2-\delta_1,2) \mid Y(1) = y) < \frac{\epsilon^3}{5}, 
\end{equation}
% \blu{(To obtain the uniformity in $y$, we note that the probability is bounded by $P(\sigma^Y_{1,\delta_3}<2 \mid Y(1) = y)$ which tends to zero as $y$ becomes large, and that by additivity of BESQ$^0$ processes  $|P(\sigma^Y_{1,\delta_3} \in (2-\delta_1,2) \mid Y(1) = y) -  P(\sigma^Y_{1,\delta_3} \in (2-\delta_1,2) \mid Y(1) = y+c) | \le P(\sigma^Y_{1,\delta_3}>3/2 \mid Y(1) = c)$.)}
and, therefore, $d_{TV}\left( \Sigma_b, \, \tilde{\Sigma}_b \right) < 2\epsilon^3/5$. 
%\begin{equation}\label{dTV-RtildeR}
%d_{TV}\left( \Sigma_b, \, \tilde{\Sigma}_b \right) 
%%%\leq P(\sigma^Y_{1,\delta_3} < 2 \leq \sigma_{1,0}^Y ) 
%< \frac{2\epsilon^3}{5}. 
%\end{equation}
In a similar manner, 
\begin{align*}
 P( \Sigma^m_b \neq \tilde{\Sigma}^m_b ) &= P\left( \sigma_{m,\delta_3 m}^{Z^m} < 2m \leq \sigma_{m,0}^{Z^m} \right)\\
&\leq P\left( \sigma_{m,\delta_3 m}^{Z^m} \leq (2-\delta_1)m, \, \sigma_{m,0}^{Z^m}  \geq 2m \right) + P\left( \sigma_{m,\delta_3 m}^{Z^m} \in ((2-\delta_1)m,2m) \right).
\end{align*}
By the argument at the end of the proof of Lemma \ref{thetruth} we can
bound the first probability above by $6\epsilon^3/5$ for $\delta_3$
small and $m$ large enough, while Theorem \ref{da0.25} together with
\eqref{unif-sd3} imply that the second probability can also be bounded
above by $\epsilon^3/5$ as $m\to \infty$.  Then using the joint
convergence
$(\Sigma^m_b/m^2, \tilde{\Sigma}^m_b/m^2) \Rightarrow
(V,\tilde{V})$, we get that
\begin{equation}\label{dTV-VV1} 
 d_{TV}(V,\tilde{V}) \leq P(V \neq \tilde{V}) \leq \liminf_{m\to\infty} P( \Sigma^m_b \neq \tilde{\Sigma}^m_b ) \leq \frac{7\epsilon^3}{5}. 
\end{equation}
Combining the above estimates we conclude that 
\[
 d_{TV}( V,\, \Sigma_b) \leq d_{TV}(V,\, \tilde{V}) + d_{TV}(\tilde{V}, \tilde{\Sigma}_b )+ d_{TV}(\tilde{\Sigma}_b,\, \Sigma_b ) < \frac{7 \epsilon^3}{2}.
\]

{\em Step 3.}  Now any weak limit of $\{\Sigma^m\}_{m\ge 1}$ can be
written as $U+V$ where $(U,V)$ is a weak limit (possibly along a
further subsequence) of
$\{((\Sigma^m_a + \Sigma^m_c)/m^2, \Sigma^m _b/m^2)\}_{m\ge 1}$.
Moreover, as we have shown in Steps 1 and 2,
\begin{itemize}
\item $V$ has law within
  $7\epsilon^3/2 $ of the law of
  $\int_{\epsilon^5}^{\sigma_{1, \delta_3}} Y(s) ds
  \ind{\sigma_{1,0} < 2}$ in total variation distance.
\item $U\le \epsilon^8/2$ outside of probability $\epsilon^3/2$.
\end{itemize}
By Step 1(ii) and Step 2, we can adjoin a
positive variable $U^\prime $ so that $U^\prime + V $ has law within
$7\epsilon^3/2$ in total variation distance of that of
$ \int_{0}^{2} Y(s) ds \ind{\sigma_{1,0} < 2}$ and
$U' \le \epsilon^8/2$ outside of probability
$ 4\epsilon^3$.  Indeed, let $(U',V)$ have joint
  distribution defined by the regular conditional probability
\[
 P(U' \in A \mid V = x) = 
P\left(\Sigma_a + \Sigma_c \in A \mid \Sigma_b = x \right). 
\]
%\[
% P(U' \in A \mid V = x) = 
%P\left( \left( \int_0^{\epsilon^5} Y(s)\, ds + \int_{\sigma_{1,\delta_3}}^2 Y(s) \, ds \right) \ind{\sigma_{1,0} < 2} \in A \, \biggl| \, \int_{\epsilon^5}^{\sigma_{1,\delta_3}} Y(s)\, ds\ind{\sigma_{1,0} < 2} = x \right). 
%\]
This implies first of all that
\[
 d_{TV}\left( U'+V, \int_0^2 Y(s)\, ds\ind{\sigma_{1,0} < 2} \right) 
\leq d_{TV}\left( V, \Sigma_b \right) \leq \frac{7}{2}\,\epsilon^3, 
\]
and secondly that 
\[
 \left| P\left( U'> \frac{\epsilon^8}{2} \right) - P\left( \Sigma_a + \Sigma_c > \frac{\epsilon^8}{2} \right) \right| \leq \frac{7}{2}\,\epsilon^3
\]
and then Step 1(ii) implies that
$P(U'> \epsilon^8/2 ) \leq 7\epsilon^3/2 + \epsilon^3/2 =
4\epsilon^3$.

From this we see that we can write the limiting distribution of
$\{\Sigma^m\}_{m\ge 1}$ along this subsequence as that of $(V +U')+U - U'$.
We now introduce the following kernels or sub-kernels
\begin{itemize} 
\item
$L$ denotes the regular conditional probability of $U-U^\prime $ given $V+U^\prime $;
\item $L^\epsilon $ denotes the sub-kernel
  $L^\epsilon (x, D) = L (x, D\cap [-\epsilon ^8, \epsilon
  ^8])$;
\item $H(x,D) = L(x,D-x)$, and $H^\epsilon (x,D) = L^\epsilon(x,D-x)$.
\end{itemize} 
Thus, if $\gamma$ denotes the law of $V+U'$, then
$\int H(x,\cdot)\gamma (dx)$ is the law of $U+V$ and
\[
P(U+V\in A)=\int H^\epsilon(x,A)\gamma (dx) +  \int (H-H^\epsilon)(x,A)\gamma (dx).
\]
The second term is a measure which we denote by $\gamma_b$. Note that
$\gamma_b(\R) = P(|U-U'| > \epsilon^8) < 9\epsilon^3/2$.  We then take
$\lambda_1 $ to be the measure $\gamma_a + \gamma_b$, where
$ \gamma_a $ is the measure that is absolutely continuous with respect
to $\gamma $ with
$\frac{d \gamma_a}{d \gamma}(x) = H^\epsilon(x,\R)$.  Direct
  calculation then yields that
\[
  \lambda_1(A) = P\left( U'+V \in A, \, |U-U'| \leq \epsilon^8 \right)
  + P\left( U+V \in A, \, |U-U'| > \epsilon^8 \right),
\]
from which one obtains both that $\lambda_1$ is a probability measure
and that the total variation distance between $\lambda_1$ and $\gamma$
is at most $P(|U-U'| > \epsilon^8) \leq 9\epsilon^3/2$.  Since the
total variation distance between $\gamma$ and
$\int_{0}^{2} Y(s) ds \ind{\sigma_{1,0} < 2}$ is at most
$7\epsilon^3/2$ we then get that the total variation distance between
$\lambda_1$ and $\int_{0}^{2} Y(s) ds \ind{\sigma_{1,0} < 2}$ is at
most $8\epsilon^3$.  Finally, the proof is completed by letting $K$ be
the kernel
\[K(x,A) = \frac{H^\epsilon(x,A)}{H^\epsilon(x,\R)}
\frac{d\gamma_a}{d\lambda_1}(x) + \ind{x \in A}
\frac{d\gamma_b}{d\lambda_1}(x),\] since one can check easily that
$\int H(x,\cdot)\gamma (dx ) = \int K(x,\cdot) \lambda_1(dx)$ and 
$K(x,[x-\epsilon^8,x+\epsilon^8]^c) = 0$.  
\end{proof}

\subsection{Concatenation lemma}

The walk can go from the bulk to a previously untouched territory and
then back to the bulk. For this reason we need to consider BLPs in an
environment, which is possibly a concatenation of the one our walk
created in the bulk and the original i.i.d.\ environment. 
Lemma
\ref{hitpr} below says, loosely speaking, that if the environment
  is ``good'' then 
our mesoscopic walk moves right or left and
  ``chooses'' its extrema with probabilities close of those of the
  basic BMPE-walk as given in Corollary \ref{exitc}. This lemma addresses all situations: the very
first step of the mesoscopic walk, transitions between the bulk and
the boundary, and steps in the bulk.

Let $0\le \wu\le w = 1\le \wo\le b = 2 $ as well as small values
$0<\delta < \delta _ 1 < \epsilon ^ 3$ be fixed and
$\wu_m,w_m,\wo_m,b_m \in \Z / m $ converge to $\wu,w,\wo,b$
respectively as $m\to\infty$. We shall consider a BLP
$\{ Z^m_k \}_{0\le k\le b_mm}$ which evolves as a $V^+$ process
on $\bint{0,w_mm}$ and as a $U^+ $ process afterwards. In
  particular, $Z^m$ has an immigrant in each generation up to
  $\fl{w_mm}$ and, thus, $0$ becomes an absorbing state only after
  $\fl{w_mm}$.

Next we define which environments are considered ``good'' by
  imposing three conditions. In Section \ref{sec:ei} we shall show that
  they are satisfied with high probability.
\begin{asm}\label{asm:concat}
The following properties hold for all sufficiently large
  $m$.
 \begin{enumerate}[(i)]
 \item Either $\wu\ge\epsilon^3$ and the first cookies on
   $\bint{0,\wu_mm}$ are i.i.d.\ with marginal $\eta $ or
   $\wu= \wu_m=0$ and the first cookie environment on
   $\bint{0,\delta_1m}$ is $ \delta_2 m$-lifting from the left.
 \item The first cookie environment on $\bint{\wu_mm,w_mm}$ is $(m^{1/4},\nu/2-1)$-good and the first cookie environment on $\bint{w_mm,\wo_mm}$ is  $(m^{1/4},0)$-good.
 \item Either $\wo \leq 2 - \epsilon ^ 3 $ and the first cookies on
   $\bint{\wo_mm,b_mm}$ are i.i.d.\ with marginal $\eta $ or
   $\wo = 2 $ and the first cookie environment on
   $\bint{(b_m- \delta_1) m,b_mm}$ is $\delta_2 m$-grounding from the
   left.
 \end{enumerate} 
\end{asm}

To clarify the meaning of the above conditions, let us mention that the
first step of the mesoscopic walk corresponds to
$\wu=w=\wo=1$.  The case when $\wu =0$ and $\wo = 2 $
corresponds to steps in the bulk. The other cases are when the step is
at the boundary, and then the interval $(\wu,\wo)$ represents the bulk
region on the macroscopic scale.

\begin{lemma}[Concatenation lemma]\label{hitpr}
  There is a constant $K$ such that for every $\epsilon>0$ and all
  first cookie environments satisfying Assumption \ref{asm:concat} with
  sufficiently small $0< \delta_2 <\delta_1<\epsilon^3$, the following
  statements hold.
  \begin{enumerate}[(i)]
    \item For all $m$ sufficiently large, $|P(\sigma_{w_mm,0}\le b_mm)- P(\tau_0(\wu ,w,\wo  )\le \tau_b(\wu ,w,\wo ))|\le K\epsilon ^ 3$.
    \item Let $\wo\le 2-\epsilon^3$, intervals
        $J_\ell, \ 1\le \ell \le L$, be as in \eqref{Jell}, and choose
        $\ell_0$ (which might depend on $m$) so that
        $\wo_m \in J_{\ell_0}+1$.  Then for all sufficiently large
        $m$
   \[
   \sum_{\ell > \ell_0 + 1} \vert P(m^{-1}\sigma_{w_mm,0} \in  J_\ell+1)- P(S(\tau_0(\wu,w,\wo)) \in  J_\ell+1)  \vert < K \epsilon ^ 3. 
   \] 
  \end{enumerate}
\end{lemma}

\begin{rem}\label{concrem}
Since
$\{\sigma_{w_mm,0}\le b_mm\}=\{\exists
\ell\in\bint{\ell_0,L}:\,m^{-1}\sigma_{w_mm,0} ) \vee \wo_m\in
J_\ell+1\}$% and
, (i)
and (ii) imply that for some (possibly different) constant $K$ and
all sufficiently large $m$
\begin{equation*}
\vert P( (m^{-1}\sigma_{w_mm,0} ) \vee \wo_m \in (J_{\ell_0} \cup J_{\ell_0+1})+1)- P( S(\tau_0(\wu,w,\wo)) \in (J_{\ell_0} \cup J_{\ell_0+1})
 +1 ) \vert
\ \leq \ K \epsilon ^3 .
\end{equation*}
Part (ii) deals with the maximum $S$ for a BMPE when
$\wo \le 2 - \epsilon ^ 3$. The argument given applies equally to
the minimum $I$ when $\wu \ge \epsilon ^ 3 $.
\end{rem}

\begin{proof}
We begin with part (i) and let $\{ Y(t) \}_{0 \le t \le 2}$ be
a solution of \eqref{daa} with
\[D(t)=\frac{\nu}{2}((1-\theta^-)\ind{0\le t<\wu}+\ind{\wu\le
    t<1}+\theta^+\ind{\wo\le t\le 2})\ \text{ and $Y(0)=0$.}\] The
process $\{ Y(t) \}_{0 \le t \le 2}$ is a constant multiple of
$\{ y(t) \}_{0\le t\le 2}$ from Proposition~\ref{RNbmpe}. Therefore,
by part (i) of Corollary~\ref{exitc} we have that
$P( \tau_0 (\wu,w,\wo ) < \tau_b (\wu,w,\wo))=P(\sigma^Y_{w,0}< b)$.
Let $g(x) := P(\sigma_{w,0}^Y < b \mid Y(w) = x) $. It is clear that
$g$ is continuous and
\begin{align*}
P( \tau_0 (\wu ,w,\wo ) < \tau_b (\wu ,w,\wo )) = E[g( Y(w))],
\\\shortintertext{while by Lemma~\ref{thetruth}}
\left| P( \sigma_{w_mm,0} \le b_m m \mid Z^m_{w_m m }) - g\left(m^{-1}Z^m_{w_m m }\right) \right| < 4 \epsilon^3 .
\end{align*}
So our proof for (i) comes down to showing that for $m$ large
\begin{equation}
\label{boils}
\left| E\left[g( Y(w) )\right] - E\left[ g(m^{-1}Z^m_{w_m m })\right] \right| < K \epsilon^3
\end{equation}
for a universal $K$. If $\wu =0$ then \eqref{boils} is an immediate consequence of Corollary~\ref{isBESQ2}. If $\wu \ge \epsilon^3$ then by Theorem~\ref{DA0}, 
\begin{equation*}
  \left\{ m^{-1}Z^m_{\fl{mt}}\right\}_{0 \leq t \leq \wu_m} \stackrel { J_1 } {\Longrightarrow}\quad \{Y(t)\}_{0\le t\le \wu}.
\end{equation*}
Moreover, by Theorem~\ref{da0.25} and the fact that a BESQ$^2$ process a.s.\ does not hit $0$,
\begin{equation*}
\left\{ m^{-1}Z^m_{\fl{mt}} \right\}_{\wu_m \leq t \leq w_m} \stackrel { J_1 } {\Longrightarrow} \quad \{Y(t)\}_{\wu\le t\le 1}
\end{equation*}
In particular, $m^{-1}Z^m_{w_m m }\Longrightarrow Y(1)$ and
  \eqref{boils} holds. Part (i) of the lemma is
  proven.

The proof of part (ii) splits into two cases according to whether
    $\wu\ge \epsilon^3$ or $\wu=0$. The first case is easier (though
  essentially the same) so we content ourselves with the second case.

  It follows from Corollary \ref{da0.25dead} and Corollary
  \ref{isBESQ2} that any limit point $\tilde{Y}$ of
  $ \{ m^{-1}Z^{m} _{\fl{mt}}\}_{w_m\le t\le 2}$ solves the same
  equation \eqref{daa} as $Y$ on time domain $[1, 2 ]$ and that the
  law of $\tilde{Y}(1)$ is within $3 \epsilon^3/2 $ of the law of
  $Y(1)$ in total variation norm. The closeness of the laws of
    $Y(1) $ and $\tilde Y (1) $ and the Markov property imply that
\begin{equation*}
\sum_{\ell > \ell_0 + 1} \big| P( \sigma^{\tilde Y } _{1, 0} \in J_\ell +1) - P(\sigma^Y_{1,0} \in J_\ell+1 ) \big| \leq \frac{3}{2}\,\epsilon^3.
\end{equation*}
A slight subtlety, arising from the weakness of the conclusion of
Corollary~\ref{da0.25dead} compared to Theorem \ref{DA0} is that we
cannot claim that
  $ m^{-1}\sigma _{w_mm, 0}\Rightarrow
  \sigma^{\tilde{Y}}_{1,0}$. However, given the power of Theorem
\ref{DA0} we can assert that
$ m^{-1}\sigma_{w_mm, 0} \vee \wo_m$ converges to
$\sigma^{\tilde{Y}}_{1,0}\vee \wo$. This and the fact that
  the law of $\sigma^{\tilde{Y}}_{1,0}\vee \wo$ has no atoms in
  $(\wo,2)$ permits us to conclude that for $ \delta_1 $ fixed
sufficiently small and all $m$ sufficiently large
\begin{align*}
\sum_{\ell\in\bint{\ell_0+2,L}} \big| P( m^{-1}\sigma _{w_mm, 0}\vee \wo_m \in J_\ell +1) - P(\sigma^Y_{1,0}\vee\wo &\in J_\ell+1) \big| \\
= \sum_{\ell\in\bint{\ell_0+2,L}}  \big| P( m^{-1}\sigma _{w_mm, 0} \in J_\ell +1) - P(\sigma^Y_{1,0} &\in J_\ell+1) \big| \leq 2\epsilon^3.
\end{align*}
Noting that for every
$\ell\in\bint{\ell_0 +1,L},\ P(\sigma^Y_{1,0} \in J_\ell+1) =
P(S(\tau_0(\wu,w,\wo)) \in J_\ell +1 )$ completes the proof.  
\end{proof}

\section{Environmental issues}\label{sec:ei}

The applicability of our tools from previous sections depends on
whether the environment is ``good'' in some way. Maintaining the
desired properties of the environment as the walk moves from one
mesoscopic site to another is crucial for our arguments. In this
section we shall prove some important properties of the cookie
environment modified by the walk. This will allow us to couple our
rescaled ``mesoscopic'' ERW with a modified BMPE-walk $\twe$ and
establish the desired functional limit theorem.

We note that for a fixed $\epsilon>0$ the scaling parameter $m$
  in Section~\ref{sec:tb} is roughly of order $\epsilon\sqrt{n}$. This is
  why $m^{1/4}$-goodness of the first cookie environment becomes
  $n^{1/8}$-goodness in this section.

% We shall start with a remark.  We have not yet got around to using it
% but it does help to create a correct picture.
% \begin{rem}
%   The measure\footnote{the distribution of the ``first
%     cookies'' in the stack.}
%   $\left(\otimes_{x<0} \pi^-\right)\otimes \mu\otimes
%   \left(\otimes_{x>0}\pi^+\right)$ on $(0,1)^{\Z}$ with the product
%   Borel $\sigma$-algebra is invariant under the dynamics of the ERW
%   with $X_0=0$.
% \end{rem}
\subsection{$n^{1/8}$-goodness of the environment.} 
\begin{lemma}\label{18good}
For $n\geq 1$ and $K<\infty$, let $A_{n,K}$ be the event that at every time $k$ until exiting the interval $\bint{-K\sqrt{n},K\sqrt{n}}$ the remaining first cookie environment on the interval $\bint{X_k,S_k}$ is $(n^{1/8},0)$-good 
and the remaining first cookie environment on the interval $\bint{I_k,X_k}$ is $(n^{1/8},\frac{\nu}{2}-1)$-good. 
If $\max\{\theta^+,\theta^-\}<1$, then $\lim_{n\to\infty} P(A_{n,K}) = 1$ for any $K<\infty$. 
\end{lemma}

%\begin{lemma}\label{18good}
% For any fixed $\epsilon>0$ and $T<\infty$, for all $n$ sufficiently large the probability that at every time $k\leq Tn$ 
%the remaining first cookie environment on the interval $[X_k,S_k]$ is $(n^{1/8},0)$-good 
%and the remaining first cookie environment on the interval $[I_k,X_k]$ is $(n^{1/8},\frac{\nu}{2}-1)$-good is at least %$1-\epsilon^3$  
%\end{lemma}

Before giving the proof of Lemma \ref{18good} we state the following simple corollary which follows from the fact that the walk doesn't exit the interval $[-k\epsilon\sqrt{n}, k\epsilon\sqrt{n}]$ before the stopping time $\tenk$. 
%$[-\epsilon^{-3/2} K \sqrt{n}, \epsilon^{-3/2}K\sqrt{n}]$ before the stopping time $T_{\fl{ \epsilon^{-2}K}}^{\epsilon,n}$. 

\begin{cor}\label{18good-cor}
 Let $\max\{ \theta^+,\theta^-\} < 1$. 
 For any fixed $\epsilon>0$ and $k,n\geq 1$ let $A^{\epsilon,n}_k$ be the event that at time $\tenk$ the remaining first cookie environment on the interval $\bint{X_{\tenk},S_{\tenk}}$ is $(n^{1/8},0)$-good 
and the remaining first cookie environment on the interval $\bint{I_{\tenk},X_{\tenk}}$ is $(n^{1/8},\frac{\nu}{2}-1)$-good. 
Then, $\lim_{n\to\infty} P(A^{\epsilon,n}_k) = 1$, for any $k\geq 1$. 
\end{cor}

\begin{rem}
 The intuition behind Lemma \ref{18good} is that after the walk has taken a large number of steps, the remaining first cookie environment of the sites to the right (resp.\ left) of the present location up to the running maximum (resp.\ running minimum) are approximately independent and distributed according to $\pi^+$ (resp.\ $\pi^-$). 
%and the fact that
%$\pi^+ \cdot \mathbf{r}^+ = 0$ and $\pi^- \cdot \mathbf{r}^+ = \frac{\nu}{2}-1 - \pi^-\cdot \mathbf{r}^- = \frac{\nu}{2}$ (see \cite[Corollary 3.5]{kpERWMCS}). 
\end{rem}

The proof of Lemma~\ref{18good} will rely on some
preliminary estimates regarding the BLP.

%\begin{lemma}[Corollary 3.2 and Lemma 3.4 in \cite{kpERWPCS}]\label{Kbound}
% If $\max\{\theta^+,\theta^-\} < 1$, then 
%\[
% \lim_{K\to\infty} \limsup_{n\to\infty} P\left( \max_{k\leq n} |X_k| > K \sqrt{n} \right) = 0, 
%\]
%and
%\[
% \lim_{K\to\infty} \limsup_{n\to\infty} P\left( \max_{x \in \Z} \mathcal{L}(n,x) > K \sqrt{n} \right) = 0.  
%\]
%
%\end{lemma}

\begin{lemma}[Lemma 3.6 in \cite{kpERWPCS}]\label{small1}
 If $\theta^+<1$, then for any $0<\alpha<\beta$ there exist constants $C,c>0$ such that 
\[
 \sup_{j\geq 0} P\left( \sum_{i=0}^{\sigma_0^{U^+}-1} \ind{U^+_i < m^\alpha} > m^\beta \, \biggl| \, U_0^+ = j \right) \leq C e^{-c m^{\beta-\alpha}}, \quad m\geq 1. 
\]
Similar statements hold for the BLPs $U^-$, $V^+$, and $V^-$ if the assumption $\theta^+<1$ is replaced by $\theta^-<1$, $\theta^->0$, and $\theta^+>0$, respectively. 
\end{lemma}
\begin{rem}
 Note that while \cite{kpERWPCS} was written for excited random walks in periodic cookie stacks, the proof of the above lemma in this paper relied only on some facts concerning the BLPs that were also proved for excited random walks with markovian cookie stacks in \cite{kpERWMCS}.
\end{rem}

Lemma \ref{small1} controls the time spent by a BLP below a certain level before reaching
level zero. While zero is an absorbing point for
  $U^{\pm}$, it is not absorbing for
  $V^{\pm}$, and we will at times need
to control the time spent by  these
  processes below some level on a fixed time
interval.  The following lemma accomplishes
  this. It is similar to Lemma 3.8 in \cite{kpERWPCS}, but the
statement here is more flexible for the applications we
need. Moreover, the proof below corrects an error in the proof of
Lemma 3.8 in \cite{kpERWPCS}.
\begin{lemma}\label{small2}
 If $\theta^- <1$, $\alpha \in (0,1-(\theta^-\vee 0))$ and $\beta \in ( (\theta^-\vee 0)+\alpha, 1 )$, then there exist constants $C,c,r>0$ such that 
\[
 \sup_{j\geq 0} P\left( \sum_{i\leq m} \ind{V^+_i < m^\alpha} > m^\beta \, \biggl| \, V_0^+ = j \right)
 \leq C e^{-c m^r}. 
\]
A similar statement holds for the process $V^-$ if $\theta^-$ is replaced everywhere above by $\theta^+$. 
\end{lemma}

\begin{proof}
  Since the probability in the statement of the lemma is
  non-decreasing in $j$, we need only to prove the inequality when
  $j=0$.  Also, we will assume that $\theta^- \in (0,1)$ since if
  $\theta^-\leq 0$ we can couple it to another BLP
  which has parameter $\theta^- \in (0,1)$ and which is always less
  than or equal to $V^+$ (see Lemma 5.1 in \cite{kpERWMCS}).

%Let $\gamma = \frac{\theta^-}{4}+\frac{\beta-\alpha}{2}$. 
Now, fix some $\gamma \in (\theta^-, \beta-\alpha)$ (note that this is
possible by the assumptions on $\alpha$ and $\beta$).  Then, if the
event $\{ \sum_{i\leq m} \ind{V^+_i < m^\alpha} > m^\beta \}$ occurs,
either
\begin{enumerate}
 \item the process $V^+$ returns to $0$ at least $\ceil{m^\gamma}$ times in the first $m$ steps of the Markov chain, 
 \item or in one of the first $\ceil{m^\gamma}$ excursions from 0 of
   the process $V^+$ it stays below $m^\alpha$ for at least
   $m^{\beta-\gamma}$ steps.
\end{enumerate}
The first of these events implies that each of the first
$\ceil{m^\gamma}$ excursions from 0 lasts
at most $m$ steps and is thus its probability is bounded above
by
\[
  \left( 1- P(\sigma_0^{V^+} > m \, | \, V_0^+ = 0)
  \right)^{\ceil{m^\gamma} } \leq (1-c m^{-\theta^-} )^{m^\gamma} \leq
  e^{-c m^{\gamma-\theta^-}},
\]
where the first inequality follows from known tail asymptotics for
$\sigma_0^{V^+}$ when $\theta^- > 0$; see \cite[Theorem
2.7]{kpERWMCS}.  On the other hand, by Lemma
  \ref{small1}, the probability of the second event does not exceed
\[
 m^\gamma P\left( \sum_{i<\sigma_0^{V^+}} \ind{V^+_i < m^\alpha} > m^{\beta-\gamma} \, \biggl| \, V_0^+ = 0 \right) \leq C m^{\gamma} e^{-c m^{\beta-\alpha-\gamma}}. 
\]
Choosing $r\in(0,(\beta-\alpha-\gamma)\wedge(\gamma-\theta^-))$
 we have that both events considered above have at least a stretched exponential decay in $m$. 
\end{proof}

We will also need the following lemma which gives control on the number of times any site can be visited before exiting a fixed interval. 

\begin{lemma}\label{maxL}
 If $\max\{ \theta^+,\theta^- \} < 1$, then 
\[
 \lim_{r\to \infty} \limsup_{m\to \infty} P\left( \max_{|x|\leq m} \mathcal{L}(\tau_m^X \wedge \sigma_{-m}^X, x) > r m \right) = 0. 
\]
\end{lemma}

\begin{proof}
Clearly it is enough to prove an upper bound on 
\[
 P\left( \max_{x \in [-m,0]} \mathcal{L}(\tau_m^X \wedge \sigma_{-m}^X, x) > r m \right) 
 \leq P\left( \max_{x \in [-m,0]} \mathcal{L}( \sigma_{-m}^X, x) > r m \right),
\]
as a similar argument will control the local time to the right of the origin. 
To this end, 
recall from \eqref{downsteps} that $\mathcal{E}^m_x$ is the number of steps right from $x$ before time $ \sigma_{-m}^X$
and note that
%note that if $\mathcal{E}^{(-n,0)}_x = |\{k\leq  \sigma_{-n}^X:\, X_k=x,\, X_{k+1}=x+1 \}|$ is the number of steps right from $x$ before time $ \sigma_{-n}^X$ then 
\[
 \mathcal{L}( \sigma_{-m}^X,x) = \mathcal{E}^{m}_x + \mathcal{E}^{m}_{x-1} + 1, \qquad \text{for all } -m<x\leq 0. 
\]
Since $(\mathcal{E}^{m}_{-m},\mathcal{E}^{m}_{-m+1},\ldots,\mathcal{E}^{m}_{-1},\mathcal{E}^{m}_{0} )$ has the same distribution as the BLP $(V^+_0,V^+_1,\ldots,V^+_{m-1},V^+_{m})$ started with $V^+_0=0$,
%(note that here we are using $\theta^+<1$, which ensures that $P(\sigma_{-n}^{X}<\infty)= 1$). 
then 
\[
 P\left( \max_{x \in [-m,0]} \mathcal{L}( \sigma_{-m}^X, x) > r m \right)
 \leq P\left( \max_{i\leq m} V^+_i > rm/3 \mid V^{+}_0 = 0 \right). 
\]
Finally, it follows from the diffusion approximation in Lemma \ref{DA0} that the probability on the right converges to $0$ as first $m\to \infty$ and then $r \to \infty$. 
\end{proof}

\begin{proof}[Proof of Lemma \ref{18good}]
We begin by introducing some new notation that will be used in this proof. 
For $x\in \Z$ and $m\geq 0$ let $\tau_{x,m}$ be the stopping time of the $(m+1)$-st visit of the ERW $X$ to location $x$. That is, $\tau_{x,m} = \inf\{k\geq 0: \sum_{i\leq k} \ind{X_i=x} = m+1\}$.  
Also, $y \in \Z$ let $\mathcal{E}^{(x,m)}_y$ and $\mathcal{D}^{(x,m)}_y$ be the number of steps right and left from $y$, respectively, prior to time $\tau_{x,m}$. That is, 
\[
 \mathcal{E}^{(x,m)}_y = \sum_{n=0}^{\tau_{x,m}-1} \ind{X_n = y, \, X_{n+1} = y+1}
\quad \text{and}\quad
\mathcal{D}^{(x,m)}_y = \sum_{n=0}^{\tau_{x,m}-1} \ind{X_n = y, \, X_{n+1} = y-1}. 
\]
In the proof below we will use the following facts concerning these the directed edge local times. 
First of all, we note that 
%if $x<y$ then $\mathcal{D}_y^{(x,m)} = \mathcal{E}_{y-1}^{(x,m)} + \ind{x<y\leq 0}$. 
\begin{equation}\label{DyEy}
 \mathcal{D}_y^{(x,m)} = \mathcal{E}_{y-1}^{(x,m)} + \ind{x<y\leq 0}
\quad \text{for } x < y. 
\end{equation}
%Indeed, if $x<y\leq 0$ then any path from 0 to $x$ crosses the edge $(y-1,y)$ from right to left one more time than it crosses the same edge from left to right, whereas if $y>x \vee 0$ then any path from $0$ to $x$ crosses the edge $(y-1,y)$ an equal number of times in each direction.
Secondly, the process $\{\mathcal{E}^{(x,m)}_y \}_{y\geq x}$ has the same distribution as a BLP or concatenation of BLPs. If $x\geq 0$ then this is a $U^+$ process using the cookie environment on $\llbracket x,\infty)$ but if $x < 0$ then it is a concatenation of a $V^+$ process using the cookie environment on $\bint{x,0}$ with a $U^+$ process using the cookie environment on $\llbracket 0,\infty)$ (see Section 2.2 of \cite{kpERWPCS} for more details on this connection with BLPs).

%For $x\in \Z$ and $m\geq 0$ let $\tau_{x,m}$ be the stopping time of the $(m+1)$-st visit of the ERW $X$ to location $x$. That is, $\tau_{x,m} = \inf\{k\geq 0: \sum_{i\leq k} \ind{X_i=x} = m+1\}$. 
%Then for any $n\geq 1$ 
Using the above notation, for any $n\geq 1$ and $K,K'<\infty$ let  $\tilde{A}_{K,K',n}$ be the event that at every (random) time $\tau_{x,m} $ with $|x|\leq K\sqrt{n}$ and $m\leq K'\sqrt{n}$ the remaining first cookie environment is $(n^{1/8},0)$-good on $\bint{X_{\tau_{x,m}},S_{\tau_{x,m}}\wedge K\sqrt{n}}$ and and $(n^{1/8},\frac{\nu}{2}-1)$-good on $\bint{I_{\tau_{x,m}}\vee -K\sqrt{n},X_{\tau_{x,m}}}$. 
Since 
\[
 P(A_{n,K}^c) \leq P(\tilde{A}_{K,K',n}^c) + P\left( \max_{|x|\leq K\sqrt{n}} \mathcal{L}( \tau_{K\sqrt{n}}^X \wedge \sigma_{-K\sqrt{n}}^X, x) > K' \sqrt{n} \right), 
\]
and since Lemma \ref{maxL} implies that the second term on the right can be made arbitrarily small for $n$ large by taking $K'$ large enough, it is enough to show that $\lim_{n\to\infty} P(\tilde{A}_{K,K',n}^c) = 0$ for all $K,K'<\infty$. 

Now, for the remainder of the proof, we'll only prove that the
remaining first cookie environments are $(n^{1/8},0)$-good to the
right of the current location using the directed edge local times
$\mathcal{E}^{(x,m)}_y$ and the corresponding BLPs $U^+$ and $V^+$.
The proof that the first cookie environments are
$(n^{1/8},\frac{\nu}{2}-1)$-good to the left of the current location
is similar using $\mathcal{D}^{(x,m)}_y$ and the BLPs $U^-$ and $V^-$.
For $K<\infty$, $|x|\leq K\sqrt{n}$, and $m,n\geq 1$ define
\[
  B^{(x,m)}_{K,n}
= \left\{ \left| \sum_{y \in J} r^+\left(R^y_{\mathcal{L}(\tau_{x,m},y)+1}  \right) \ind{y\leq S_{\tau_{x,m}} }  \right| > \frac{n^{1/8}}{\ln n}, \, \text{for some } J \subset \bint{x,K\sqrt{n}} \text{ with } |J| =  \fl{n^{1/8}} \right\}. 
\]
Then, 
\begin{equation}\label{notgood}
 P(\tilde{A}_{K,K',n}^c)% P\left( \begin{array}{cc} \text{remaining first cookie environment on $\bint{X_{\tau_{x,m}},S_{\tau_{x,m}}\wedge K\sqrt{n}}$} \\ \text{is not $(n^{1/8},0)$-good for some $|x|\leq K\sqrt{n}$, $m\leq K'\sqrt{n}$} \end{array} \right)
 \leq \sum_{\substack{|x|\leq K\sqrt{n}\\m\leq K'\sqrt{n}}} P\left(B^{(x,m)}_{K,n}\right).
\end{equation}
Thus it remains only to bound the probabilities $P( B^{(x,m)}_{K,n} )$. 
To this end, first let for $y\ge x$
\[
 r_y^{x,m} = \E\left[ r^+\left(R^y_{\mathcal{L}(\tau_{x,m},y)+1}\right) \, | \, \mathcal{G}_{y-1}^{x,m} \right], 
 \quad \text{where} \quad \mathcal{G}_{z}^{x,m} = \sigma(\mathcal{E}_u^{(x,m)}, \, R_j^u, \,  x \leq u \leq z \text{ and } j\geq 1 ).
\]
That is, $\mathcal{G}_{x-1}^{x,m}$ is the trivial $\sigma$-field and % $\mathcal{G}_{z}^{x,m} = \sigma(\mathcal{E}_y^{x,m}, \, R_j^y, \,  x \leq y \leq z \text{ and } j \geq 1 )$.
% That is,
  $\mathcal{G}_z^{x,m}$, $z\ge x$, contains all the information about
the number of steps right from sites in $\bint{x,z}$
and  all of the cookies in the stacks in
$\bint{x,z}$.  With this notation we have
\begin{align}
 P(B^{(x,m)}_{K,n}) 
%&\leq P\left( \exists J\subset[x,K\sqrt{n}], \, |J|\leq n^{1/8}: \, \left| \sum_{y \in J} \left\{ r^+\left(R^y_{\mathcal{L}(\tau_{x,m},y)+1}  \right) - r^{x,m}_y \right\} \ind{y\leq S_{\tau_{x,m}} }  \right| > \frac{n^{1/8}}{ \ln n} \right) \\
&\leq \sum_{\substack{J\subset\bint{x,K\sqrt{n}}\\ |J|= \fl{n^{1/8}}}} P\left(  \left| \sum_{y \in J} \left\{ r^+\left(R^y_{\mathcal{L}(\tau_{x,m},y)+1}  \right) - r^{x,m}_y \right\} \ind{y\leq S_{\tau_{x,m}} }  \right| > \frac{n^{1/8}}{ 2 \ln n} \right) \label{martpart} \\
&\qquad + P\left(  \exists J\subset\bint{x,K\sqrt{n}}, \, |J|=\fl{n^{1/8}}: \,\left| \sum_{y \in J} r^{x,m}_y \ind{y\leq S_{\tau_{x,m}} }  \right| > \frac{n^{1/8}}{2 \ln n} \right) \label{rxypart}
\end{align}
For the first term on the right, first note that \eqref{DyEy} implies
that
$\mathcal{L}(\tau_{x,m},y) = \mathcal{E}_{y-1}^{(x,m)} +
\mathcal{E}^{(x,m)}_y+ \ind{x<y\leq 0}$ so that the terms in braces
are $\mathcal{G}^{x,m}_y$-measurable.  Secondly, note that if
$y \leq 0$ then $\ind{y\leq S_{\tau_{x,m}} } = 1$ whereas if $y>0$
then
$\{y\leq S_{\tau_{x,m}} \} = \{\mathcal{E}^{(x,m)}_{y-1} \geq 1\}$. In
either case we have that $\ind{y\leq S_{\tau_{x,m}} }$ is
$\mathcal{G}^{x,m}_{y-1}$-measurable, and thus the sums inside the
first probability on the right are martingale
difference sums with bounded increments.  Therefore, it follows from
Azuma's inequality that
\[
%%% P\left( \left| \sum_{y \in J} \left\{ r^+\left(R^y_{\mathcal{L}(\tau_{x,m},y)+1}  \right) - r^{x,m}_y \right\} \ind{y\leq S_{\tau_{x,m}} }  \right| > \frac{n^{1/8}}{\ln n} \right)
\text{the sum in \eqref{martpart}}
\leq \sum_{\substack{J\subset\bint{x,K\sqrt{n}}\\ |J|=\fl{n^{1/8}}}} e^{-c \frac{n^{1/4}}{(\ln n)^2|J|} } 
\leq C K n^{1/2} e^{-c \frac{n^{1/8}}{(\ln n)^2}}. 
\]
To bound the probability in \eqref{rxypart}, note first of all that
$R^y_{\mathcal{L}(\tau_{x,m},y)+1}$ represents the next cookie to be
used at $y$ after time $\tau_{x,m}$. If $y>x$ then the last visit to
$y$ prior to $\tau_{x,m}$ resulted in a step to the left. Since
$\mathcal{D}^{(x,m)}_y$ is the number of steps left from $y$ prior to
$\tau_{x,m}$, we have that the distribution of
$R^y_{\mathcal{L}(\tau_{x,m},y)+1}$ conditioned on
$\{ \mathcal{D}^{(x,m)}_y = \ell \}$ is equal to the distribution of
the next cookie in a stack after the $\ell$-th step left, and this
distribution is known to converge to $\pi^+$ exponentially fast in
$\ell$ (see \cite[Section 3]{kpERWMCS}). Since
$\mathcal{D}^{(x,m)}_y = \mathcal{E}^{(x,m)}_{y-1} + \ind{x<y\leq 0}$
is $\mathcal{G}^{x,m}_{y-1}$-measurable and
$\pi^+ \cdot \mathbf{r}^+ = 0$, this implies that there are constants
$C,c>0$ such that
\[
 |r^{x,m}_y| = |r^{x,m}_y - \pi^+ \cdot \mathbf{r}^+| \leq C e^{-c \mathcal{E}^{(x,m)}_{y-1} }, \qquad \text{for all } y > x. 
\]
%Since $r^{x,m}_x$ is bounded by the largest absolute value of $r^+$ we have that for any $\alpha>0$ 
Therefore, we have that for any $\alpha>0$
\begin{align*}
\left| \sum_{y \in J} r^{x,m}_y \ind{y\leq S_{\tau_{x,m}} } \right| 
&\leq C + \sum_{y \in J \backslash \{x\}} C e^{-c \mathcal{E}^{(x,m)}_{y-1}}\ind{y\leq S_{\tau_{x,m}} } \\
&\leq C\left(1 + |J|e^{-c n^\alpha} + \sum_{y \in J \backslash \{x\}} \ind{\mathcal{E}^{(x,m)}_{y-1} < n^\alpha} \ind{y\leq S_{\tau_{x,m}} } \right).
\end{align*}
Using this we obtain that for $\alpha>0$, $\beta<\frac{1}{8}$ and $n$ sufficiently large
\begin{align*}
  \eqref{rxypart}
  &\leq  P\left(  \exists J\subset \bint{x+1,K\sqrt{n}}, \, |J|= \fl{n^{1/8}}: \, C \sum_{y \in J} \ind{\mathcal{E}_{y-1}^{(x,m)}< n^\alpha} \ind{y \leq S_{\tau_{x,m}}} > \frac{n^{1/8}}{4 \ln n} \right)  \\
  &\leq \sum_{x<z\leq 0} P\left( \sum_{y=z}^{(z+\lfloor n^{1/8} \rfloor) \wedge 0} \ind{\mathcal{E}_{y-1}^{(x,m)}< n^\alpha} > n^\beta \right) 
    + P\left( \sum_{x\vee 0 < y } \ind{\mathcal{E}^{(x,m)}_{y-1} < n^\alpha} \ind{y\leq S_{\tau_{x,m}} } > n^{\beta} \right) \\
  & \leq |x| \sup_{j\geq 1} P\left( \sum_{i=0}^{\fl{n^{1/8}}} \ind{V_i^+ < n^\alpha} > n^\beta \mid V_0^+ = j \right) 
    + \sup_{j\geq 0} P\left( \sum_{i = 0}^{\sigma_0^{U^+}-1} \ind{U_i^+ < n^\alpha} > n^\beta \mid U_0^+ = j \right),
\end{align*}
where the last inequality follows from connection of $\{\mathcal{E}^{(x,m)}_y \}_{y\geq x}$ with the BLPs $U^+$ and $V^+$ noted at the beginning of the proof. 
If we then choose $\alpha \in (0, \frac{1-(\theta^-\vee 0)}{8})$ and $\beta \in (\frac{(\theta^- \vee 0)}{8} + \alpha, \frac{1}{8})$ we can apply
Lemmas \ref{small1} and \ref{small2} to bound the last line above by $C (|x|+1) e^{-c n^r}$ for some constants $C,c,r>0$. 
% Lemmas \ref{small2} to bound the first term by $|x| C e^{-c n^r} \leq C K\sqrt{n} e^{-c n^r}$ and Lemma \ref{small1} to bound the second term by $C e^{-c n^{\beta-\alpha}}$.
Applying this, together with the bound on the sum in
\eqref{martpart} we obtain that
\[
 \eqref{notgood} \leq C (2K\sqrt{n}+1)K'\sqrt{n} \left( K n^{1/2} e^{-c \frac{n^{1/8}}{(\ln n)^2}} + (1+K\sqrt{n}) e^{-c n^r} \right),
\]
%\begin{align*}
% & P\left( \begin{array}{cc} \text{remaining first cookie environment on $[X_{\tau_{x,m}},S_{\tau_{x,m}}\wedge K\sqrt{n}]$} \\ \text{is not $(n^{1/8},0)$-good for some $|x|\leq K\sqrt{n}$, $m\leq K'\sqrt{n}$} \end{array} \right) \\
%&\leq C (2K\sqrt{n}+1)K'\sqrt{n} \left( K n^{5/8} e^{-c \frac{n^{1/8}}{(\ln n)^2}} + (1+K\sqrt{n}) e^{-c n^r} \right)
%\end{align*}
for $n$ large enough. Since the right side vanishes as $n\to \infty$ for any $K,K'<\infty$, this completes the proof of Lemma \ref{18good}. 
\end{proof}

\subsection{Lifting and grounding properties of the environment.}
We shall show that in a small neighborhood of every mesoscopic site
except for the site occupied by the walk, the environment is locally
close to i.i.d.\ in an appropriate equilibrium. This property is
preserved with probability close to 1 as the walk moves from one
mesoscopic site to another for any fixed (possibly very large) number
of steps (order $\epsilon^{-2}$). The important consequence of this is
that the environment around every mesoscopic site in the bulk will
 have lifting and grounding properties (see
Definitions \ref{lifting} and \ref{grounding}) which together with
$n^{1/8}$-goodness will allow us to use our diffusion approximations
(i.e.\ versions of generalized Ray-Knight theorems). We start with
several definitions.
  \begin{defn}\label{bulkreg}
    Given $\epsilon>0,\delta_1\in(0,\epsilon^3),\delta_2>0$, and
    $k\in\Z_+$, for the ERW $X$ stopped at time $\tenk$, the
    first cookie environment on $\bint{\xenk-\esrn,\xenk+\esrn}$ is
    said to be {\em bulk regular for $(\delta_1,\delta_2)$} if
    \begin{enumerate}[\tiny$\bullet$]
    \item $\ienk\le \xenk-\esrn% -\fl{\epsilon^3\sqrt{n}}
      $ and
      $\xenk+\esrn% +\fl{\epsilon^3\sqrt{n}}
      \le \senk$;
    \item on $\bint{\xenk,\xenk+\esrn}$ the first cookie environment is $(n^{1/8},0)$-good and \\ on $\bint{\xenk-\esrn,\xenk}$ the first cookie environment is $(n^{1/8},\nu/2-1)$-good;
    \item on $\bint{\xenk+\esrn-\edsrn,\xenk+\esrn}$ the first cookie
      environment is $\delta_2 \epsilon\sqrt{n}$-lifting from the
      right and $\delta_2 \epsilon\sqrt{n}$-grounding from the left;
      on $\bint{\xenk-\esrn,\xenk-\esrn+\edsrn}$ the first cookie
      environment is $ \delta_2 \epsilon\sqrt{n}$-lifting from the
      left and $\delta_2 \epsilon\sqrt{n}$-grounding from the right.
    \end{enumerate}
\end{defn}
\begin{defn}\label{exreg}
  Given $\epsilon>0,\delta_1\in(0,\epsilon^3),\delta_2>0$, and
  $k\in\Z_+$, for the ERW $X$ stopped at time $\tenk$, the
  first cookie environment on $\bint{\xenk-\esrn,\xenk+\esrn}$ is said
  to be {\em S-regular for $(\delta_1,\delta_2)$} if
    \begin{enumerate}[\tiny$\bullet$]
    \item % either $k=0$ or
      % $\dist(\ienk,\esrn \Z)>\fl{\epsilon^3\sqrt{n}}$ and
      $\senk\in \bint{\xenk,\xenk+\esrn-\fl{\epsilon^4\sqrt{n}}}$;
    \item on $\bint{\xenk,\senk}$ the first cookie environment is $(n^{1/8},0)$-good and on $\bint{\xenk-\esrn,\xenk}$ the first cookie environment is $(n^{1/8},\nu/2-1)$-good;
    \item on $\bint{\xenk-\esrn,\xenk-\esrn+\edsrn}$ the first cookie
      environment is $\delta_2 \epsilon\sqrt{n}$-lifting from the left and
      $\delta_2 \epsilon\sqrt{n}$-grounding from the right.
    \end{enumerate}
    The notion of {\em I-regular for $(\delta_1,\delta_2)$} on
    $\bint{\xenk-\esrn,\xenk+\esrn}$ first cookie environment is defined
    in a symmetric manner. 
  \end{defn}
  We shall say that the first cookie environment on
  $\bint{\xenk-\esrn, \xenk+\esrn}$ is {\em regular for
    $(\delta_1,\delta_2)$} if it is either bulk regular, or $S$-regular,
  or $I$-regular.  
  % Note that for $k=0$ and only for $k=0$ the first
  % cookie environment on $\bint{-\esrn,\esrn}$ is both $I$- and $S$-regular.

\begin{defn}\label{deltarel}
  Given $\epsilon>0,\delta_1\in(0,\epsilon^3),\delta_2>0$, we shall
  write $\delta_1\overset{\epsilon}{\sim}\delta_2$ if for all
  sufficiently large $n$, 
\begin{align*}
 &P_{\pi^+}\left( \tau^{V^-}_{ \delta_2 \epsilon\sqrt{n} } \leq \delta_1 \epsilon \sqrt{n} \right) \geq 1-\epsilon^6, 
\quad &P_{\pi^+}\left( \sigma_0^{U^+} \leq \delta_1 \epsilon \sqrt{n} \mid U^+_0 =  \fl{\delta_2 \epsilon\sqrt{n} } \right) \geq 1-\epsilon^6, \\
 &P_{\pi^-}\left( \tau^{V^+}_{ \delta_2 \epsilon\sqrt{n} } \leq \delta_1 \epsilon \sqrt{n} \right) \geq 1-\epsilon^6, 
\quad\text{and}
&P_{\pi^-}\left( \sigma_0^{U^-} \leq \delta_1 \epsilon \sqrt{n} \mid U^-_0 =  \fl{\delta_2 \epsilon\sqrt{n} } \right) \geq 1-\epsilon^6. 
\end{align*}
\end{defn}

\begin{rem}
  In the proof below, we will need the fact that one can always find
  parameters $\delta_1$ and $\delta_2$ that are sufficiently small and
  in the relation $\delta_1\overset{\epsilon}{\sim}\delta_2$.  To see
  this, recall that $\theta^+(\pi^+) = 0$ and $\theta^-(\pi^-) = 0$. 
%if the first cookie environment is i.i.d.\ with
%  marginal $\pi^+$ then the parameter
%  $\theta^+ = \theta(\pi^+,K,p(\cdot)) = 0$ and if the first cookie
%  environment is i.i.d.\ with marginal $\pi^-$ then the parameter
%  $\theta^- = \theta(\pi^-,K,1-p(\cdot)) = 0$.  
Then it follows from
  Lemmas \ref{liftIID} and \ref{sm0} (and the discussion at the
  beginning of Section \ref{daiid}) that the conditions in Definition
  \ref{deltarel} hold if $\delta_2$ is sufficiently small compared to
  $\delta_1$.
\end{rem}

\begin{rem}\label{rem:drel}
 Note that the events in Definition \ref{deltarel} are closely related to the definitions of lifting and grounding first cookie environments in Definitions \ref{lifting} and \ref{grounding}. In particular, if $\delta_1\overset{\epsilon}{\sim}\delta_2$ then a first cookie environment on  $\bint{0,\delta_1\epsilon\sqrt{n}}$ with $\pi^+$-product measure will be $ \delta_2\epsilon\sqrt{n}$-lifting from the right and $\delta_2 \epsilon\sqrt{n}$-grounding from the left with probability at least $1-2\epsilon^6/\epsilon^3 = 1-2\epsilon^3$. 
\end{rem}

\begin{lemma}\label{e1a}
    Given an $\epsilon>0$ let
    $\delta_1\overset{\epsilon}{\sim}\delta_2$ and $\delta_1$ be
    sufficiently small. Suppose that for some $k\in\Z_+$ we have
    $\xenk=x$ and the environment on $\bint{x-\esrn,x+\esrn}$ is
    regular for $(\delta_1,\delta_2)$. Then there is a $C>0$ not
    depending on $\epsilon$ such that for $n$ large, outside of
    probability $C\epsilon^3$,
\begin{enumerate}[(1)]
\item on the event $\{T^{\epsilon,n}_{k+1}=T^{\epsilon,n,-}_{k+1}\}$
  the first cookie environment on $\bint{x-\edsrn,x}$ is
  $\delta_2 \epsilon\sqrt{n}$-lifting from the right and
  $\delta_2 \epsilon\sqrt{n}$-grounding from the left;
    \item on the event $\{T^{\epsilon,n}_{k+1}=T^{\epsilon,n}_{k+1,+}\}$ the first cookie environment on $\bint{x,x+\edsrn}$ is $\delta_2 \epsilon\sqrt{n}$-lifting from the left and $ \delta_2 \epsilon\sqrt{n}$-grounding from the right.
    \end{enumerate}
  \end{lemma} 

The proof relies on the following two lemmas.
  \begin{lemma}\label{unif}
    There is a constant $\Cl{unif}$ such
    that uniformly over all first cookie environments environments
    satisfying the conditions of Lemma~\ref{e1a} for all
    sufficiently large $n$
    \[\Cr{unif}\le P(\tenk=T^{\epsilon,n}_{k,-}\mid
      \xenk=x,\ R^y_1,\ y\in\bint{x-\esrn,x+\esrn}) \le 1-\Cr{unif}.\]
  \end{lemma}

  \begin{proof}
    By Lemma \ref{hitpr}, it is enough to show that uniformly in
    $\wu \in (0,1) $ and $\wo \in (1,2) $ the probability
    $P(\tau_0(\wu,1,\wo) < \tau_2(\wu,1,\wo))$ is bounded away from
    $0$ and $1$.  By symmetry of the problem, we only have to argue that
    this probability is uniformly bounded away from $0$.  By
    Proposition \ref{RNbmpe} we must simply show that for the process
    $y(\cdot)$ and $w = 1$ the probability of $y(\cdot)$ hitting zero
    in $(1,2) $ is bounded away from zero as $\wo $ and $\wu$ vary.
    But this follows easily by noting that the minimum is achieved
    with $\wu$ equal to $0$ or $1$ and $\wo$ equal to $1$ or $2$.
  \end{proof}
  
  \begin{lemma}\label{lifting1}
    Assume the conditions of Lemma~\ref{e1a}. There is a
    $\delta_3=\delta_3(\epsilon,\theta^-)>0$ such that with
    probability $1-2\epsilon^3$ for $n$ large enough, the process $V^+$
    (resp.\ $V^-$) starting with $0$ particles in generation $0$ and
    using the first cookie environment on $\bint{x-\esrn+1,x}$ for
    generations $1,2,\dots,\esrn$ (resp.\ $\bint{x,x+\esrn-1}$
    for generations $\esrn,\esrn-1,\dots,1$) 
    satisfies
    \[V^+_j (\text{resp.\ }V^-_j)\ge \delta_3\sqrt{n},\quad \forall
      j\in \bint{\esrn-\edsrn,\esrn}. \] 
  \end{lemma}
\begin{proof}
  We shall only consider the process $V^+$, the other case follows by
  a symmetric reasoning. 

{\em Step 1.} We start with the case when the first cookie
    environment on $\bint{x-\esrn,x+\esrn}$ is either bulk regular or
    S-regular for $(\delta_1,\delta_2)$. To simplify the notation and
    without loss of generality we shall assume that $x=\esrn$ and
    consider the process $V^+$ on $\bint{0,\esrn}$. Given our
  assumptions on the first cookie environment, we have
  $\tau^{V^+}_{\delta_2 \epsilon\sqrt{n}}\le\edsrn$ outside of probability
  $\epsilon^3$. By the strong Markov property and monotonicity of BLPs
  with respect to the initial number of particles we have that,
  conditional on
  ${\cal G}_{\tau^{V^+}_{\delta_2 \epsilon\sqrt{n}}}\coloneqq
  \sigma(V^+_j,\,j\le \tau^{V^+}_{\delta_2 \epsilon\sqrt{n}})$,
  $V^+_{\tau^{V^+}_{\delta_2 \epsilon\sqrt{n}}+\ell}$ will be stochastically
  larger than the BLP $Z^+_\ell$,
  $\ell=0,1,\dots,\esrn-\tau^{V^+}_{\delta_2 \epsilon\sqrt{n}}$, which starts with
  $\fl{\delta_2 \epsilon\sqrt{n}}$ particles in generation $0$ and
  evolves in the environment on
  $\bint{\tau^{V^+}_{\delta_2 \epsilon\sqrt{n}}+1,\esrn}$ for generations
  $1,2,\dots,\esrn-\tau^{V^+}_{\delta_2 \epsilon\sqrt{n}}$. Without loss of
  generality we can extend the process $Z^+$ to the full interval
  $\bint{0,\esrn}$ by choosing the environment on
  $\bint{ \esrn-\tau^{V^+}_{\delta_2 \epsilon\sqrt{n}}+1,\esrn }$ to be in
  $\pi^+$ product measure. By Theorem~\ref{da0.25} and the fact that
  the environment is assumed to be either bulk- or S-regular, for each
  $\delta\in(0,\delta_2/2)$ the processes
  \[\frac{Z^+_{\fl{t\epsilon\sqrt{n}}\wedge
        \sigma^{Z^+}_{\delta\epsilon\sqrt{n}}}}{\fl{\epsilon \sqrt{n}}},\quad t\in[0,1],\]
  converge weakly as $n\to\infty$ to a constant multiple of a
  BESQ$^{2}$, $Y(t\wedge \sigma_{\delta}),\ t\in[0,1]$, with
  $Y(0)=\delta_2>0$. We can choose
  $\delta_3=\delta_3(\epsilon,\delta_2)>0$ so that
  \[P\left(\min_{t\in[0,1]}Y(t)>\delta_3\right)>1-\frac{\epsilon^3}{2}.\]
  Then outside of probability $\epsilon^3$ for $n$ large
  $Z^+_\ell\ge \delta_3\sqrt{n}$ for all $\ell\in \bint{0,\esrn}$. By
  stochastic domination we conclude that for the same $\delta_3$ and
  all $n$ large $V^+_j\ge \delta_3\sqrt{n}$ for all
  $j\in \bint{\esrn-\edsrn, \esrn}$ as claimed.

  {\em Step 2.} Suppose now that the first cookie environment on
    $\bint{x-\esrn,x+\esrn}$ is I-regular for $(\delta_1,\delta_2)$
    and $\delta_1$ is sufficiently small. Then $x\le 0$. However,
    after an appropriate shift we may again assume that $x=\esrn$ so
    that $I_{\tenk}\in\bint{\fl{\epsilon^4\sqrt{n}},\esrn}$. The
    process $V^+_j$ will be evolving in the product environment with
    marginal $\eta$ for $j\in\bint{0, I_{\tenk}-1}$ and then for
    $j\in\bint{I_{\tenk},\esrn}$ will use the environment modified by
    the walk. Recall that the first cookies on the latter interval are
    a part of the information known at time $\tenk$. By the regularity
    assumption, $I_{\tenk}\ge \fl{\epsilon^4\sqrt{n}}$ so that we can
    use Theorem~\ref{DA0} at least on the time interval
    $[0,\epsilon^4]$. The diffusion approximation of Theorem~\ref{DA0}
    is a $\frac\nu{4}$BESQ$^{2(1-\theta^-)}$ process $Y$ with
    $Y(0)=0$. Since $2(1-\theta^-)>0$, by scaling properties of BESQ
    processes we get that
  \[\inf_{t\in[\epsilon^4,1]}P(Y(t)>\delta)=\inf_{t\in[\epsilon^4,1]}P(tY(1)>\delta)=P(Y(1)>\delta\epsilon^{-4})\to 1\quad\text{as}\quad \delta\to
    0.\] Therefore, given $\epsilon>0$, we can find
  $\delta=\delta(\epsilon,\theta^-)>0$ such that
  $\inf_{t\in[\epsilon^4,1]}P(Y(t)>\delta)>1-\epsilon^3/2$.  Since the
  ``switch point'' from the original product environment to the
  environment modified by the walk, $I_{\tenk}$, is a part of the
  information given at time $\tenk$ and since
  $Y(n^{-1/2}(I_{\tenk}-1))>\delta$ with probability at least
  $1-\epsilon^3/2$, we get by Theorem~\ref{DA0}
  that
  \begin{equation}
    \label{joint}
    P\left(V^+_{I_{\tenk}-1}\ge \fl{\delta\sqrt{n}}\right)\ge 1-\epsilon^3\quad\text{for all sufficiently large $n$.}
  \end{equation}
  
  Next we shall choose $\delta_3$. Let
  $s_n:=n^{-1/2}(I_{\tenk}-1)$. Using the fact that BESQ$^2$ process
  $\tilde{Y}$ with $\tilde{Y}(0)=\delta$ a.s.\ does not hit zero we
  can find a $\delta_3=\delta_3(\epsilon,\delta)\in(0,\delta)$ such
  that
  $P(\sigma^{\tilde{Y}}_{\delta_3}>1\mid\tilde{Y}(0)=\delta)\ge
  1-\epsilon^3/4$. The requirement for
  $\delta_1=\delta_1(\epsilon,\theta^-)$ to be sufficiently small
  comes from the fact that we do not have any control on how close
  $I_{\tenk}$ is to $\esrn$. It could happen that
  $I_{\tenk}\in\bint{\esrn-\edsrn,\esrn}$. We know that \eqref{joint}
  holds and we need to show that
  \begin{equation}
    \label{want}
    P\left(V^+_j\ge \fl{\delta_3\sqrt{n}} \ \forall
    j\in\bint{\esrn-\edsrn,\esrn}\right)\ge 1-2\epsilon^3\quad\text{for all sufficiently large $n$.}
  \end{equation}
  By our choice of $\delta_3$ the process $\tilde{Y}$ with
  $\tilde{Y}(s_n)>\delta$ stays above $\delta_3$ on
  $[s_n, s_n+1]$ with probability at least
  $1-\epsilon^3/4$. We shall choose
  $\delta_1=\delta_1(\epsilon,\delta_3,\delta)$ so that on the event
  $\{Y(s_n)\ge \delta\}$ the process $Y$ stays above
  $\delta_3$ on
  $[s_n-\delta_1\epsilon,s_n]$ with
  probability at least $1-\epsilon^3/4$. Thus, we let
  $\delta_1=\delta_1(\epsilon,\delta_3)>0$ be so small that
  \[\max_{y\ge 0}P\left(\min_{t\in[0,\delta_1\epsilon]}Y(t)\le
      \delta_3, Y(\delta_1\epsilon)\ge \delta\mid Y(0)=y\right)\le
    P\left(\tau^Y_\delta<\delta_1\epsilon\mid
      Y(0)=\delta_3\right)<\epsilon^3/4.\] Note that
  $\delta,\delta_3,\delta_1$ depend only on $\epsilon$ and
  $\theta^-$. Theorem~\ref{DA0}, Theorem~\ref{da0.25} and our choice
  of $\delta,\delta_3,\delta_1$ give \eqref{want}.
\end{proof}

\begin{proof}[Proof of Lemma~\ref{e1a}]
  It is enough to show (1).  Apart from Lemmas~\ref{unif} and
  \ref{lifting1} we shall use the fact that the auxiliary
  Markov chain 
%which counts the number of ``successes'' before each ``failure'' in a single cookie stack 
which keeps track of the next cookie in the stack after each successive ``failure'' in the corresponding sequence Bernoulli trials
converges to its equilibrium distribution $\pi^+$ exponentially fast (see (18) on p.\,1472 of \cite{kpERWMCS}).

  Without loss of generality we shall assume
  that $x=\esrn$. Note that at sites visited by the walk by
  time $\tenk$, the first cookies are non-random while on any
  unvisited interval they are in the initial product measure
  with marginal $\eta$.  In all cases the first cookie distribution on
  $\bint{0,\esrn}$ is a (possibly degenerate) product measure. Given the
  conditions imposed on the environment and the ERW at time $\tenk$,
  consider the event
  \begin{align*}
    A=&\{\text{at time $T^{\epsilon,n}_{k+1,-}$ the first cookie
        environment on $\bint{\esrn - \fl{\epsilon \delta_1\sqrt{n}, \esrn} }$ is}
    \\ &\text{either not $\delta_2 \epsilon\sqrt{n}$-lifting
         from the right or not $\delta_2 \epsilon\sqrt{n}$-grounding from the left.}\}
\end{align*}
We can estimate the probability of $A$ by considering a BLP $V^+$ from
Lemma~\ref{lifting1} which uses the cookie environment created by the
walk on $\bint{1,\esrn}$ up to time $\tenk$ for generations
$1,2,\dots,\esrn$.

{\em Step 1.} By Lemma~\ref{lifting1}, if
$\sigma=\inf\{j\ge \esrn-\edsrn:\,V^+_j<\delta_3\sqrt{n}\}$ then
$P(\sigma \leq \esrn) \leq 2\epsilon^3$ 
for all
sufficiently large $n$.

  {\em Step 2.} For $j\in\bint{\esrn-\edsrn,\esrn}$, let
  $\psi_{j,k}$ be the state of the cookie Markov chain at site
  $j$ at time $\tenk$,
  \[\psi_{j,k}=R^j_{{\cal L}(\tenk,j)+1},\] and
  $\psi^-_{j,k+1}$ be the state of the cookie Markov chain at site
  $j$ at time $T^{\epsilon,n}_{k+1,-}$,
  \[\psi^-_{j,k+1}=R^j_{{\cal
      L}(T^{\epsilon,n}_{k+1,-},j)+1}.\] 
   Then by (18) on p.\,16 of
  \cite{kpERWMCS}, for every $i\in{\cal R}$
  \[\left|P\left(\psi^-_{j,k+1}=i \,\Big|\,
        \psi_{j,k},\,V^+_{j-1}=k,\,\psi^-_{\ell,k+1},\ 0\le
        \ell<j\right) -\pi^+(i)\right|\le c_7e^{-c_8k}, \quad \forall
    k\in\N.\] Summing up over $i\in{\cal R}$ and using induction over 
  $j\in \bint{\esrn-\edsrn+1,\esrn}$ we conclude that on the event
  $\{\sigma> \esrn\}$ for all sufficiently large $n$ 
the total variation distance between the joint distribution of $\{\psi^-_{j,k+1} \}_{\esrn-\edsrn < j\le \esrn}$ and a $\pi^+$-product measure is at most 
 $N(\edsrn+1) c_7e^{-c_8\delta_3 \sqrt{n}}\le \epsilon^3$. 

  {\em Step 3.} 
By remark \ref{rem:drel} and the assumption that $\delta_1\overset{\epsilon}{\sim}\delta_2$,
 the probability that a
  first cookie environment on the interval $\bint{\esrn-\edsrn,\esrn}$
  sampled from the $\pi^+$-product measure will be either not
  $\delta_2 \epsilon\sqrt{n}$-lifting from the right or not
  $\delta_2 \epsilon\sqrt{n}$-grounding from the left is at most
  $2\epsilon^6/\epsilon^3=2\epsilon^3$ for all $n$ large.

  Adding up the probabilities from {\em Steps 1--3} we conclude that
  given that the first cookie environment at time $\tenk$ satisfies
  all conditions of the lemma, the probability of event $A$ does not
  exceed $5\epsilon^3$ for all sufficiently large $n$.

  {\em Step 4.} Finally, we have to also condition on the event
  $\{T^{\epsilon,n}_{k+1}=T^{\epsilon,n}_{k+1,-}\}$. We know by
  Lemma~\ref{unif} that the probability of this event is at least
  $\Cr{unif}>0$ uniformly overall environments satisfying the
  conditions of Lemma~\ref{e1a}. Therefore, conditioning on
  $\{T^{\epsilon,n}_{k+1}=T^{\epsilon,n}_{k+1,-}\}$ we get that the
  probability of $A$ is less than or equal to $5\epsilon^3/\Cr{unif}$.
  \end{proof}

\begin{lemma}\label{Away}
  Under the assumptions of Lemma~\ref{e1a}, there is a constant $C>0$
  such that for all sufficiently small $\epsilon>0$ and all
  $n\ge n_0(\epsilon)$ outside of probability $C\epsilon^{2.9}$
 \[\max_{\tenk\le i\le T^{\epsilon,n}_{k+1}}X_i-\min_{\tenk\le i\le T^{\epsilon,n}_{k+1}}X_i\le 2\fl{\epsilon\sqrt{n}}-\fl{\epsilon^4\sqrt{n}}.\]
\end{lemma}

\begin{proof}
Without loss of generality we shall assume that
$(I_{\tenk},X_{\tenk},S_{\tenk})=(m^{\epsilon,n}_k,0,M^{\epsilon,n}_k)$
for some integers $m^{\epsilon,n}_k\le 0$ and $M^{\epsilon,n}_k\ge 0$. Let
\[ \tau=\inf \{n\ge \tenk:\
|X_i|=\fl{\epsilon\sqrt{n}}-\fl{\epsilon^4\sqrt{n}}\}.\] Assume
for definiteness that
$X_\tau=-\fl{\epsilon\sqrt{n}}+\fl{\epsilon^4\sqrt{n}}$.
Heuristically, if subsequent to $\tau$ the walk $X$ were a simple
  symmetric random walk on spatial interval
  $\bint{-\fl{\epsilon\sqrt{n}},0}$ starting from
  $\fl{\epsilon^4\sqrt{n}}-\fl{\epsilon\sqrt{n}}$ then by gambler's
  ruin considerations $X$ would (outside of
probability of order $\epsilon^3$) hit $-\fl{\epsilon\sqrt{n}} $
before $0$. Given the nature of our problem, we recast this in terms
of upcrossings: outside of this order of probability we do not expect
an upcrossing to $0$ between time $\tau $ and
$T^{- \epsilon,n}_{k+1}$. Since the cookie environment
  equilibrates very fast, these simple heuristics happen to be almost
  correct.

  We will consider, as usual, upcrossings from
  $- \fl{\epsilon\sqrt{n}} +i$ made between times $T^{\epsilon,n}_{k}$
  and $T^{ -\epsilon,n}_{k+1}$. We will decompose these as the sum of
  upcrossings between $T^{\epsilon,n}_{k}$ and $\tau$ and
  ``additional'' upcrossings made afterwards. We will show that the
  number of these additional upcrossings (outside probability of order
  $\epsilon^{2.9}$) becomes small and stays small until it becomes $0$
  before $i= \fl{\epsilon\sqrt{n}}$. Below we denote by $C$ possibly
  different positive constants. 

{\em Step 1.}  Consider the BLP $V^+$ which starts with $0$ particles
in generation $0$ and uses the first cookie environment recorded at
time $\tenk$ on $\bint{-\fl{\epsilon \sqrt{n}}+1,0}$ for
generations $\bint{1,\fl{\epsilon\sqrt{n}}}$.  Denote by $\tilde{V}^+$
the same type of process, but let it instead use the first cookie
environment on
$\bint{-\fl{\epsilon\sqrt{n}}+\fl{\epsilon^4\sqrt{n}}+1,0}$ for
generations
$\bint{1,\fl{\epsilon\sqrt{n}}-\fl{\epsilon^4\sqrt{n}}}$. This process
gives the number of upcrossings from
$-\fl{\epsilon\sqrt{n}}+\fl{\epsilon^4\sqrt{n}}+i$ by time $\tau$.

We embed $\tilde{V}^+$ into $V^+$  and denote by $\tilde{U}^+$
  the number of ``additional'' upcrossings. Namely, we observe that
the number of particles in generation $j\ge \fl{\epsilon^4\sqrt{n}}$
of the process $V^+$ is equal to the number of particles of
$\tilde{V}^+$ in generation $j-\fl{\epsilon^4\sqrt{n}}$ plus the
number of particles of the process $\tilde{U}^+$ in generation
$j-\fl{\epsilon^4\sqrt{n}}$, where $\tilde{U}^+$ uses the environment
created by the ERW at time $\tau$ on
$\bint{[-\fl{\epsilon\sqrt{n}}+\fl{\epsilon^4\sqrt{n}}+1,0}$ for
generations $\bint{1,\fl{\epsilon\sqrt{n}}-\fl{\epsilon^4\sqrt{n}}}$.
In short,
  \begin{equation}
    \label{short}
    V^+_j=\tilde{V}^+_{j-\fl{\epsilon^4\sqrt{n}}}+\tilde{U}^+_{j-\fl{\epsilon^4\sqrt{n}}},\quad
    j\ge \fl{\epsilon^4\sqrt{n}}.
  \end{equation}
  Note that $\tilde{V}^+_0=0$ and $\tilde{U}^+$ starts with
  $V^+_{\fl{\epsilon^4\sqrt{n}}}$ particles in generation $0$.
Our goal is to show that with large probability the process $\tilde{U}^+$ dies out before $j=\esrn$.

{\em Step 2.}  We consider process $V^+$.  According to
Lemma~\ref{lifting1}, there is a $\delta_3>0$ such that with
probability at least $1-2\epsilon^3$
\begin{equation}
    \label{claim}
    V^+_j\ge \delta_3\sqrt{n},\quad\forall j\in
    \bint{\fl{\epsilon\sqrt{n}}-\fl{ \epsilon \delta_1\sqrt{n}},\fl{\epsilon\sqrt{n}}}.
  \end{equation}
The process $V^+$ is dominated by the process which starts with
$\fl{\epsilon^4\sqrt{n}}$ particles in generation $0$ and uses the
same environment and the same coin tosses. Using an appropriate
diffusion approximation (depending on whether
  $m^{\epsilon,n}_k\le -\fl{\epsilon\sqrt{n}}$ or
  $m^{\epsilon,n}_k\ge
  -\fl{\epsilon\sqrt{n}}+\fl{\epsilon^4\sqrt{n}}$) we can say that
with probability at least $1-C\epsilon^3$ for all sufficiently large
$n$ the process $V^+$ will have no more than
$\fl{\epsilon^{3.9}\sqrt{n}}$ particles in generation
$\fl{\epsilon^4\sqrt{n}}$, provided that $\epsilon$ was fixed
sufficiently small. That is outside of probability $C \epsilon ^3$
we have that
$V^+_{\fl{\epsilon^4\sqrt{n}}} = \tilde{U}^+_0 \le
\fl{\epsilon^{3.9}\sqrt{n}}.$

{\em Step 3.} We consider $\tilde{U}^+$. Our aim is to show that
outside of probability $C\epsilon ^{2.9}$ it
becomes small in time $\fl{\epsilon\sqrt{n}} /2$. By Step 2 and
  monotonicity of BLPs in the initial number of particles, it will be
sufficient to analyze the same process but starting
from a larger value $ \fl{\epsilon^{3.9}\sqrt{n}}$. We shall denote this process by
  $\hat{U}^+$. By Theorem \ref{da0.25}, for any fixed
  $\delta'>0$, processes
  ${\fl{\epsilon\sqrt{n}}}^{-1}\hat{U}^+_{\fl{\epsilon\sqrt{n}s}\wedge
    \sigma_{\fl{ \epsilon \delta ' \sqrt{n}}}},\ s \in [0, \frac12 ]$,
  $n\in\N$, converge in distribution as $n\to\infty$ to a
$\frac{\nu}{4}$BESQ$^0$ process $Y$ starting at $\epsilon^{2.9}$ and
stopped on hitting $\delta'$. By scaling and tail decay of
  extinction probabilities (see, for example, \cite[Lemma
  3.3]{kmLLCRW}),
  \[P(\sigma^Y_{\delta'}>1/2\mid Y(0)=\epsilon^{2.9})\le
    P(\sigma^Y_0>1/2\mid
    Y(0)=\epsilon^{2.9})=P(\sigma^Y_0>1/(2\epsilon^{2.9})\mid Y(0)=1)<
    C\epsilon^{2.9}.\] 
So for
any $\delta'>0$ and $n$ large enough,
$P( \sigma^{\tilde{U}^+}_{\fl{\delta'\sqrt{n}}} >
\fl{\epsilon\sqrt{n}}/2 )< C \epsilon^{2.9}$. We now fix $\delta' $ so
that $(\delta', \delta_3 / 2) $ are as $(\delta', \delta )$ for Lemma
\ref{dead}. We conclude that outside of probability $C\epsilon^{2.9}$
\[
\tilde{U}^+_k \le \hat{U}^+_k\le\delta_3 \sqrt{n}/2 \quad\mbox{ for } \quad \fl{\epsilon\sqrt{n}}/2 \le k \le \fl{\epsilon\sqrt{n}}-\fl{\epsilon^4\sqrt{n}}.
\]

{\em Step 4.} The last inequality, \eqref{claim} and \eqref{short}
imply that outside of probability $C\epsilon^{2.9}$, for $n$ large
\[\tilde{V}^+_k \ge \delta_3 \sqrt {n}/3\quad\text{for }\quad
  k \in \bint{\fl{\epsilon\sqrt{n}}- \fl{\delta_1 \epsilon \sqrt {n}}-
    \fl{\epsilon^4\sqrt{n}}, \fl{\epsilon\sqrt{n}}-
    \fl{\epsilon^4\sqrt{n}}}.\] This is enough to argue exactly
  as in the proof of Lemma \ref{e1a} that outside of probability
$C \epsilon ^{2.9} $ for $n$ large the first cookie environment on
this interval is $\delta_2 \epsilon \sqrt {n} $-grounding from the right
 at time $\tau$. Therefore, we can conclude
that with probability $1-C\epsilon^{2.9}$ the process $\tilde{U}^+$
will die out by generation
$\fl{\epsilon\sqrt{n}}-\fl{\epsilon^4\sqrt{n}}$ (assuming as we may
that $\delta_3/2 < \delta_2 \epsilon $) . In other words, after
hitting $-\fl{\epsilon\sqrt{n}}+\fl{\epsilon^4\sqrt{n}}$ the ERW will
hit $\fl{-\epsilon\sqrt{n}}$ before
$\fl{\epsilon\sqrt{n}}-\fl{\epsilon^4\sqrt{n}}$ (even before $0$) with
probability $1-C\epsilon^{2.9}$.
\end{proof}

\section{Coupling of the rescaled mesoscopic walk and 
  BMPE-walk}\label{sec:coup}

Let
$W^{\epsilon,n}_k=\frac{X_{\tenk}}{\fl{\epsilon\sqrt{n}}}$,
$k\in\Z_+$, be a re-scaled mesoscopic walk taking values in $\Z$ and
\[I^{\epsilon,n}_k=\min_{j\le
      \tenk}\frac{X^{\epsilon,n}_j}{\fl{\epsilon\sqrt{n}}},
    \quad S^{\epsilon,n}_k=\max_{j\le
      \tenk}\frac{X^{\epsilon,n}_j}{\fl{\epsilon\sqrt{n}}},\quad
    k\in\Z_+,\] be its running minimum and maximum respectively. The
walk $\{ (I^{\epsilon,n}_k,W^{\epsilon,n}_k,S^{\epsilon,n}_k)\}_{k\ge 0}$
is non-markovian. It depends on the ERW path in a random cookie
environment. To make it into a Markov process we have to retain some
information about the environment at each mesoscopic
step. Let \[\mu^{\epsilon,n,x}_k=
  \begin{cases}
    \eta,&\text{if }k=0\ \text{or }x\not\in\bint{I^{\epsilon,n}_k,S^{\epsilon,n}_k};\\
    \delta_i,&\text{where } i=R^x_{{\cal L}(\tenk,x)+1}\ \text{otherwise}.
  \end{cases}
\] In words, for $k=0$ or if a site has not been visited before time
$\tenk$ we set the distribution of the first cookie at that
site to $\eta$. For each site that has been visited before time
$\tenk$ we record the next state of the cookie Markov chain
(and, thus, fix the first cookie in the stack) at this site at
time $\tenk$.  Now the process
\[\{({\cal X}^{\epsilon,n}_k,\otimes_{x\in\Z}\,\mu^{\epsilon,n,x}_k)\}_{k\ge0}\coloneqq \{ ((I^{\epsilon,n}_k,W^{\epsilon,n}_k,S^{\epsilon,n}_k),\otimes_{x\in\Z}\,\mu^{\epsilon,n,x}_k)\}_{k\ge
  0}\] is a Markov process, since the information collected at each
step is sufficient to generate the next.

We now describe a coupling between ${\cal X}^{\epsilon,n}$ and a
{modified BMPE} $\txe\coloneqq (\tie,\twe,\tse)$ defined in
Section~\ref{mwalk}. For each $n$ we have to use a different version of
BMPE-walk, $\txen$, which is indicated by an additional superscript
$n$. In this coupling we will, in particular, address the filtration
$\cal{F}^{\epsilon,n}_k $ associated to our discrete time process.  We
take ${\cal X}^{\epsilon,n}$ as a primary object and use its randomness
(plus auxiliary, independent randomness) to define the coupling.  Our
description will detail how to construct $\txen$ in full but we will
talk of the coupling being ``broken'' for certain time indices.  This
term will signify that from this point the two processes are no longer
close (or that we do not expect them to be close).

Our goal is to couple ${\cal X}^{\epsilon,n}$, with the process
$\txen$ so that if the coupling is not broken by step $k$, then
${\cal X}^{\epsilon,n}_j=\txen_j$ for each $0 \leq j \leq k $.  We
repeat that in describing $\txen$, we must also describe the
filtration $\cal{F}^{\epsilon,n}_k $ to which it is adapted.  We will
certainly have that for each $k$ the filtration
$\cal{F}^{\epsilon,n}_k $ contains the $\sigma$-algebra generated by
our ERW up to time $\tenk$.  We will also suppose (after
enlarging the probability space if need be) that for each $k$,
$\cal{F}^{\epsilon,n}_k $ contains a number of i.i.d.\ uniform random
variables independent of the ERW $X$ and its cookie environment. These
uniform random variables will be used to generate the evolution of
$\txen$ once the coupling is broken: if the coupling is broken at step
$k$, then $\txen_{k+1}$ is generated corresponding to a BMPE with
initial conditions $\txen_k$ using these additional uniform random
variables. 

We begin with $k=0$ by setting $
{\cal X}^{\epsilon,n}_0=\txen_0=(0,0,0) 
$ and saying that  at step $k=0$ the coupling is unbroken.

\subsection{First step}

This step is for the coupling between ${\cal X}^{\epsilon,n}_1$ and
$\txen_1$ but it introduces ideas that will be used later in coupling
near extrema.  This step is special as it is the only step when both
extrema will change.

Recall that intervals $J_\ell$ were defined in \eqref{Jell}. We
compute for each $\ell$ the probabilities $p^n_{1,\ell}$ where
\begin{align*}
p^n_{1,\ell} &= P(A^n_\ell)\coloneqq P( W^{\epsilon,n}_1 = 1,  \ I^{\epsilon,n}_1 \in  J_\ell )\ \ \text{for }\ -\ell=1,2,\dots,L;\\
p^n_{1,\ell} &= P(A^n_\ell)\coloneqq P(  W^{\epsilon,n}_1 = -  1, \ S^{\epsilon,n}_1\in J_\ell )\ \ \text{for }\ \ell=1,2,\dots,L.
\end{align*}

We also compute the corresponding probabilities for a BMPE, $(I,W,S)$,
starting from $(0,0,0)$:
\begin{align*}
  q_{1,\ell}  &=  P( W(\tau(1,0,0,0))  = 1 ,\ I(\tau(1))  \in  J_\ell )\ \ \text{for }\  -\ell=1,2,\dots,L; \\
  q_{1,\ell}  &=  P( W(\tau(1,0,0,0))  = - 1 ,\ S(\tau(1))  \in  J_\ell )\ \ \text{for }\ \ell=1,2,\dots,L.
\end{align*}
The triple $(\tien_1,\twen_1,\tsen_1)$ is obtained by utilizing the maximal coupling of probability measures $\{ p^n_{1,\ell} \}_{0<|\ell|\le L}$ and $\{ q_{1,\ell} \}_{0<|\ell|\le L}$.
More precisely, if the event $A^n_\ell$ occurs for the ERW then with probability $1 \wedge \frac{q_{1,\ell}}{p^n_{1,\ell}} $ we let 
\[
 (\tien_1,\twen_1,\tsen_1) = 
(I^{\epsilon,n}_1,W^{\epsilon,n}_1,S^{\epsilon,n}_1) = 
\begin{cases}
 (I^{\epsilon,n}_1, 1,  1) & \text{if } -\ell = 1,2,\dots,L \\
 (-1,-1,S^{\epsilon,n}_1) & \text{if } \ell = 1,2,\dots,L, 
\end{cases}
\]
and if the above has not yet determined $(\tien_1,\twen_1,\tsen_1)$ (which is true with probability $\sum_{0<|\ell|\leq L} (q_{1,\ell}-p_{1,\ell}^n)_+$) then we use auxiliary independent randomness to determine $(\tien_1,\twen_1,\tsen_1)$ so that 
\[
 (\tien_1,\twen_1,\tsen_1)
= \begin{cases}
   (\frac{\ell+\frac{1}{2}}{L},1,1) & \text{if } -\ell=1,2,\dots,L \\
   (-1,-1,\frac{\ell-\frac{1}{2}}{L}) & \text{if } \ell =1,2,\dots,L, 
  \end{cases}
\quad \text{with probability } \frac{(q_{1,\ell} - p_{1,\ell}^n)_+}{\sum_{0<|\ell'|\leq L} (q_{1,\ell'}-p_{1,\ell'}^n)_+}. 
\]
Note that this coupling is such that
\begin{align*}
  P(\twen_1 = 1,\, \tsen_1 = 1, \, \tien \in J_\ell ) &= q_{1,\ell} && \text{for } -\ell = 1,2,\dots,L; \\
 \text{and }  P(\twen_1 = -1, \, \tien_1=-1, \, \tsen \in J_\ell ) &= q_{1,\ell} && \text{for } \ell = 1,2,\dots,L, 
\end{align*}
and such that ${\cal X}^{\epsilon,n}_1 = \txen_1$ 
%$(\tien_1,\twen_1,\tsen_1) = (I^{\epsilon,n}_1,W^{\epsilon,n}_1,S^{\epsilon,n}_1)$ 
with probability at least 
\[
 1-\sum_{0<|\ell|\le L}(p^n_{1,\ell}-q_{1,\ell})_+=1 - \frac12 \sum _{0<|\ell|\le L} | p^n_{1,\ell}-q_{1,\ell} |.
\]

\begin{defn}\label{broken}
  We say that the coupling is broken after step $j$ if
\begin{align}
  \label{broken1}
  &\txen_j\ne {\cal X}^{\epsilon,n}_j\ \ \text{or }\\ 
\label{broken2} 
&W^{\epsilon,n}_j-W^{\epsilon,n}_{j-1}=-1\ \ \text{and }\ \max_{T^{\epsilon,n}_{j-1}\le i<T^{\epsilon,n}_j} \left( X_i-X_{T^{\epsilon,n}_{j-1}} \right)  > \left( 1 - \epsilon ^ 3 \right) \fl{\epsilon\sqrt{n}} \ \ \text{or }
\\ 
\label{broken3} 
&W^{\epsilon,n} _j-W^{\epsilon,n}_{j-1}=1\ \ \text{and }\ \min_{T^{\epsilon,n}_{j-1}\le i<T^{\epsilon,n}_j} \left( X_i-X_{T^{\epsilon,n}_{j-1}} \right)  < \left( -1 + \epsilon ^ 3 \right) \fl{\epsilon\sqrt{n}}.
\end{align}
Once the coupling is broken it remains so subsequently.
\end{defn}
Note that conditions \eqref{broken2} and \eqref{broken3} ensure that the coupling is broken on the $j$-th step if the walk goes very close to the right (or left) endpoint of $\bint{X_{T^{\epsilon,n}_{j-1}}-\fl{\epsilon\sqrt{n}}, X_{T^{\epsilon,n}_{j-1}}+\fl{\epsilon\sqrt{n}}}$ but then ultimately reaches the left (or right) endpoint first. 
Therefore, if the coupling is unbroken after step $j$ then the environment in an
$\epsilon^3$-neighborhood of each integer point in the range of
$W^{\epsilon,n}$ up to time $j$ except for $W^{\epsilon,n}_j$ and
$W^{\epsilon,n}_{j-1}$ remains unchanged by the $j$-th step of the
walk $W^{\epsilon,n}$ and thus preserves any lifting or grounding
properties. Because of this, Lemma \ref{e1a} will allow us to get, with high probability, lifting and grounding properties at all sites of $\fl{\epsilon\sqrt{n}} \Z$ other than the position of the walk at time $\tenk$ (see Lemma \ref{nextbeta} below).

\subsection{Steps after the first}

We now pass to the coupling for the $k+1$-th step given that the
$k$-th step has been completed.  As already stated, if the coupling is
broken before or at step $k$, then
$(\tien_{k+1},\twen_{k+1},\tsen_{k+1})$ is chosen independently using
auxiliary uniform random variables independent of the cookie process.
So in the following we assume that the coupling is unbroken.  We note
that (unlike in the first step) %outside of an event of probability tending to zero as $n$ tends to infinity
 for $\epsilon $ fixed
$\{I^{\epsilon,n} _k ,S^{\epsilon,n} _k \}\not\subset (W^{\epsilon,n}_k
- 1, W^{\epsilon,n}_k+ 1)$.

\emph{Steps in the bulk}: We first give the coupling in the case
$\{I^{\epsilon,n} _k ,S^{\epsilon,n} _k\}\cap(W^{\epsilon,n}_k - 1,
W^{\epsilon,n}_k+ 1)=\emptyset$.  We let
\[
p^n_k  = P( W^{\epsilon,n}_{k+1} = W^{\epsilon,n}_k + 1 \mid \cal{F}^{\epsilon,n}_k ),
\]
and note that the corresponding  probability for the BMPE is exactly $\frac12$ since we are away from the extremes.  
If $W^{\epsilon,n}_{k+1} = W^{\epsilon,n}_k- 1$, then we
take
 \[(\tien_{k+1},\twen_{k+1},\tsen_{k+1}) =(I^{\epsilon,n}_k, W^{\epsilon,n}_k- 1 , S^{\epsilon,n}_k)\ \ \text{with
 probability}\ \ 1 \wedge \frac{1/2}{1-p^n_k}.\] 
If $W^{\epsilon,n}_{k+1} = W^{\epsilon,n}_k+ 1$, then we take 
\[(\tien_{k+1},\twen_{k+1},\tsen_{k+1}) =(I^{\epsilon,n}_k,W^{\epsilon,n}_k+1,S^{\epsilon,n}_k) \
  \ \text{ with probability} \ \ 1 \wedge \frac{1/2}{p^n_k}.\]
 If $(\tien_{k+1},\twen_{k+1},\tsen_{k+1})$ is undefined we use an auxiliary uniform random variable to define it so that it satisfies properties (i)-(iii) of Section~\ref{mwalk}.  Then we check if the coupling is broken (see Definition~\ref{broken}).

\emph{Steps at the boundary}: It remains to detail the coupling if
$\{I^{\epsilon,n}_k,S^{\epsilon,n}_k \}\cap (W^{\epsilon,n}_k -1,
W^{\epsilon,n}_k+1)\ne \emptyset$.  We suppose that
\[\{I^{\epsilon,n}_k,S^{\epsilon,n}_k \}\cap [W^{\epsilon,n}_k -1,
  W^{\epsilon,n}_k+1]=\{S^{\epsilon,n}_k\}\] and omit details for the
other case. For notational clarity and to emphasize the congruence
with the first step, we translate the space so that
$W^{\epsilon,n} _k=\twen_k=0$ and
$S^{\epsilon,n}_k=\tsen_k\in [0,1-\epsilon^3)$. 
 
We divide up $(-1,1 )$ into the same intervals
  \eqref{Jell} as in the first step and find $\ell_k$ such that
  $S^{\epsilon,n}_k\in J_{\ell_k}$. Note that since 
  $S^{\epsilon,n}_k<1-\epsilon^3$, we know that
  $J_{\ell_k+1}\subset(0,1)$. We shall join $J_{\ell_k}$ and
  $J_{\ell_k+1}$ to form a single interval which we shall again call
  $J_{\ell_k+1}$. Then we compute for $\ell\ge \ell_k+1$
\begin{align*}
  p^n _{k,\ell} &= P( W^{\epsilon,n}_{k+1}= -  1, \ S^{\epsilon,n}_{k+1}\in J_\ell  \mid \cal{F}^{\epsilon,n}_k  );\\
  q_{k,\ell}  &= P( W(\tau(1,-1,0,S^{\epsilon,n}_k))  = - 1,\ S(\tau(1,-1,0,S^{\epsilon,n}_k)) \in J_\ell ),
\end{align*}
where $(I,W,S)$ is a  BMPE with initial condition $(I_0,W_0,S_0) = (-1, 0,S^{\epsilon,n}_k)$. We also compute
\[
p^n_k = P(W^{\epsilon,n}_{k+1}=1 \mid \cal{F}^{\epsilon,n}_k)
\quad\text{and}\quad
q_k = P( W(\tau (1, -1,0,S^{\epsilon,n}_k))  = 1 ).
\]
If $W^{\epsilon,n}_{k+1}=W^{\epsilon,n}_k+1$, then we take 
$\twen_{k+1}=\twen_k+1$ with probability $ 1 \wedge \frac{q_k}{p^n_k}$, and in this case
\[
  (\tien_{k+1}, \twen_{k+1}, \tsen_{k+1})=(I^{\epsilon,n}_k,
  W^{\epsilon,n}_k+1, W^{\epsilon,n}_k+1).
\]
If $W^{\epsilon,n}_{k+1}=W^{\epsilon,n}_k-1$ and
$S^{\epsilon,n}_{k+1}\in J_\ell $, $\ell\in\bint{\ell_k+1,L}$, then
with probability $1 \wedge \frac{q_{k,\ell}}{p^n_{k,\ell}}$, we take
\[
 (\tien_{k+1},\twen_{k+1},\tsen_{k+1}) = (I^{\epsilon,n}_k , W^{\epsilon,n}_k-1, S^{\epsilon,n}_{k+1}).
\]
If at this point $(\tien_{k+1},\twen_{k+1},\tsen_{k+1})$ is undefined, we use the auxiliary independent randomness in a similar manner as on the first step. 
That is, we let $ (\tien_{k+1},\twen_{k+1},\tsen_{k+1})$ equal
\[
 %= (\tien_{k+1},\twen_{k+1},\tsen_{k+1}) =
 (I^{\epsilon,n}_k, W^{\epsilon,n}_k - 1,\frac{\ell-\frac{1}{2}}{L}),
\quad \text{with probability } \frac{(q_{k,\ell} - p_{k,\ell}^n)_+}{(q_k-p_k^n)_+ + \sum_{0<|\ell'|\leq L} (q_{1,\ell'}-p_{1,\ell'}^n)_+}. 
\]
 for $\ell\in\bint{\ell_k+1,L}$ and 
\[
 %(\tien_{k+1},\twen_{k+1},\tsen_{k+1}) =
 (I^{\epsilon,n}_k, W^{\epsilon,n}_k + 1, W^{\epsilon,n}_k + 1 ), \quad \text{with probability } \frac{(q_k - p_k^n)_+}{(q_k-p_k^n)_+ + \sum_{0<|\ell'|\leq L} (q_{1,\ell'}-p_{1,\ell'}^n)_+}. 
\]

\subsection{The Coupling Theorem} Having constructed the coupling, we
can now state the main result of this section.
\begin{thm}\label{close}
  For every $K>0$ 
  \[ \lim_{\epsilon\to 0} \liminf_{n\to \infty}P\left(\inf\{k\ge 0:\ {\cal X}^{\epsilon,n}_k\ne\txen_k\}\ge\frac{K}{\epsilon^2}\right)=1.\]
\end{thm}

\begin{proof}
  Fix an arbitrary $K>0$. We shall make a list of conditions on the
  ERW path and on the first cookie environments at each mesoscopic
  step which will ensure that the coupling is preserved with high
  probability. These conditions involve two additional parameters
  $\delta_1,\delta_2>0$ which will depend only on $\epsilon$ and which
  we shall choose later. For now it is enough to say that 
    $\delta_1$ and $\delta_2$ are chosen so that
    $\delta_1 \overset{\epsilon}{\sim} \delta_2$ and are small enough
    so that Lemmas \ref{hitpr}, \ref{e1a} and \ref{Away} can be
    applied. 
%Note that to apply Lemma \ref{hitpr} we need
 %   $(\delta_1,\delta_2,\fl{\epsilon \sqrt{n}})$ to be as
 %   $(\delta_1,\delta,m)$ in Lemma \ref{hitpr}.) Elena: I suggest to
 %   use $\delta_2$ instead of $\delta$ in Lemma \ref{hitpr} to make
  %  this notation consistent throughout the paper.} 
 The conditions to
  be satisfied at each step $j$ are as follows.
  \begin{enumerate}[(Ei)]
  % \item Upon completion of step $k$ the coupling remains unbroken (see
  %   \eqref{broken1}-\eqref{broken3}).
  \item The interval
    $\bint{ \fl{\epsilon\sqrt{n}}(W^{\epsilon,n}_j-1),\fl{\epsilon\sqrt{n}}(W^{\epsilon,n}_j-1)+\fl{\epsilon \delta_1 \sqrt{n}}}$
    is $\delta_2 \epsilon\sqrt{n}$-lifting from the left and $\delta_2 \epsilon\sqrt{n}$-grounding from
    the right.
  \item The interval
    $\bint{\fl{\epsilon\sqrt{n}} (W^{\epsilon,n}_j+1)-  \fl{\epsilon \delta_1 \sqrt{n}},\fl{\epsilon\sqrt{n}} (W^{\epsilon,n}_j+1)}$
    is $\delta_2 \epsilon\sqrt{n}$-lifting from the right and $\delta_2 \epsilon\sqrt{n}$-grounding from
    the left.
 \item The first cookie environment is $(n^{1/8},\nu/2-1)$-good on
$\bint{\fl{\epsilon\sqrt{n}}(I^{\epsilon,n}_j\vee (W^{\epsilon,n}_j-1)) ,\fl{\epsilon\sqrt{n}}W^{\epsilon,n}_j}$
 and is $(n^{1/8},0)$-good on $\bint{\fl{\epsilon\sqrt{n}}W^{\epsilon,n}_j, \fl{\epsilon\sqrt{n}}( (W^{\epsilon,n}_j+1)\wedge S^{\epsilon,n}_j) }$. 
%   $\fl{\epsilon\sqrt{n}}[I^{\epsilon,n}_j\vee (W^{\epsilon,n}_j-1),W^{\epsilon,n}_j]$, and is
%     $(n^{1/8},0)$-good on $\fl{\epsilon\sqrt{n}}[W^{\epsilon,n}_j, (W^{\epsilon,n}_j+1)\wedge S^{\epsilon,n}_j]$. 
   \end{enumerate}
   We remark that for $j=0$ the condition (Eiii) is vacuous, and we
   shall agree that it automatically holds.  Let
  \[\beta^{\epsilon,n}\coloneqq \inf\{j\ge 0:\
    \text{at least one of conditions (Ei)-(Eiii) above 
      does not hold for }j\}\]
  and $\tau^{\epsilon,n}$ be the step at which the coupling breaks down, i.e.\ 
  \[\tau^{\epsilon,n}\coloneqq \inf\{j\ge 1:\ \text{at least one of
      \eqref{broken1}-\eqref{broken3} does not hold for }j\}.\] We
  have also agreed that at time $0$ the coupling is unbroken.  
  This
  implies that $\tau^{\epsilon,n}>0$. Using this notation we can say
  that
% With this notation,
  \begin{equation}\label{lb1}
    P\left(\inf\{k\ge 0:\ {\cal X}^{\epsilon,n}_k\ne\txen_k\}>\frac{K}{\epsilon^2} \right)\ge P\left( \tau^{\epsilon,n}>\frac{K}{\epsilon^2},\,
      \beta^{\epsilon,n}>\frac{K}{\epsilon^2}\right).
  \end{equation}
  
  As a first step toward controlling the probability on the right, we
  need the following lemmas
\begin{lemma}\label{nextbeta}
  There exists a constant $\Cl{ze}>0$ such that for every $\epsilon>0$ and $n$ large enough that 
  \[
   P\left(\text{Conditions (Ei) and (Eii) hold for $j=k$} \mid \tau^{\epsilon,n} > k-1, \beta^{\epsilon,n} > k-1\right) 
   \geq 1 - \Cr{ze} \epsilon^3, 
  \]
 for all $k\geq 1$. 
\end{lemma}
\begin{proof}
If $\tau^{\epsilon,n}>k-1$ and $\beta^{\epsilon,n} > k-1$, then the remaining first cookie environment on $\llbracket X_{T^{\epsilon,n}_{k-1}}-\fl{\epsilon \sqrt{n}}, X_{T^{\epsilon,n}_{k-1}}-\fl{\epsilon \sqrt{n}} \rrbracket$ is ``regular'' as defined in the Definitions \ref{bulkreg} and \ref{exreg}. Thus the conclusion of Lemma \ref{nextbeta} follows directly from Lemma \ref{e1a}.
\end{proof}

 \begin{lemma}\label{nexttau}
 There exists a constant $\Cl{on}>0$ such that for every $\epsilon>0$ there is an $n_0=n_0(\epsilon,\delta_1(\epsilon),\delta_2(\epsilon))$ such that for all  $n\geq n_0$ 
and all $k\in\N$
  \begin{equation}
   P\left(\tau^{\epsilon,n}>k\mid \beta^{\epsilon,n}>k-1,\,\tau^{\epsilon,n}>k-1 \right)\ge 1-\Cr{on} \epsilon^{2.9}. \label{c}
  \end{equation}
 \end{lemma}

\begin{proof}
 The validity of inequality \eqref{c} has to be checked for three
  different cases:
  \begin{enumerate}[(1)]
\item the first step,
  i.e.\ $k=1$; 
\item  $k\ge 2$ and
  $\{I^{\epsilon,n}_{k-1},S^{\epsilon,n}_{k-1}\}\cap[W^{\epsilon,n}_{k-1}-1,W^{\epsilon,n}_{k-1}+1]=\emptyset$;
\item  $k\ge 2$ and
  $\{I^{\epsilon,n}_{k-1},S^{\epsilon,n}_{k-1}\}\cap[W^{\epsilon,n}_{k-1}-1,W^{\epsilon,n}_{k-1}+1]\ne
  \emptyset$.
\end{enumerate}

\noindent{\em Case (1).} Let $k=1$. Recall that $\tau^{\epsilon,n}>0$
and we start with ${\cal X}^{\epsilon,n}_0=\txen_0=(0,0,0)$ and an
i.i.d.\ cookie environment with the marginal distribution $\eta$. The
probability that the coupling breaks down at the first step is bounded
above by
\begin{equation*}
\frac12\max_{0<|\ell|\le L}|q_{1,\ell}-p^n_{1,\ell}|+P(W^{\epsilon,n}_1=1,\, I^{\epsilon,n}_1<-1+\epsilon^3)+P(W^{\epsilon,n}_1=-1,\, S^{\epsilon,n}_1 > 1- \epsilon^3).
\end{equation*}
We started with a product measure, so all conditions of the
concatenation lemma (Lemma~\ref{hitpr}) and Lemma \ref{Away}
hold. Therefore, there exists an
$n_1 = n_1(\epsilon,\delta_1(\epsilon),\delta_2(\epsilon))$ such that
all the terms above are bounded by a constant multiple of
$\epsilon^{2.9}$ for $n\geq n_1$.  Thus, \eqref{c} is satisfied in
this case with $n_0=n_1$ for some $\Cr{on}=C_{6,1}$.

\medskip

\noindent{\em Case (2).} Let $k\ge 2$, $\beta^{\epsilon,n}>k-1$,
$\tau^{\epsilon,n}>k-1$, and
$\{I^{\epsilon,n}_{k-1},S^{\epsilon,n}_{k-1}\}\cap[W^{\epsilon,n}_{k-1}-1,W^{\epsilon,n}_{k-1}+1]=\emptyset$. Then the probability that the coupling breaks down at step $k$ does not exceed 
\[
\left|p^n_k-\frac12\right|
+P(W^{\epsilon,n}_k-W^{\epsilon,n}_{k-1}=1,\, I^{\epsilon,n}_k - W^{\epsilon,n}_{k-1} < - 1 + \epsilon^3)
+P(W^{\epsilon,n}_k-W^{\epsilon,n}_{k-1}=-1,\, S^{\epsilon,n}_k - W^{\epsilon,n}_{k-1} > 1 - \epsilon^3).
\]
Since the coupling hasn't broken by the $(k-1)$-th step, the remaining first cookie environment in $\bint{X_{T^{\epsilon,n}_{k-1}} - \fl{\epsilon\sqrt{n}}, X_{T^{\epsilon,n}_{k-1}} + \fl{\epsilon\sqrt{n}} }$ satisfies the conditions of the Lemma~\ref{hitpr} and \ref{Away}. Thus, the above sum does not exceed $C_{6,2}\epsilon^{2.9}$ for all $n\ge n_2(\epsilon,\delta_1(\epsilon),\delta_2(\epsilon))$, where $C_{6,2}$ does not depend on either $\epsilon$ or $k\ge 2$. 

\medskip

\noindent{\em Case (3).} Let $k\ge 2$, $\beta^{\epsilon,n}>k-1$,
$\tau^{\epsilon,n}>k-1$, and
$\{I^{\epsilon,n}_{k-1},S^{\epsilon,n}_{k-1}\}\cap[W^{\epsilon,n}_{k-1}-1,W^{\epsilon,n}_{k-1}+1]=\{S^{\epsilon,n}_{k-1}\}$. The other case is symmetric and we shall not give details.

Under the above assumptions, the probability that the coupling breaks
down at step $k$ does not exceed
\begin{multline*}
  \frac12|p^n_k-q_k|+\frac12\max_{\ell_k+1\le \ell\le L}|q_{k,\ell}-p^n_{k,\ell}|\\ 
+P(W^{\epsilon,n}_k-W^{\epsilon,n}_{k-1}=1,\, I^{\epsilon,n}_k - W^{\epsilon,n}_{k-1} < - 1 + \epsilon^3)
+P(W^{\epsilon,n}_k-W^{\epsilon,n}_{k-1}=-1,\, S^{\epsilon,n}_k - W^{\epsilon,n}_{k-1} > 1 - \epsilon^3).
\end{multline*}
%\begin{multline*}
%  \frac12|p^n_k-q_k|+\frac12\max_{\ell_k+1\le \ell\le L}|q_{k,\ell}-p^n_{k,\ell}|\\ +P(W^{\epsilon,n}_k-W^{\epsilon,n}_{k-1}=1,\,\dist(I^{\epsilon,n}_k,\Z)<\epsilon^3)+P(W^{\epsilon,n}_k-W^{\epsilon,n}_{k-1}=-1,\,\dist(S^{\epsilon,n}_k,\Z)<\epsilon^3).
%\end{multline*}
Again, since the coupling has not yet been broken we can apply Lemmas \ref{hitpr} and \ref{Away} to conclude that this
sum does not exceed $C_{6,3}\epsilon^{2.9}$
 for all
$n\ge n_3(\epsilon,\delta_1(\epsilon),\delta_2(\epsilon))$, where
$C_{6,3}$ does not depend on either $\epsilon$ or $k\ge 2$.

This completes the proof of \eqref{c} with $\Cr{on} = \max\{C_{6,1},C_{6,2},C_{6,3}\}$
and $n_0 = \max\{n_2,n_3\}$ for $k\geq 2$. 
\end{proof}

We will next use Lemmas \ref{nextbeta} and \ref{nexttau} obtain a lower bound on \eqref{lb1}. In particular, we will show that for every $\epsilon >0$ and every $n\geq n_0' = n_0'(\epsilon,\delta_1(\epsilon),\delta_2(\epsilon),k)$ we have 
\begin{equation}\label{geolb}
 P\left( \tau^{\epsilon,n}>k, \, \beta^{\epsilon,n}>k \right) 
 \geq (1-\Cl{two} \epsilon^{2.9})^{k+1}, \qquad \text{where } \Cr{two} = \Cr{ze}+\Cr{on}+1. 
\end{equation}
We will prove \eqref{geolb} by induction. 

\noindent{\em Base case: $k=0$}. Since $\tau^{\epsilon,n}>0$ by definition and since (Eiii) is vacuous at step 0, we need only check that conditions (Ei) and (Eii) hold. Using Lemmas \ref{liftIID} and \ref{sm0}, and choosing $\delta_1$ and $\delta_2$ appropriately (depending on $\epsilon$) we have that $P\left( \tau^{\epsilon,n}>0, \, \beta^{\epsilon,n}>0 \right) > 1-\epsilon^3$ for all $n$ large. 

\noindent{\em Induction step: $k\geq 1$}. We will assume that \eqref{geolb} holds for $k-1$. 
Next, first of all that 
\begin{align}
 &P\left( \tau^{\epsilon,n}>k, \, \beta^{\epsilon,n}>k \right) \nonumber \\
 &= P\left( \tau^{\epsilon,n}>k-1, \, \beta^{\epsilon,n}>k-1 \right) P\left( \tau^{\epsilon,n}>k, \, \beta^{\epsilon,n} > k \mid  \tau^{\epsilon,n}>k-1, \, \beta^{\epsilon,n}>k-1 \right) \nonumber \\
 &\geq P\left( \tau^{\epsilon,n}>k-1, \, \beta^{\epsilon,n}>k-1 \right) \nonumber \\
  &\quad \times \left\{ P\left( \tau^{\epsilon,n} > k \mid \tau^{\epsilon,n}>k-1, \, \beta^{\epsilon,n}>k-1 \right) + P\left(\beta^{\epsilon,n}>k \mid \beta^{\epsilon,n}>k-1, \tau^{\epsilon,n} > k-1 \right) - 1 \right\} \nonumber \\
  &\geq (1-\Cr{two} \epsilon^{2.9})^{k} \left\{ - \Cr{on} \epsilon^{2.9} + P\left(\beta^{\epsilon,n}>k \mid \beta^{\epsilon,n}>k-1, \tau^{\epsilon,n} > k-1 \right) \right\}, \label{induc1}
\end{align}
where the last inequality holds by the induction assumption and Lemma
\ref{nexttau} for $n$ large enough (depending on $\epsilon$ and $k$). 
  For the last probability in the braces on the
right, Lemma \ref{nextbeta} controls the conditional probability that
conditions (Ei) and (Eii) hold and Corollary \ref{18good-cor} controls
the (unconditional) probability that condition (Eiii) holds. More
precisely, since the event $A^{\epsilon,n}_k$ in the statement of
Corollary \ref{18good-cor} implies that condition (Eiii) holds then
for $n$ large enough
\begin{align*}
 P\left(\beta^{\epsilon,n}>k \mid \beta^{\epsilon,n}>k-1, \tau^{\epsilon,n} > k-1 \right) 
 &\geq 1 - \Cr{ze} \epsilon^3 - P\left( (A^{\epsilon,n}_k)^c \mid \beta^{\epsilon,n}>k-1, \tau^{\epsilon,n} > k-1 \right) \\ 
 &\geq 1 - \Cr{ze} \epsilon^3 - \frac{ P\left( (A^{\epsilon,n}_k)^c \right) }{ P\left( \beta^{\epsilon,n}>k-1, \tau^{\epsilon,n} > k-1 \right)  } \\
 &\geq 1 - \Cr{ze} \epsilon^3 - \frac{ P\left( (A^{\epsilon,n}_k)^c \right) }{ (1-\Cr{two}\epsilon^3)^{k}  }, 
\end{align*}
and since $  P\left( (A^{\epsilon,n}_k)^c \right) \to 0$ as $n\to \infty$ (by Corollary \ref{18good-cor}) it follows that the right side is larger than 
$(1-(\Cr{ze}+1)\epsilon^{3})\geq (1-(\Cr{ze}+1)\epsilon^{2.9})$ for $n$ large enough (again depending on $k$ and $\epsilon$). 
Applying this to \eqref{induc1} finishes the proof of \eqref{geolb}. 

Finally, applying \eqref{geolb} to \eqref{lb1} we obtain that 
\[
 \lim_{\epsilon\to 0}\liminf_{n\to\infty}  P\left(\inf\{k\ge 0:\ {\cal X}^{\epsilon,n}_k\ne\txen_k\}>\frac{K}{\epsilon^2} \right)
 \geq \lim_{\epsilon\to 0} (1-\Cr{two}\epsilon^{2.9})^{\fl{K/\epsilon^2}+1} = 1. 
\]
%\begin{align*}
%  P\left( \tau^{\epsilon,n}>\frac{K}{\epsilon^2},\, \beta^{\epsilon,n}>\frac{K}{\epsilon^2}  \right)
%  \geq (1-\Cr{two}\epsilon^3)^{\fl{K/\epsilon^2}+1} 
%  \geq(1-2\Cr{two} K\epsilon)
%  \end{align*}
This completes the proof of Theorem \ref{close}.
\end{proof}

\section{Time control and the proof of Theorem~\ref{main}}\label{sec:fin}

The previous section established that the embedded process
$\{ W^{\epsilon,n}_k \}_{k \geq 0}$ is close to a modified
BMPE walk. From Section~\ref{sec:disc} we know that modified walks
converge to BMPE. To complete the proof of Theorem \ref{main} we just
have to show a law of large numbers for the variables
$T^{\epsilon , n} _k,\ k\ge 0$.  
\begin{lemma}
  \label{time}
  For each $K,h > 0$ there is an $\epsilon_0 > 0 $ such that for all $\epsilon < \epsilon_0 $ and all $ n \geq n_0( \epsilon) $
\[
P \left( \sup_{k<\epsilon^{-2}K}  \left|T^{\epsilon, n} _k -  \frac{\nu}{2}\,k n\epsilon^2 \right|    >  h n  \right)  <  h.
\]
\end{lemma}
Let us assume for the moment that this lemma holds and give a proof of Theorem~\ref{main}.

\begin{proof}[Proof of Theorem~\ref{main}]
% By  Corollary~\ref{dfccor} there exists a family of modified BMPE-walks
%  $({\tilde I^ \epsilon }_k ,{\tilde W^ \epsilon }_k, {\tilde S^
%    \epsilon }_k)_{k \geq 0}$, there exists a continuous time
%  BMPE $ (W^\epsilon(t))_{t \geq 0} $ so that
% \begin{equation} \label{convv}
% \forall \delta, T>0 \qquad P\left(\sup_{t \leq T }  \big| \epsilon {\tilde W}^\epsilon_{ \fl{\epsilon^{-2}t}}  -  W^\epsilon(t)  \big|>\delta\right)  \rightarrow  0\quad\text{as }\ \ \epsilon\to 0.
% \end{equation}
From Theorem~\ref{close} we know that for all $K,\delta>0$ there is an $\epsilon_0>0$ such that for all $\epsilon<\epsilon_0$ and all sufficiently large $n$, with probability at least $1-\delta$ we have
\[
(I^ {\epsilon,n}_k  ,W^ {\epsilon,n} _k, S^  {\epsilon,n}_k)=(\tien_k  ,\twen_k, \tsen_k),\quad 0\le k< \epsilon^{-2}K,
\]
where $(\tien_k  ,\twen_k, \tsen_k)$ is a modified BMPE-walk.  This and Corollary \ref{dfccor} imply that there
exists a family of BMPEs $\{ W^{n , \epsilon } (t) \}_{t \geq 0}$ such
that for all $T,\delta > 0$ there exists $\epsilon_0 > 0 $ such that
\begin{equation}
  \label{cp}
  \forall \epsilon\in(0,\epsilon_0)\ \  \exists n_0(\epsilon) \ \ \text{such that }\quad P \left( \sup_{t \leq T }  \big| \epsilon W^{\epsilon ,n}_{ \fl{\epsilon^{-2}t}}  -  W^{n ,\epsilon }(t)  \big|   >  \delta  \right)  <  \delta\quad\text{for all }n\ge n_0(\epsilon).
\end{equation}
To complete the proof it is enough to replace
$\epsilon W^{\epsilon ,n}_{ \fl{\epsilon^{-2}t}}$ with
$\hat{W}^{\epsilon , n}(t):= \epsilon W^{\epsilon , n}_{k_t}$ where $k_t=k_t(n,\epsilon)$ is such that
$ T^{\epsilon, n} _{k_t}\leq \frac{\nu}{2} tn < T^{\epsilon,
  n} _{k_t+1} $.  Indeed, for all large $n$ the process $X_{\fl{tn\nu/2}}/\sqrt{n}$ always stays within $\epsilon$ of $\epsilon X_{T^{\epsilon,n}_{k_t}}/\fl{\epsilon\sqrt{n}}=\epsilon W^{\epsilon,n}_{k_t}$, and if we know that
\begin{equation}
  \label{hat}
  P \left( \sup_{t \leq T }  \big| \hat{W}^{\epsilon , n}(t)  -  W^{n ,\epsilon }(t)  \big|   >  \delta  \right)  <  \delta\quad\text{for all }n\ge n_0(\epsilon),
\end{equation}
then we have the convergence claimed in Theorem~\ref{main}.\footnote{Note that we are also using here the fact that $2/\nu = 1-\theta^+-\theta^-$ to get the scaling constant as in the statement of Theorem \ref{main}.} To see
that \eqref{hat} holds we simply note that
\[ \big| \hat{W}^{\epsilon , n}(t) - W^{n ,\epsilon }(t)
  \big|=|\epsilon W^{\epsilon,n}_{k_t}-W^{n,\epsilon}(t)|\le
  \big|\epsilon W^{\epsilon,n}_{k_t} - W^{n ,\epsilon }(k_t\epsilon^2)
  \big|+\big| W^{n,\epsilon}(k_t\epsilon^2)- W^{n ,\epsilon }(t)
  \big|,\] where both terms in the right hand side are controlled by \eqref{cp},
Lemma~\ref{time}, and path continuity of
BMPE.
\end{proof}

\begin{proof}[Proof of Lemma~\ref{time}]
  Just as in the proof of Proposition \ref{DiscBMPE} we argue that the
  increments in the bulk are dominant and increments at extremes are
  negligible.  Thus in analyzing the bulk increments we must be more
  precise, whereas a reasonable bound on increments at the extremes
  will meet our purpose.

  To improve legibility, we drop $\epsilon $ and $n$ from the notation
  and write
  $H_i = T^{\epsilon, n} _i - T^{\epsilon, n} _{i-1} , i \in \N$.  We
  wish to use the law of large numbers for i.i.d. random variables but
  the $\{ H_i / n \}_{i \geq 1 } $ are neither identically distributed
  nor independent (even in the limit as $n$ tends to infinity).  As a
  first step to address this, we separate out the $H_i $ according to
  whether the walk is in the bulk or at an extreme at time
  $ T^{\epsilon, n}_{i-1}$.  Accordingly, we set ${\cal B}$ as the set
  of indices $i< K/\epsilon^2 $ such that
  $I_{T^{\epsilon, n}_{i-1}} + \fl{ \epsilon \sqrt{n} } \leq
  X_{T^{\epsilon, n}_{i-1}} \leq S_{T^{\epsilon, n}_{i-1}} - \fl{
    \epsilon \sqrt{n} }$, set ${\cal S}$ to be those
  $i < K/\epsilon^2 $ for which
  $ S_{T^{\epsilon, n}_{i-1}} < X_{T^{\epsilon, n}_{i-1}} + \fl{
    \epsilon \sqrt{n} }$, and ${\cal I}$ for the remainder, that is
  those $i<K/\epsilon^2 $ for which
  $I_{T^{\epsilon, n}_{i-1}}> X_{T^{\epsilon, n}_{i-1}} -\fl{ \epsilon
    \sqrt{n} } $ % is close to the minimum
  .

  The random variables $\{ H_i/(n \epsilon^2)\}_{i \in {\cal B}}$ are
  still not proven to be i.i.d., even in a limit as $n$ tends to
  infinity.  But they are ``close'' to i.i.d. random variables whose
  law is that of the time for a variance $2/\nu$ Brownian motion,
  starting at $0$, to leave $(-1,1)$.

  To show our convergence it will be enough to show that
  $ \forall K, h\in(0,\infty)$ there exists $\epsilon_0 > 0 $ such
  that for all $\epsilon\in(0, \epsilon _0)$ and
  $n\geq n_0(\epsilon)$
 \begin{align}
   \label{num_ex}
     P &\left(   | {\cal{S} } \cup {\cal{I}} | > \frac{2  h}{\nu \epsilon^2}  \right) <  h \\
   \label{time_ex}
   P &\left( \sum_{i \in {\cal{S} } \cup {\cal{I}} } H_{i}   >   h n \right)  <   h,\quad\text{and}\\
   \label{time_b}
   P &\left(  \sup_{k<  \epsilon^{-2}K}  \bigg\vert \sum_{j\in \bint{1,k} \cap {\cal B} } \left( H_{j}  - \frac12\nu \epsilon ^ 2 n\right) \bigg\vert >   h n \right)  <   h.
 \end{align}
 We expect \eqref{num_ex} to hold since it should be the case that a
 negligible fraction of steps for the embedded process
 $W^{\epsilon,n}_k$ are at the extremes. To make this precise, note
 that with high probability using Theorems \ref{dfc} and \ref{close}
 we can couple the embedded process with a BMPE walk.  Then
 \eqref{num_ex} follows by showing that almost surely a BMPE walk
 spends a negligible fraction of time at its extremes. This fact about
 BMPE walks is proved in \eqref{rext} in the Appendix. It remains now
 to prove \eqref{time_ex} and \eqref{time_b}.

{\em Step 1.} We begin with \eqref{time_ex}.  We will show the inequality
 % $ \forall K < \infty,\  h > 0,\ \exists \epsilon _ 0 > 0 $ so that
 % for $\epsilon < \epsilon _0 $ and all $ n \geq n_0( \epsilon) $
 \[
 P \bigg( \sum_{i\in\cal{S} }  H_{i}   >   h  n\bigg)  <   h.
 \]
 The analogous inequality with $\cal{S} $ replaced by $\cal{I}$
   is proved similarly and so is not explicitly treated.  % In this
 % part $i$ will denote an index in $\cal{S}$.
 
 As before we introduce BLP $\{ Z^i_k\}_{k \geq 0}$ where $Z^i_k$ is
 the number of jumps from
 $X_{T^{\epsilon, n}_{i-1}} - \fl{\epsilon \sqrt{n} } + k $ to
 $X_{T^{\epsilon, n}_{i -1}} - \fl{\epsilon \sqrt{n} } + k +1$ in time
 interval $(T^{\epsilon, n}_{i -1}, T^{\epsilon,n}_{i,-}]$ where
$T^{\epsilon,n}_{i,-}$ is defined in \eqref{tenk}.
 % where
 % $T^{\epsilon,n}_{i,-} = \inf \{k >T^{\epsilon, n}_{i-1 }: X_k =
 % X_{T^{\epsilon, n}_{i -1}} - \fl{\epsilon \sqrt{n} } \}$.
  (So with
 reasonable probability $T^{\epsilon,n}_{i,-} = T^{\epsilon, n}_{i } $
 and with reasonable probability it is definitely larger.)  If
 $T^{\epsilon,n}_{i,-} = T^{\epsilon, n}_{i } $ then
\begin{equation} \label{EQ2}
H_{i}  =   T^{\epsilon, n}_i  -   T^{\epsilon, n}_{i-1}  =  2 \sum _{k=1}^{2  \fl{\epsilon \sqrt{n}  }} Z^i_k  + \fl{\epsilon \sqrt{n}  },
 \end{equation}
 otherwise it is less than the right-hand side.
So it is enough to show that for $\epsilon < \epsilon _ 0 $ and $n > n_0(\epsilon) $
\begin{equation} \label{EQ3} P \bigg( \sum_{i\in \cal{S} } \sum _{k=1}
  ^{2 \fl{\epsilon \sqrt{n} }} Z^{i}_k >  h n \bigg) <  h.
 \end{equation}
 We write
 $\Sigma_i:= \sum _{k=1}^{2 \fl{\epsilon \sqrt{n} }} Z^i_k$,
 and we will use the trivial inequality
  \[
 \Sigma_i \le   2\fl{\epsilon\sqrt{n}}\max_{1\le k\le 2  \fl{\epsilon \sqrt{n}  }} Z^i_k,
 \]
 together with Corollary \ref{C0} to bound $\Sigma_i $ above.
For any $K'\in(0,\infty) $ and $
 \epsilon > 0 $, Corollary \ref{C0} implies that for $i \in \cal{S}$ (and $n$ sufficiently large), given $\cal{F}_{T^{\epsilon, n}_{i -1}} $, on the set that the coupling has not been broken,
 \[
 \Sigma_i \ind{\Sigma_i \le 2 K' \fl { \epsilon \sqrt{n} } ^2 } 
 \ \ \text{ 
 is stochastically dominated by }\ 
 2 \fl { \epsilon \sqrt{n} } ^2 \zeta\] where
 $ P(\zeta\ge x ) = 2\Cr{d} e ^{-\Cr{sd} x /2} \wedge 1$.  It is
 important to note that the law of $\zeta$ does not depend on
 $ \epsilon $ or $K'$.  Given $h > 0$, we fix $K'$ so that
 \[
 2\Cr{d} e ^{-\Cr{sd} K' /2} < \frac{\epsilon ^ 2  h}{ 4K}.
 \]
 Then for $n \ge n_0(K')$ (by Corollary \ref{C0} ) we have
 \begin{equation} \label{EQA}
   P \left( \forall i \in \cal{S}: \ \Sigma_i = \Sigma_i \ind{\Sigma_i \le 2 K' \fl { \epsilon \sqrt{n} } ^2}  \right) \ge 1 -  \frac{h}{4}.
\end{equation}
We choose $ \alpha > 0$ so that $4\alpha K E[\zeta]<  h$.  By the
weak law of large numbers there exists $\epsilon _ 1 > 0 $ so that for
all $N\geq \alpha K/ \epsilon_1^2 $,
 \begin{equation} \label{EQB}
 P \left(  \frac{1}{N}  \sum_{j=1}^N \zeta_j  > 2 E[\zeta]\right) <  \frac{h}{4},\ \ \text{ where the $\zeta_j $ are i.i.d. copies of $\zeta$.}
 \end{equation}
 As noted in the proof of \eqref{num_ex} above, it follows from
 Theorems \ref{dfc}, \ref{close} and \eqref{rext} in the Appendix,
 that there exists $\epsilon_2 > 0$ so that for
 $\epsilon < \epsilon _ 2 $ and $n \ge n_0(\epsilon ) $,
 \begin{equation} \label{EQC}
 P(|\cal{S}| > \alpha K/ \epsilon ^ 2) <  \frac{h}{4}.
\end{equation}
 Finally, let $\epsilon_3  > 0 $ be such that for $\epsilon  < \epsilon _ 3 $, the probability that the coupling breaks down before time $K/ \epsilon ^ 2 $ is less than $  h / 4 $ for $n$ sufficiently large.
 
We are now ready to prove inequality \eqref{time_ex}.  Choose
 $ \epsilon _0 < \epsilon_1 \wedge \epsilon _ 2 \wedge \epsilon _ 3 $.
 Given $ \epsilon < \epsilon _ 0 $ we have $n_0 = n_0( \epsilon ) $ so
 that for $n \ge n_0$
 \begin{enumerate}[(i)]
\item $P( \mbox{coupling breaks down before } K/ \epsilon ^ 2 ) <  h / 4 $;
\item $P \left( \exists i \in \cal{S}: \Sigma_i > 2K'  \fl{ \epsilon \sqrt{n} }^2 \right) <  h / 4 $;
\item $P( |\cal{S}| > \alpha K / \epsilon ^ 2 ) <  h / 4$.
\end{enumerate}
% We also note that, provided the coupling has not broken down,  for $n$ large 
% the conditional stochastic domination of $ \Sigma_i \ind{\Sigma_i \le 2 K' \fl { \epsilon \sqrt{n}}^2 }$ by $\zeta$ holds for $i \in  \cal{S}$.
Then for $n \ge n_0$ outside probability $ h/4$ by (ii) we have 
\[
\sum _{i \in \cal{S}} \Sigma_i = \sum _{i \in \cal{S}} \Sigma_i \ind{\Sigma_i \le 2K'  \fl{ \epsilon \sqrt{n} }^2},
\]
which (if the coupling has not broken down) is stochastically
dominated by
$2 \fl { \epsilon \sqrt{n} } ^2 \sum_{j=1} ^ {|\cal{S}|} \zeta_j $
where $\{\zeta_j\}_{j\ge 1}$ is a sequence of independent copies of
$\zeta$.  By (iii), outside of a further set of probability
$ h / 4 $, the last expression is bounded stochastically by
$2 \fl { \epsilon \sqrt{n} } ^2 \sum_{j=1} ^ {\alpha K / \epsilon ^ 2}
\zeta_j $.  Finally by \eqref{EQB}, excluding a final set of
probability $ h / 4 $, we have that this sum is less than
\[4 \alpha K \epsilon ^{-2} E\zeta\fl{ \epsilon \sqrt{n} }^2 \le  h n
\] by our choice of $\alpha $.  This completes Step 1.

{\em Step 2.} We now turn to the inequality \eqref{time_b}.  % It is
% enough to show that for each $  h > 0 ,\ \exists \epsilon_0$ so
% that for each $ \epsilon < \epsilon_0 ,\ \exists n_0( \epsilon )$ so
% that for all $n \geq n_0$
% \[
%  P \bigg( \sup_{k< K/ \epsilon ^ 2} \Big\vert  \sum_{i \in \bint{1,k} \cap \cal{B}} \bigg(\frac{ H_{i} }{n}  -  \frac{\nu \epsilon ^ 2}{2} \Big) \bigg\vert  >   h  \bigg)  <   h.
% \]
We write $H_i=H_{i,-}+H_{i,+}$ where
$H_{i,-}= H_i\ind{T^{\epsilon, n}_i=T^{\epsilon,n}_{i,-}} $.  It is
enough to show that for all $ K, h > 0 ,\ \exists \epsilon _ 0$
so that for each $ \epsilon < \epsilon_0 ,\ \exists n_0=n_0( \epsilon )$
so that 
\begin{equation*} P \bigg( \sup _{k< K/ \epsilon ^ 2}
  \Big\vert \sum_{i \in \bint{1,k} \cap \cal{B}} \left(H_{i,\pm}-
  \nu \epsilon^2n/4 \right) \Big\vert >  h n \bigg) <  h\quad\text{for all $n \geq n_0$}.
\end{equation*}
As the proofs are identical, we just treat the sum of $H_{i,-}$.

As in Step 1, for $k\in\bint{0,2\fl{\epsilon\sqrt{n}}}$
we define  $ Z^i_k$ as the number of jumps from
$X_{T^{\epsilon, n}_{i -1}} - \fl{ \sqrt{n} \epsilon} +k $ to
$X_{T^{\epsilon, n}_{i -1}} - \fl{ \sqrt{n} \epsilon} +k +1$ in time
interval $(T^{\epsilon, n}_{i -1} , T^{\epsilon,n}_{i,-} ]$.  
We note that if $Z^i_k = 0$ for some
$k \in \bint{\fl{ \sqrt{n} \epsilon}, 2 \fl{ \sqrt{n} \epsilon} -1}$,
then $T^{\epsilon,n}_{i,-} = T^{\epsilon, n}_{i } $ and we have
\[
H_i  =  H_{i,-} =  T^{\epsilon, n}_i - T^{\epsilon, n}_{i -1} = 2 \sum_{k=0} ^ {2 \fl{ \epsilon \sqrt{n}} -1} Z^i_k  +  \fl{\epsilon \sqrt{n}}.
\]
%and conversely.
Or, restating,
$H_{i,-} = \left( 2 \sum_{k=0} ^ {2 \fl{\epsilon \sqrt{n} } -1} Z^i_k
  + \fl{\epsilon \sqrt{n}} \right) \ind {\sigma^{Z^i} _{\fl{\epsilon \sqrt{n}
      } ,0} < 2 \fl{\epsilon \sqrt{n}} }
$.  % In order to remove the annoying and insignificant term $ \fl{\epsilon \sqrt{n}} $, we introduce $\Sigma_i = 2 \sum_{k=0} ^ {2 \fl{ \sqrt{n} \epsilon} -1} Z^i_k \ind
% {\sigma^{Z^i} _{\fl{ \sqrt{n} \epsilon} ,0} < 2 \fl{ \sqrt{n} \epsilon} } $.
% As before
Therefore, it suffices to show that 
  \begin{equation}
    \label{Wwlln}
    P \bigg( \sup _{k< K/ \epsilon ^ 2}
    \Big\vert \sum_{i \in \bint{1,k} \cap \cal{B}} \left(\Sigma^n_i-
      \nu/8 \right) \Big\vert >  \epsilon^{-2}h \bigg) <  h\quad\text{for all $n \geq n_0$},
  \end{equation}
where \[\Sigma^n_i:=  \frac{1}{\fl{\epsilon \sqrt{n}}^2}\sum_{k=0} ^ {2 \fl{ \epsilon \sqrt{n}} -1} Z^i_k   \ind{\sigma^{Z^i} _{\fl{ \epsilon} \sqrt{n} ,0} < 2 \fl{ \epsilon \sqrt{n}} }, \]
  and this we will do. We will exploit
  Proposition \ref{BESQ2int}. 
\begin{defn}\label{Hde}
  Given a law $\lambda_0$ on $\R $ with
  $\ell = \int x \lambda_0 (dx ) $ well defined and finite, we define
  $\cal{H}_{\delta, \epsilon} $ to be the collection of laws on
  $\R, \lambda$, such that $\lambda $ can be written as
  $\lambda = \int K(z,\cdot) \lambda_1 (dz) $ where
 \begin{enumerate}[(i)]
 \item $K(\cdot,\cdot) $ is a probability kernel satisfying $K(z, [z- \delta, z +\delta ]^c ) = 0 $ for all $z$ and 
 \item  $\| \lambda_0 - \lambda_1 \|_{TV} < 8 \epsilon^3$. 
 \end{enumerate}
 We say that a sequence of random variables
 $\Xi_1, \Xi_2, \dots, \Xi_N $ is an
 $\cal{H}_{\delta, \epsilon} $-chain if the law of $\Xi_1 $ is in
 $\cal{H}_{\delta, \epsilon} $ and for $1< j \le N$, the conditional
 law of $\Xi_j$ given $\Xi_1, \Xi_2, \dots, \Xi_{j-1} $ is in
 $\cal{H}_{\delta, \epsilon}
 $.
\end{defn}
We will also need following lemma and corollary.
\begin{lemma} \label{lemH} For an $\cal{H}_{\delta, \epsilon} $-chain
  $\Xi_i,\ 1\le i\le N $, taking only finitely many values and
  all $c > 0 $
 \[
 P \left( \sup_{j \le N}  \Big\vert \sum_{i=1} ^j \Xi_i  - j \ell     \Big\vert \ge c \right) \le P \left( \sup_{j \le N}  \Big\vert \sum_{i=1} ^j \zeta_i  - j \ell    \Big\vert \ge c -N \delta \right) + 8N \epsilon ^ 3.
 \]
 where $\zeta_i,\ 1\le i \le N$, are i.i.d.\ random variables with law $\lambda_0$.
 \end{lemma}

 \begin{rem}
   The assumption that the random variables take only finitely many
   values is simply an artificial condition that suits our purposes
   and avoids measurability issues.
\end{rem}
  \begin{proof}
    We claim that an $\cal{H}_{\delta, \epsilon} $-chain
    $\Xi_i, \ 1\le i\le N $, can be coupled with i.i.d.\
    $\zeta_i, \, 1 \le i \le N $, with law $\lambda_0$ so that
    $ P(\vert \zeta_i - \Xi_i \vert \ge \delta) \le 8 \epsilon ^3$ for all
    $1\le i \le N$.  We then note that
\[
\bigg\{ \sup_{j \le N}  \Big\vert \sum_{i=1} ^j \Xi_i  - j \ell     \Big\vert \ge c \bigg\} \subset \bigg\{  \sup_{j \le N}  \Big\vert \sum_{i=1} ^j \zeta_i  - j \ell    \Big\vert \ge c -N \delta \bigg\} \bigcup  \bigg\{\bigcup\limits_{j=1}^N \{    \vert  \zeta_j  -\Xi_j   \vert \ge \delta \}\bigg\} 
\]
The conclusion is now simply an application of the union bound.  So it
remains to establish the claim.  The coupling is based on finding,
given some $\lambda \in \cal{H}_{\delta, \epsilon}$, a measure $\nu$
on $\mathbb{R}^2$ having respective marginals $\lambda $ and
$\lambda _0$ and such that
$\nu(\{(x,y): \vert x-y \vert > \delta \})< 8 \epsilon ^3$.  The
existence of such a measure is shown by Lemma \ref{triple}. Given
$\lambda_j $, the conditional law of $\Xi_j$ given
$\Xi_1,\Xi_2, \cdots \Xi_{j-1}$, we take $\nu_j$ to be the
corresponding coupled law on $\mathbb{R}^2$.  Then for $L_j$, the
regular conditional kernel for $y$ (the second coordinate) given $x$
under law $\nu_j$, we choose $\zeta_j$ according to probability
$L_j (\Xi_j,\cdot)$ using an auxiliary uniform random variable in the
usual manner.
\end{proof}

The following corollary is a direct consequence of Lemma \ref{lemH}
 and convergence in distribution (see \cite{ekMP}, Section 3, Theorem
 1.2).
 \begin{cor} \label{corH} For a fixed positive integer $N$ let
   $(\Xi^n_1, \Xi^n_2, \dots, \Xi^n_N)_{n \ge 1}$ be a sequence of
   finite valued random vectors in $\R^N$ such that every
     distributional limit point of $\Xi^n_1$ and of the conditional
     probability of $\Xi^n_j$ given
     $\Xi^n_1, \Xi^n_2, \cdots \Xi^n_{j-1}$, $1<j\le N$, as
     $n\to \infty$ (considered as a probability on ${ \bar \R} $) is
     in $\cal{H}_{\delta, \epsilon} $.  Then for all $c>0$
 \[
   \limsup_{n \rightarrow \infty} P \left( \sup_{j \le N} \Big\vert
     \sum_{i=1} ^j \Xi^n_i - j \ell \Big\vert \ge c \right) \le P
   \left( \sup_{j \le N} \Big\vert \sum_{i=1} ^j \zeta_i - j \ell
     \Big\vert \ge c -N \delta \right) + 8N \epsilon ^ 3
 \]
 where $\zeta_i,\ i \ge 1 $, are i.i.d.\ random variables with law $\lambda_0$.
 \end{cor}
 To apply this corollary we restrict our attention to the event
   that the coupling does not break down before $\epsilon^{-2}K$. This
   event has probability at least $1-h/4$ for all sufficiently small
   $\epsilon$. We fix such an $\epsilon$ and enumerate the points in
   $\bint{1,\epsilon^{-2}K}\cap{\cal B}$ by $j\in\bint{1,N}$ so that
   $N=|{\cal B}|\le \epsilon^{-2}K$. Next we let $\Xi^n_j$,
   $j\in\bint{1,N}$, be equal to the corresponding $\Sigma^n_i$,
   $i\in\bint{1,\epsilon^{-2}K}\cap{\cal B}$. Then by
   Proposition~\ref{BESQ2int} the sequence
   $(\Xi^n_1, \Xi^n_2, \dots, \Xi^n_N)_{n \ge 1}$ satisfies the
   conditions of Corollary~\ref{corH} with $\delta=\epsilon^8$ and
   $\lambda_0$ equal to one half $\delta_0$ plus one half
   the law of $\nu/4$ times the time for the standard Brownian motion
   to exit $(-1,1)$ (note that this gives $\ell=\nu/8$). Choosing
   $c=\epsilon^{-2}h$ we arrive at \eqref{Wwlln} provided that
   $\epsilon=\epsilon(K,h)$ was chosen sufficiently small.
 \end{proof}

\appendix

\section{}

\subsection{Proofs of facts regarding BMPE}

\begin{proof}[Proof of Lemma~\ref{BMPEcoup}] {\em Step 1.} We shall
  restate the question in terms of Brownian motion and its running
  maximum $B^*(t)=\max_{0\le s\le t}B(s)$. Note that by
  \cite[p.\,242]{cpyBetaPBM}
  \begin{equation}
    \label{rep}
    W(t)=B(t)+\frac{\theta^+}{1-\theta^+ }B^*(t)
  \end{equation}
  is a pathwise unique solution of the equation
  $W(t)=B(t)+\theta^+ S(t)$ with $ S(0)=W(0)=0$. To allow for non-zero
  initial data we may assume that $B_i(t)$, $i=1,2$, are defined for
  $t\in[-1,\infty)$ and that
  \[B_i(0)=W_i(0)-\theta^+ S_i(0),\quad
    B^*_i(0)=\max_{t\in[-1,0]}B_i(t)=(1-\theta^+ )S_i(0), \quad
    i=1,2.\] Then
\[W_i(t)=B_i(t)+\frac{\theta^+ }{1-\theta^+ }\,B_i^*(t)\] is a solution
of $W_i(t)=B_i(t)+\theta^+ S_i(t)$ for $t\ge 0$ with the given initial
pair $(W_i(0),S_i(0))$,
$i=1,2$. We conclude that \[
    \begin{bmatrix}
      W_i(t)\\S_i(t)
    \end{bmatrix}=\frac{1}{1-\theta^+ }
    \begin{bmatrix}
      1-\theta^+ &\theta^+ \\0&1
    \end{bmatrix}
    \begin{bmatrix}
      B_i(t)\\B^*_i(t)
    \end{bmatrix},\quad i=1,2,
\] and, thus,
  \[(B_1(1),B^*_1(1))=(B_2(1),B^*_2(1))\quad\Rightarrow\quad
  (W_1(1),S_1(1))=(W_2(1),S_2(1)).\] Moreover,
  if $(B_i(t),B^*_i(t))\in[-K,K]^2$ for all $t\in[-1,1]$ then
  \[S_i(0)\vee\max_{t\in[0,1]}|W_i(t)|\le K\left(1+\frac{|\theta^+ |}{1-\theta^+ }\right),\quad i=1,2.
  \] 

  {\em Step 2.} We shall now couple two pairs of Brownian motions and their
  running maxima. Without loss of generality we can shift one starting
  point to the origin and assume that $(B_1(0),B_1^*(0))=(0,b_1)$,
  $(B_2(0),B_2^*(0))=(a_2,b_2)$, where
  $(0,b_1),(a_2,b_2)\in\{(x,y): x\le y\}$. Note that for $b_1>0$ the
  distribution of $(B_1(1),B^*_1(1))$ is not absolutely continuous as
  the line $y=b_1$ carries a positive measure. But $(B_1(1),B^*_1(1))$
  has a density on $\{y>b_1\}$. A similar remark applies to the other
  pair. Denote by $\mu_{0,b_1}$ and $\mu_{a_2,b_2}$ the absolutely
  continuous parts of distributions of $(B_1(1),B^*_1(1))$ and
  $(B_2(1),B^*_2(1))$ respectively . Then there is a $c_0>0$ such that
  for $b_1,|a_2|, |b_2|\le c_0$
  \[\|\mu_{0,b_1}\|_{TV}\ge
    \frac9{10},\quad\|\mu_{a_2,b_2}\|_{TV}\ge\frac9{10},\quad\|\mu_{0,b_1}-\mu_{a_2,b_2}\|_{TV}\le\frac15.\]
  Next we choose $r_0>c_0$ such that for a standard Brownian motion
  $B(\cdot)$
  \[P\left(\max_{0\le t\le 1}B(t)\ge r_0-c_0\right)\le \frac1{10}\]
  and let $\mu^{r_0}_{a,b}$ denote the distribution of $(B(1),B^*(1))$
  with $B(0)=a$, $B^*(0)=b$, $-c_0\le a\le b\le c_0$, restricted to
  $\R\times(b,\infty)$, and killed upon leaving $[-r_0,r_0]^2$. Then
  for $b_1,|a_2|,|b_2|\le c_0$
  \[\|\mu^{r_0}_{0,b_1}\|_{TV},\|\mu^{r_0}_{a_2,b_2}\|_{TV}\ge \frac9{10}-\frac1{10}=\frac45,\quad
    \|\mu^{r_0}_{0,b_1}-\mu^{r_0}_{a_2,b_2}\|_{TV}\le
    \frac15+\frac1{10}=\frac3{10}.\] Thus, for
  $b_1,|a_2|,|b_2|\le c_0$ we can couple $(B_1(t),B_1^*(t))$ and
  $(B_2(t),B_2^*(t))$ so that $(B_1(1),B_1^*(1))=(B_2(1),B_2^*(1))$
  and $\max_{t\in[0,1]}|B_i(t)|\le r_0$, $i=1,2$, with probability $4/5-3/10=1/2$.

  From Steps 1 and 2 we conclude that there are constants
  $c_1>0,\ r_1\in (c_1,\infty)$ such that if $(W_1(0),S_1(0))=(0,\wo_1)$ and $(W_2(0),S_2(0))=(w_2,\wo_2)$ with $\wo_1,|w_2|,|\wo_2|\le c_1$ then there is a coupling such that with probability $1/2$ \[(W_1(1),S_1(1))=(W_2(1),S_2(1))\quad\text{and}\quad \max_{t\in[0,1]}|W_i(t)|\le r_1,\ \ i=1,2.\] 

{\em Step 3.} Let $A_1$ be the event that
\begin{enumerate}[(i)]
\item $W_i(1/2),S_i(1/2)\in [3/2,2], \ i=1,2$, and are within $c_1/r_1$ of each other;
\item $S_i(0)\vee\max_{t\in[0,1/2]}|W_i(t)|\le 2, \ i=1,2$.
\end{enumerate}
Note that without loss of generality we can assume that
$r_1\ge \sqrt{2}$ so that $1/2+1/r_1^2\le 1$.
It is easy to see that $P(A_1)\ge p_1$ for some $p_1=p_1(c_1,r_1)>0$
uniformly over $W_i(0),S_i(0)\in[-1,1]$, $i=1,2$.

Let $A_2$ be the event that $(W_1(t),S_1(t))$ and $(W_2(t),S_2(t))$,
$t\in[1/2,1/2+1/r_1^2]$, are coupled as above and scaled accordingly so that
\[(W_1(1/2+1/r_1^2),S_1(1/2+1/r_1^2))=(W_2(1/2+1/r_1^2),S_2(1/2+1/r_1^2))\
  \ \text{and}\ \ \max_{1/2\le t\le 1/2+1/r_1^2}|W_i(t)|\le 3.\] By Steps 1, 2, and scaling, the conditional probability of $A_2$ given $A_1$ is $1/2$ uniformly over $W_i(1/2),S_i(1/2)$, $i=1,2$. Once we have the coupling, we note that the probability that over the leftover time period $[1/2+1/r_1^2,1]$ the coupled processes do not exit $[-4,4]$ is strictly positive. This finishes the proof.
\end{proof}

\begin{proof}[Proof of Proposition~\ref{DiscBMPE}]
Let ${\tilde W} (t)=\epsilon W(\epsilon^{-2}t)$, $\tau^\epsilon_0=0$, and
\[\tau^\epsilon _k  = \inf\{ s > \tau^\epsilon _{k-1} : \ \vert {\tilde W}(s)\ - \ {\tilde W} (\tau^\epsilon _{k-1}) \vert  = \epsilon \},\quad k\in\N.\]
With this notation, establishing \eqref{labe1} is equivalent to showing
\begin{equation}
\label{labe2}
 \sup_{0 \leq s \leq T} \big\vert {\tilde W}(s) - \ {\tilde W}(\tau^\epsilon _{\fl{\epsilon^{-2}s}}) \big\vert \overset{\text{P}}{\longrightarrow} 0\quad\text{as }\ \epsilon\to 0.
\end{equation}
As ${ \tilde W}(s), s\ge 0,$ is pathwise continuous (and its law does not depend on $\epsilon$)  \eqref{labe2} is implied by %showing
\[
 \sup_{0 \leq s \leq T} \big\vert s - \ \tau^\epsilon _{\fl{\epsilon^{-2}s}} \big\vert \overset{\text{P}}{\longrightarrow} 0\quad\text{as }\ \epsilon\to 0.
\]
In turn this is equivalent to showing that for each $0 < T < \infty $,
\begin{equation}
\label{labe3}
 \sup_{1 \leq K \leq \epsilon ^{-2}T} \Big\vert \sum_{k=1}^K ( \tau^\epsilon _{k}-\tau^\epsilon _{k-1}- \epsilon ^2) \Big\vert \overset{\text{P}}{\longrightarrow} 0\quad\text{as }\ \epsilon\to 0.
\end{equation}
Again by scaling, we see that \eqref{labe3} is equivalent to ($\tau_k$ and $(I_k,W_k,S_k)$ were defined in Section~\ref{walk})
\[
 \sup_{1 \leq K \leq N} \Big\vert \frac{1}{N}\sum_{k=1}^K ( \tau _{k}-\tau _{k-1}- 1 ) \Big\vert
\overset{\text{P}}{\longrightarrow} 0\quad \text{as }\ N\to\infty.
\]
To this end, first note that when $W_k$ is in the bulk (that is when
$I_k+1\leq W_k \leq S_k-1$) then $\tau_{k+1}-\tau_k$ has the same
distribution as the exit time of a standard Brownian motion from
$(-1,1)$.  On the other hand, if $W_k$ is at the extreme (either
$S_k < W_k+1$ or $I_k > W_k-1$) then the distribution of
$\tau_{k+1}-\tau_k$ depends on the specific values of $S_k-W_k$ or
$S_k-I_k$.  However, using the representation in \eqref{rep} we
  infer that for all $k\ge 1$ the distribution of $\tau_{k+1}-\tau_k$
given $\mathcal{F}_k = \sigma(W(t): \, t\leq \tau_k )$ is
stochastically dominated by
\[
  \inf\{t> 0: W(t)-S(t)=-2\}\overset{\eqref{rep}}{=}\inf \{ t> 0: B(t) - \max_{s \leq t}
    B(s)=-2\}.
\]
In particular, this implies that the conditional mean and variance of
$\tau_k-\tau_{k-1}$ given $\mathcal{F}_{k-1}$ are uniformly
bounded. That is, there exist constants $A,B<\infty$ such that
\[
 t_k := E[\tau_k-\tau_{k-1} \mid \mathcal{F}_{k-1} ] \leq A
\quad \text{and}\quad 
 E[(\tau_k-\tau_{k-1}-t_k)^2 \mid \mathcal{F}_{k-1} ] \leq B < \infty. 
\]
%Moreover, $t_k = 1$ if $W_{k-1}$ is in the bulk. 
Then, it follows from Doob's martingale inequality that for any $\delta > 0$
\[
 P\left( \sup_{1 \leq K\leq N} \Big\vert \frac{1}{N}\sum_{k=1}^K ( \tau _{k}-\tau _{k-1}- t_k ) \Big\vert \geq \delta \right) \leq \frac{1}{\delta^2} E\left[ \left( \frac{1}{N}\sum_{k=1}^N ( \tau _{k}-\tau _{k-1}- t_k )   \right)^2 \right] \leq \frac{B}{\delta^2 N}.  
\]
Thus, it remains only to show that 
\[
 \sup_{1 \leq K \leq N} \Big\vert \frac{1}{N}\sum_{k=1}^K ( t_k - 1 ) \Big\vert
\overset{\text{P}}{\longrightarrow} 0.
\]
However, since $t_k \equiv 1$ when $W_{k-1}$ is in the bulk and is uniformly bounded otherwise, it is enough to show that 
\[
  \lim_{N\to\infty} \frac{1}{N} \sum_{k=1}^N   \ind{ (I_k,W_k,S_k) \mbox{ is at the extreme} }  = 0, \quad P\text{-a.s..}
\]
It's enough only to consider the right extremes (that is, when $S_k < W_k+1$) since the left extremes can be handled similarly. We'll show that 
\begin{equation}\label{rext}
 \lim_{N\to\infty} \frac{1}{N} \sum_{k=1}^N  \ind{ W_k \geq 1, S_k < W_k+1}  = 0, \quad P\text{-a.s..}
\end{equation}

The proof of this will rely on the following facts. 
\begin{itemize}
 \item If $W_k = m\geq 1$ and $S_k < m+1$, the probability (conditioned on $W(t)$ for $t\leq \tau_k$) that $W_{k+1} = m+1$ is at least $p_- = (1/2)\wedge(1/2)^{1-\theta^+}  > 0$ and at most $p_+ = (1/2)\vee (1/2)^{1-\theta^+}  < 1$. This follows from Corollary \ref{roc1}. 
 \item If $W_k = m\geq 1$ and $S_k \geq m+1$, the probability that $W_{k+1} = m+1$ is exactly $1/2$. 
\end{itemize}

First of all, for any $m\geq 1$ let
$\chi_m= \sum_{k=1}^\infty \ind{W_k = m, \, S_k < m+1}$ be the total
number of times a right extreme occurs and the BMPE-walk is at
location $m$. It is easy to see that the sequence $\{\chi_m \}_{m\geq 1}$
is i.i.d.  Moreover, since whenever $W_k = m$ is at the extreme, the
probability that the next step is to the right is at least
$p_-$ and so $\chi_m$ is stochastically dominated by a Geom($p_-$)
random variable. In particular, $E[\chi_1] < \infty$.
Thus,
\begin{equation}\label{Hlim}
 \lim_{n\to\infty} \frac{1}{n} \sum_{m=1}^n \chi_m = E[\chi_1] < \infty, \quad P\text{-a.s.}
\end{equation}

Next, for $n\geq 0$ let $\rho_n = \inf\{ k\geq 0: W_k = n \}$ be the
time it takes for the walk $W_k$ to reach $n$ for the first time. It
is easy to see that $\rho_{n+1}-\rho_n$ stochastically dominates the
time it takes the Markov chain on $\{0,1,\ldots,n,n+1\}$ shown in
Figure \ref{fig:RWPE} to step from $n$ to $n+1$.
\begin{figure}[ht]
 \includegraphics[width=0.8\textwidth]{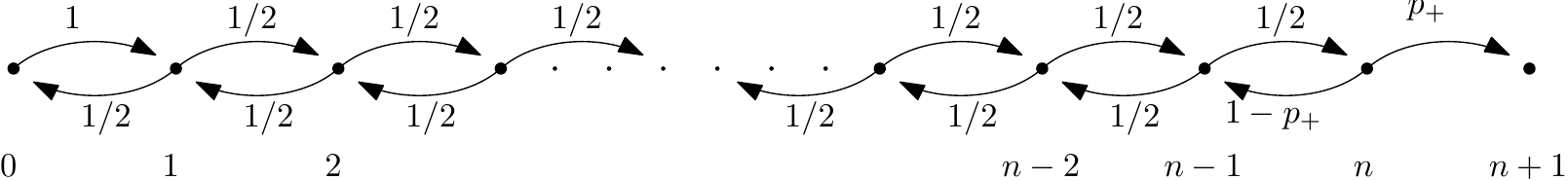}
\caption{The above Markov chain behaves like a simple symmetric random walk at $x=1,2,\ldots,n-1$, an asymmetric simple random walk at $x=n$, and reflects to the right at $x=0$.}\label{fig:RWPE}.
\end{figure}

\noindent 
Let $\{\gamma_n \}_{n\geq 0}$ be a sequence of independent random
variables where for each $n$ the random variable $\gamma_n$ has the
distribution of the time for the Markov chain in Figure \ref{fig:RWPE}
to cross from $n$ to $n+1$. Then  $\rho_n$
stochastically dominates $\sum_{k=0}^{n-1} \gamma_k$ and
thus\footnote{Note that the random variables $\{\gamma_k\}_{k\geq 0}$ are independent and $\gamma_{n+1}$ stochastically
  dominates $\gamma_n$. Moreover, for $n\in\N$ by an easy recursion
  computation,
  $E[\gamma_n]=\frac{1}{p_+}+\frac{1-p_+}{p_+}(2n-1)\to\infty$ as
  $n\to\infty$.}
\begin{equation}\label{rholim}
 \lim_{n\to\infty} \frac{\rho_n}{n} = \infty, \quad P\text{-a.s.}
\end{equation}

Finally, we are ready to prove \eqref{rext}. For each $N \geq 1$ there is a unique $n\geq 0$ such that $S_N \in [n, n+1)$ and note that $S_N \in [n,n+1)$ is equivalent to $\rho_n \leq N < \rho_{n+1}$. Therefore, on the event $\{\rho_n \leq N < \rho_{n+1} \}$ we have 
\[
 \frac{1}{N} \sum_{k=1}^N  \ind{ W_k \geq 1,  S_k < W_k+1}
%= \frac{1}{N} \sum_{k=1}^N \sum_{m=1}^n \ind{ W_k = m , \text{ and } S_k < m+1}
\leq \frac{1}{N} \sum_{m=1}^n \chi_m
\leq \left(\frac{n}{\rho_n} \right) \left( \frac{1}{n} \sum_{m=1}^n \chi_m \right). 
\]
Since $n\to \infty$ as $N\to\infty$, we have that \eqref{rext} follows from \eqref{Hlim} and \eqref{rholim}. 
\end{proof}

\subsection{Proofs of diffusion approximation results for BLPs}

\begin{proof}[Proof of Theorem~\ref{DA0}]
  (1) The proof of this part is very similar to the one of \cite[Lemma
  7.1]{kpERWMCS} and is based on \cite[Theorem 4.1,\
  p.\,354]{ekMP}. First of all, the martingale problem for
  \[A=\left\{\left(f,Gf=\frac{\nu}{2}\,x_+\,\frac{\partial^2}{\partial
          x^2}+{D}\,\frac{\partial}{\partial x}\right):\,f\in
      C_c^\infty(\R)\right\}\] on $C_\R[0,\infty)$ is well-posed by
  \cite[Corollary 3.4, p.\,295]{ekMP} and the fact that the existence
  and distributional uniqueness hold for solutions of (\ref{daa}) with
  arbitrary initial distributions.\footnote{A more detailed discussion
    of (\ref{daa}) can be found immediately following (3.1) in
    \cite{kzEERW}.} 

Define $A_m(t)$ and $B_m(t)$ for all $t\ge 0$ by 
\[ A_m(t):=\frac1{m^2}\sum_{k=1}^{\fl{mt}}\mathrm{Var}(V^+_{m,k}\,|\,V^+_{m,k-1});\quad 
  B_m(t):= \frac1m\sum_{k=1}^{\fl{mt}}E[V^+_{m,k}-V^+_{m,k-1}\,|\,V^+_{m,k-1}].
\]
Then for each $m\in\N$ the processes $M_m(t):=Y_m(t)-B_m(t)$ and
$M_m^2(t)-A_m(t)$, $t\ge 0$, are martingales with respect to the
natural filtration of $V^+_m$.

Recall that $\tau_r^{Y_m} = m^{-1}\tau_{rm}^{Z_m}$. 
%Let $\tau_r^{Y_m} = \inf\{t \geq 0: Y_m(t) \geq r \} = m^{-1}\tau_{rm}^{Z_m}$. 
To apply the cited theorem we only need to
check that for all $T,r > 0$ the following five conditions hold.
\begin{align}\label{EKc3}
 &\lim_{m\to\infty} E\left[ \sup_{t\leq T \wedge \tau_r^{Y_m}} \left| Y_m(t) - Y_m(t-) \right|^2 \right]  =0. 
\\
\label{EKc4}
  &\lim_{m\ra\infty} E\left[ \sup_{t\leq T \wedge \tau_r^{Y_m}} \left| B_m(t) - B_m(t-) \right|^2 \right]  =0. 
\\
\label{EKc5}
  &\lim_{m\ra\infty} E\left[ \sup_{t\leq T \wedge \tau_r^{Y_m}} \left| A_m(t) - A_m(t-) \right| \right]  =0. 
\\
\label{EKc6}
 &\sup_{t \leq T \wedge \tau_r^{Y_m}} \left| B_m(t) - (1+\eta\cdot\br^+) t \right| \overset{\text{P}}{\underset{m\ra\infty}{\longrightarrow}}0.
\\
\label{EKc7}
  &\sup_{t \leq T \wedge \tau_r^{Y_m}} \left| A_m(t) - \nu \int_0^t (Y_m(s))_+ \, ds  \right|\overset{\text{P}}{\underset{m\ra\infty}{\longrightarrow}}0.
\end{align}

Recalling the construction of the BLP $V^+$ in terms of the Bernoulli trials $\{\xi_j^x\}_{x\geq 0, \, j\geq 1}$ as in Section \ref{sec:BLP}, let
$G^k_i$ be the number of ``successes'' between the $(i-1)$-th and $i$-th ``failure'' in the sequence of Bernoulli trials $\{\xi_j^k\}_{j\geq 1}$ so that
\begin{equation}\label{VkSumG}
 V^+_{m,k}
 = \sum_{j=1}^{V^+_{m,k-1}+1}G^k_j 
 =V^+_{m,k-1}+1+\sum_{j=1}^{V^+_{m,k-1}+1}(G^k_j-1).
\end{equation}
Using this representation for the $V^+$ processes,
condition \eqref{EKc3} states that for
every $T,r>0$
\[
\lim_{m\to\infty}\frac{1}{m^2}\,E\left[\max_{1\le k\le (Tm)\wedge
    \tau_{rm}^{V^+_m}}\bigg|1+\sum_{j=1}^{V^+_{m,k-1}+1}(G^k_j-1)\bigg|^2\right]=0,
\] 
where $\tau_{rm}^{V^+_m} = \inf\{k\ge 0: V^+_{m,k} \geq rm \}$. 
To see that it holds we write
\begin{align*}
  \frac{1}{m^2}E \left[\max_{1\le k\le (Tm)\wedge
      \tau_{rm}^{V^+_m}}\right. &\left. \Big|\sum_{j=1}^{V^+_{m,k-1}+1}(G^k_j-1)\Big|^2\right]\le
  \frac{1}{m^2}E\left[\max_{1\le k\le Tm} \max_{1\le \ell\le
      rm+1}\Big|\sum_{j=1}^\ell(G^k_j-1)\Big|^2\right]\\ 
&   =
  \frac{1}{m^2} \sum_{y=0}^\infty P\left(\max_{1\le k\le Tm}
    \max_{1\le \ell\le
      rm+1}\Big|\sum_{j=1}^\ell(G^k_j-1)\Big|^2>y\right) 
\\&\le \frac{r^{3/2}}{\sqrt{m}}
+ (r T+1) \sum_{y\ge (rm)^{3/2}} \max_{1\le \ell\le rm+1}P\left( \Big|\sum_{j=1}^\ell(G^k_j-1)\Big|>\sqrt{y}\right).     
  \end{align*}
  Finally we apply Lemma~A.1 from \cite{kpERWMCS} to get that the
  expression in the last line does not exceed
\begin{align*}
  &\frac{r^{3/2}}{\sqrt{m}}+r T\sum_{y\ge (rm)^{3/2}} C
  \left(\exp\left\{-c \left(\frac{y}{\sqrt{y}\vee
          (8rm)}\right)\right\} + \exp\left\{-c\sqrt{y}\right\}\right)\\ \le
  &\frac{r^{3/2}}{\sqrt{m}}+r T\sum_{y\ge (rm)^{3/2}} C
  \left(\exp\left\{-c \left(\frac{y}{\sqrt{y}\vee
          (8y^{2/3})}\right)\right\} + \exp\left\{-c\sqrt{y}
    \right\}\right)\to 0 \ \text{ as } m\to\infty.
\end{align*}

Conditions \eqref{EKc4} and \eqref{EKc5} follow from Propositions 4.1
and 4.2 of \cite{kpERWMCS} respectively. Indeed, by
\cite[Proposition~4.1]{kpERWMCS} for some $c_1,c_2>0$, all 
and $n\ge 0$
\[\left|E[ V^+_1 \mid V^+_0=n]- n
    -(1+\eta\cdot\br^+)\right|\le c_1e^{-c_2n}.\] Using the Markov
property and the fact that $V^+_{m,k-1}\le rm$ for
$k\le \tau_{rm}^{V^+_m}$ we get
\begin{align*}
   \lim_{m\to\infty} &E\left[ \sup_{t\leq T \wedge \tau_r^{Y_m}} \left| B_m(t) - B_m(t-) \right|^2 \right]
= \lim_{m\to\infty} \frac{1}{m^2} E\left[ \max_{1\leq k \leq (Tn) \wedge \tau_{rm}^{V^+_m}} \left( E[ V^+_{m,k} - V^+_{m,k-1} \, | \, V^+_{m,k-1}] \right)^2 \right] \\
&\le \lim_{m\to\infty} \frac{1}{m^2}\, E\left[ \max_{1\leq k \leq (Tm) \wedge \tau_{rm}^{V^+_m}} \left( E[ V^+_{m,k} | \, V^+_{m,k-1}]- V^+_{m,k-1} -(1+\eta\cdot\br^+) \right)^2 \right]\\
&\leq \lim_{m\to\infty} \frac{c_1^2}{m^2}\,E\left[ \max_{1\leq k \leq (Tm) \wedge \tau_{rm}^{V^+_m}}e^{-2c_2V^+_{m,k-1}}\right] = 0, 
\end{align*}
Similarly, by \cite[Proposition~4.2]{kpERWMCS} there is a $c_3>0$ such that $\left|\Var(V^+_1\mid V^+_0=n)-\nu n\right|\le c_3$ for all $n\ge 0$. Therefore,
\begin{multline*}
 \lim_{m\to\infty} E\left[ \sup_{t\leq T \wedge \tau_r^{Y_m}} \left| A_m(t) - A_m(t-) \right| \right]  
= \lim_{m\to\infty} \frac{1}{m^2} E\left[ \max_{1\leq k \leq (Tm)\wedge \tau_{rm}^{V^+_m}} \Var(V^+_{m,k} \, | \, V^+_{m,k-1}) \right] \\
\le \lim_{m\to\infty} \frac{ (\nu rm+c_3)}{m^2} = 0. 
\end{multline*}

To check condition \eqref{EKc6}, note that 
\begin{align*}
 \sup_{t \leq T \wedge \tau_r^{Y_m}} &\left| B_m(t) - (1+\eta\cdot\br^+)t \right| 
\\ &\leq \frac{1+\eta\cdot\br^+}{m} + \sup_{1\leq k \leq (Tm) \wedge \tau_{rm}^{V^+_m}} \frac{1}{m} \sum_{j=1}^k \left| E\left[ V^+_{m,j} - V^+_{m,j-1} \, | \, V^+_{m,j-1} \right] - (1+\eta\cdot\br^+) \right| \\
&\leq \frac{1+\eta\cdot\br^+}{m} + \frac{c_1}{m} \sum_{j=1}^{(Tm) \wedge \tau_{rm}^{V^+_m}} e^{-c_2V^+_{m,j-1}}\le \frac{c_4}{m} + \frac{c_1}{m} \sum_{j=1}^{Tm}\ind{V^+_{m,j-1}\le m^{\alpha}}. 
\end{align*}
By Lemma~\ref{small2}, for any $\alpha \in (0,1-\theta^-)$ the last expression goes to $0$ in probability as $m\to\infty$, and we have shown that condition \eqref{EKc6} holds.

Finally, to check condition \eqref{EKc7} note that 
\begin{align*}
 \sup_{t \leq T \wedge \tau_r^{Y_m}} &\left| A_m(t)-\nu \int_0^t (Y_m(s))_+ \, ds  \right|
\\ &\le \max_{1\leq k \leq (Tm) \wedge \tau_{rm}^{V^+_m}}\left| \frac{1}{m^2} \sum_{j=1}^k \Var(V^+_{m,j} \, | \, V^+_{m,j-1}) - \frac{\nu}{m^2} \sum_{j=1}^{k} V^+_{m,j-1} \right|+\frac{\nu}{m^2}\, V^+_{m,k-1}   \\
&\leq \max_{1\leq k \leq (Tm) \wedge \tau_{rm}^{V^+_m}} \left(\frac{1}{m^2} \sum_{j=1}^k \left|\Var(V^+_{m,j} \, | \, V^+_{m,j-1}) -\nu V^+_{m,j-1}\right|+\frac{\nu}{m^2}\,V^+_{m,k-1}\right)\le \frac{c_3 T+\nu r}{m}\to 0
\end{align*}
as $m\to\infty$.  This completes the proof of condition \eqref{EKc7}
and thus also the proof of part (1).

(2) The process convergence part of the argument is based on
\cite[Theorem 3.2]{bCOPM} which we state below for the reader's
convenience.
\begin{thm}{\bf (\cite[Theorem 3.2]{bCOPM})}\label{Bil}
  Let $(S,d)$ be a metric space. Suppose that
  $Y_{m,\ell},\, Y_m,\, Y^{(\ell)}$ $(m,\ell\in\N)$ and $Y^{(\infty)}$ are
  $S$-valued random variables such that $Y_{m,\ell}$ and $Y_m$ are
  defined on the same probability space with probability measure $P^m$
  for all $m,\ell\in\N$. If
  $Y_{m,\ell}\underset{m\to\infty}{\Longrightarrow}
  Y^{(\ell)}\underset{\ell\to\infty}{\Longrightarrow}Y^{(\infty)}$
  and
  \[\lim_{\ell\to\infty}\limsup_{m\to\infty}P^m(d(Y_{m,\ell},Y_m)>\epsilon)=0\]
  for each $\epsilon>0$, then
  $Y_m\underset{m\to\infty}{\Longrightarrow}Y^{(\infty)}$.
\end{thm}
\begin{rem}\label{corproof}
    The proof of Corollary~\ref{da0.25dead} repeats the argument below
    word for word on the space $D([0,T])$ with the metric $d^\circ_T$
    (see \cite[p.\,166 and (12.16)]{bCOPM}) and use Lemma~\ref{dead}
    instead of Lemma~\ref{sm0}.
  \end{rem}

In addition to processes $Y_m$ and $Y$ defined in the statement, for
$\delta:=1/\ell>0$ we let
$Y_{m,\ell}(t)=m^{-1}U^+_{m,\fl{tm}\wedge \sigma_{m\delta}}$,
$Y^{(\ell)}(t)=Y(t\wedge\sigma_\delta)$,
$Y^{(\infty)}(t)=Y(t\wedge \sigma_0)$, $t\ge 0$, and work in the space
$D[0,\infty)$ with the $J_1$ metric $d^\circ_\infty$ (see
\cite[(16.4)]{bCOPM}). From \cite[Lemma 6.1]{kpERWMCS}\footnote{Lemma
  6.1 is stated and proved in \cite{kpERWMCS} for the processes $V^-$
  with deterministic initial conditions but it holds with the same proof for the other 3 processes and random initial distributions.} or,
alternatively, by repeating essentially word for word the proof of
part (1), we know that $\forall \ell\in\N$,
$Y_{m,\ell}\underset{m\to\infty}{\Longrightarrow} Y^{(\ell)}$.
Moreover,
$Y^{(\ell)}\underset{\ell\to\infty}{\Longrightarrow} Y^{(\infty)}$ as
$\theta^+<1$. Indeed, using the properties of BESQ$^\dim$ with $\dim<2$ we have $\forall \epsilon>0$
\begin{align*}
P\left(\sup_{t\ge 0}|Y(t\wedge \sigma_\delta)-Y(t\wedge \sigma_0)|>\epsilon\right)
&\le P\left(\sup_{t\ge \sigma_\delta}Y(t\wedge \sigma_0)>\frac{\epsilon}{2}\right) 
%\le P_\delta(\tau^Y_{\epsilon/2}<\sigma^Y_0)\to 0\ \text{as $\delta\to 0$.}
\\
 & \leq P(\tau^Y_{\epsilon/2}<\sigma^Y_0 \mid Y(0) = \delta ) \to 0\ \text{as $\delta\to 0$.}
\end{align*}
We are left to check the last condition of
Theorem~\ref{Bil}. For all $\delta\in(0,\epsilon/2)$ and $r>0$ we have
that
\begin{multline*}
  P^m\left(d^\circ_\infty(Y_{m,\ell},Y_m)>\epsilon\right)
\le P\left(\sup_{k\ge \sigma_{m\delta}}U^+_{m,k}\ge \epsilon m/2\right) 
\le P\left(\sup_{k\ge 0}\,U^+_{m,k}\ge \epsilon m/2 \mid U^+_0=\fl{\delta m}\right)\\
=P\left(\tau^{U^+}_{\epsilon m/2}<\sigma^{U^+}_0 \mid U^+_0=\fl{\delta m}\right)
\le P\left(\tau^{U^+}_{\epsilon m/2}\le rm \mid U^+_0=\fl{\delta m}\right)
+P\left(\sigma^{U^+}_0>rm \mid U^+_0=\fl{\delta m}\right).
\end{multline*}
By Lemma~\ref{ub} (see below) and Lemma~\ref{sm0} we can control the last two probabilities
and conclude that
\[\lim_{\ell\to\infty}\limsup_{m\to\infty}P^m\left(d^\circ_\infty(Y_{m,\ell},Y_m)>\epsilon\right)=0. \] By Theorem~\ref{Bil},
  $Y_m\underset{m\to\infty}{\Longrightarrow}Y^{(\infty)}$ as claimed.

  We are left to show \eqref{ht}. By the continuous mapping
    theorem, \cite[Lemma 3.3]{kzEERW}, and the a.s.\ continuity of $Y$
    we have that
  $\sigma^{Y_m}_{\delta}\underset{m\to\infty}{\Longrightarrow}\sigma^Y_\delta\underset{\delta\to
    0}{\Longrightarrow}\sigma_0^Y$. To use Theorem~\ref{Bil}
  again, we need to estimate
  $P(\sigma^{Y_m}_0-\sigma^{Y_m}_{\delta}>\epsilon \mid Y^m_0)$. By the
  strong Markov property and monotonicity in the starting point, this
  probability does not exceed
  $P(\sigma^{U^+}_0>\epsilon m \,|\,U^+_0=\ceil{\delta m})$ which
  converges to $0$ as $\delta\to 0$ by Lemma~\ref{sm0}. Thus,
  $\sigma^{Y_m}_0\underset{m\to\infty}{\Longrightarrow}\sigma^Y_0$.
\end{proof}

The proof of Theorem~\ref{da0.25} depends on several facts which we shall state and prove first. Recall that
$\max\{\theta^+,\theta^-\}<1$. The BLP $Z$ below can be any of the
BLPs $U^\pm$ and $V^\pm$.

\begin{lemma}\label{ub}
  For all $T,\epsilon>0$ there is an $L>0$ such that for an arbitrary fixed selection of the first cookies and for all $m\in \N$
  \[P\left(\max_{k\le Tm
      }Z^m_k\le Lm\,|\, Z_0=m\right)>1-\epsilon.\]
\end{lemma}

\begin{proof}
  By Propositions 3.1, 3.6, 4.1, 4.2 of \cite{kpERWMCS} we have that for all $k\in\N$
  \begin{equation}
    \label{mv}
    |E[Z^m_k\,|\,Z^m_{k-1}]-Z^m_{k-1}|\le \gamma;\quad E[(Z^m_k)^2\,|\,Z^m_{k-1}]\le (Z^m_{k-1})^2+\alpha Z^m_{k-1}+\beta,
  \end{equation}
  where constants $\alpha,\beta,\gamma$ do not depend on $k,m$ or
  a choice of the first cookies. If we
  set \[b_k:=E[(Z^m_k)^2\,|\,Z^m_0=m],\quad a_k:=E[Z^m_k\,|\,Z^m_0=m], \] then
  estimates \eqref{mv} imply that
  \[ a_k\le m+\gamma k,\quad b_k\le b_{k-1}+\alpha\gamma(k-1)+\alpha
    m+\beta.\] We conclude that
  \[E[Z^m_k\,|\,Z^m_0=m]\le m+\gamma k, \quad E[(Z^m_k)^2\,|\,Z^m_0=m]\le
    k(\alpha m +\beta) +\frac12\alpha\gamma k(k-1).\] Let $M^m_0=m$,
  $M^m_k:=Z^m_k-\sum_{j=1}^kE[Z^m_j-Z^m_{j-1}\,|\,Z^m_{j-1}]$,
  $k\in\N$. Then $M^m_k, k\ge 0$, is a martingale with respect to its
  natural filtration. Since
  $|M^m_{\fl{Tm}}-Z^m_{\fl{Tm}}|\le \gamma Tm$, we have that
  \[E[(M^m_{\fl{Tm}})^2]\le
    2E[(Z^m_{\fl{Tm}})^2\,|\,Z^m_0=m]+2(\gamma Tm)^2\le
    C(\alpha,\beta,\gamma,T)m^2.\] By the maximal inequality, for
  $L>\gamma T$ and all $m\in\N$,
  \begin{multline*}
    P\left(\max_{k\le Tm}Z^m_k\ge mL\right)\le P\left(\max_{k\le Tm}|M^m_k|\ge m(L-\gamma T)\right)\le
    \frac{4E[(M_{\fl{mT}})^2]}{(L-\gamma T)^2m^2}\le \frac{4C(\alpha,\beta,\gamma,T)}{(L-\gamma T)^2}.
  \end{multline*}
We can choose $L$ large enough to ensure that the last expression is less than $1-\epsilon$.
\end{proof}

\begin{lemma}\label{ct}
  For each $m\in\N$ let $Z^m$ be 
one of the four kinds of BLPs
%BLPs of one of the four kinds 
and $Z^m_0\le Km$ for some $K>0$. Fix $\epsilon>0$ and
  define
  \[Y^{\epsilon,m}_t:=\frac{Z^m_{\fl{tm}}}{m},\quad
    \tilde{Y}^{\epsilon,m}_t:=\frac{ Z^m_{\fl{\fl{tm^{3/4}}m^{1/4}}} }{m},\quad t\ge 0.\] Then uniformly over all first cookie environments for every $T,\delta>0$ \[P\left(\sup_{0\le t\le T}|\tilde{Y}^{\epsilon,m}_t-Y^{\epsilon,m}_t|>\delta\right)\to 0\quad\text{as }m\to\infty.\]
\end{lemma}
\begin{proof}
  Let $A_L$ be the event that $\max_{j\le Tm}Z^m_j\le Lm$. By
  Lemma~\ref{ub}, given an arbitrary $\epsilon'>0$, there is an $L$ such
  that $P(A_L)>1-\epsilon'$. Denote by $B_k$ the event
  \[\{\forall j\in\bint{1,m^{1/4}}:\,
    |Z^m_{\fl{j+1+(k-1)m^{1/4}}}-Z^m_{\fl{j+(k-1)m^{1/4}}}|\le m^{3/5}\}.\]
  Then by Lemma A.1 from \cite{kpERWMCS} there are $c,C>0$ such
  that \[P(B_k^c\cap A_L)\le Cm^{1/4}e^{-cm^{1/5}/L}.\] We conclude that 
  \begin{equation*}
    P\left(\sup_{0\le t\le T}|\tilde{Y}^{\epsilon,m}_t-Y^{\epsilon,m}_t|>\delta\right)\le P\left(\cup_{k\le Tm^{3/4}}(B_k^c\cap A_L)\right)+P(A_L^c)\le CTme^{-cm^{1/5}/L}+\epsilon'.
  \end{equation*}
Since $\epsilon'$ was arbitrary, the proof is complete.
\end{proof}

The proof of the following lemma is identical to the one of Lemma 7.1 in \cite{kpERWMCS}, and is, thus, omitted.

\begin{lemma}\label{al1}
Let $D\in \mathbb{R}$, $\nu >0$, and $\{Y(t)\}_{t\ge 0}$ be a solution of \eqref{daa}% \footnote{Let
  % us remark that $X(t):=4Y(t)/\nu$ satisfies
  % $dX(t)=(4D/\nu)dt+2\sqrt{X(t)^+}dB(t)$ and, thus, is a squared Bessel
  % process of the generalized dimension $4D/\nu$.}
% \begin{equation}
%   \label{mSDE}
%   dY(t)=D\,dt+\sqrt{\nu\, (Y(t))_+}dB(t),\quad Y(0)=x\ge 0,
% \end{equation}
% where $(B(t))_{t\ge 0}$ is the standard Brownian motion.
with $D(t)\equiv D$ and $Y(0)\sim \kappa$. Let
(time-inhomogeneous countable) Markov chains 
$Z^n_k:=\{ Z^n_k \}_{k\ge 0}$ with values in $\R$ satisfy the following
conditions:
\begin{enumerate}
\item \label{DAas1} for each $T,r>0$ there is a deterministic function $g:\R_+\to\R_+$ such that $g(x)\to 0$ as $x\to\infty$,
\begin{align*}
  &\mathrm{(E)}\quad \max_{1\le k\le (Tn)\wedge (\tau^{Z^n}_{rn}+1)}|E[Z^n_k-Z^n_{k-1}\,|\,Z^n_{k-1}]-D|\le g(n)% =o(1)  \text{as $n\to\infty$}
    ;\\&\mathrm{(V)}\quad
  \max_{1\le k\le (Tn)\wedge (\tau^{Z^n}_{rn}+1)}\Big|\frac{\mathrm{Var}(Z^n_k\,|\,Z^n_{k-1})}{Z^n_{k-1}\vee N_n}- \nu  \Big|% = o(1)  \text{as $n\to\infty$}
   \le g(n)
  \\ &\text{for some sequence $\{N_n\}_{n\in\N}$, $N_n\to\infty$, $N_n=o(n)$ as $n\to\infty$; }
\end{align*}
 \item \label{DAas2} for each $T,r>0$ 
\[
E\left[ \max_{1\le k\le (Tn)\wedge (\tau^{Z^n}_{rn}+1)}(Z^n_k-Z^n_{k-1})^2\right]=o(n^2) \text{ as $n\to\infty$}.
\]  
  \end{enumerate}
  Set % $Z^n_0=\fl{nx_n}$, $x_n\to x$ as $n\to\infty$, and
  $Y_n(t)=n^{-1}Z^n_{\fl{nt}}$, $t\ge 0$, and assume that $Y_n(0)\sim \kappa_n$ where $\kappa_n\underset{n\to\infty}{\Longrightarrow}\kappa$. Then
 $Y_n\overset{J_1}{\underset{n\to\infty}{\Longrightarrow}} Y$.
\end{lemma}

Now we have all ingredients for the proof of Theorem~\ref{da0.25}.

\begin{proof}[Proof of Theorem~\ref{da0.25}] We give a detailed proof
  only for the case $Z^m_j=:V^+_{m,j}$, $j\ge 0$, but the same proof works
  for the other BLPs.
 
  We start by modifying our process $\{V^+_{m,j}\}_{j\ge 0}$. Let
  $N_m\in\N$ satisfy $N_m\to\infty$ and $N_m=o(m^{3/4})$ as
  $m\to\infty$.  We define $\tilde{V}^+_{m,0}=V^+_{m,0}$ and
recalling the representation in \eqref{VkSumG} for $V^+_{m,k}$ we let
\begin{equation}\label{tVkSumG}
\tvmj=\tvmjj+1+\sum_{\ell=1}^{(\tvmjj+1)\vee \fl{N_mm^{1/4}}}(G^j_\ell-1).
\end{equation}
Note that
  the modified process is identical to our original process
  $\{V^+_{m,j}\}_{j\ge 0}$ up to the first entrance time in the
  interval $(-\infty, N_mm^{1/4})$. Given the conditions of our
  theorem, it is enough to prove the result for the modified
  process. For convenience of the reader, we state the expectation and
  variance estimates for $\tilde{V}^+_m$ (Propositions 4.1 and 4.2
  from \cite{kpERWMCS}). For all $m, j\in\N$
  \begin{align}
    \label{exp}
    &|E[\tvmj-\tvmjj \mid \tvmjj]-(r^+(R^j_1)+1)|\le c_{12}e^{-c_{13}(\tvmjj\vee N_mm^{1/4})}\le c_{12}e^{-c_{13}N_mm^{1/4}} =:\epsilon_m;\\\label{var} &|\text{Var}(\tvmj\,|\tvmjj)-\nu(\tvmjj\vee \fl{N_mm^{1/4}})|\le c_{14}.
  \end{align}
We are planning to apply Lemma~\ref{al1} to the process
  $Z_k^n:=m^{-1/4}\tvmmk$ with $n=\fl{m^{3/4}}$ and then
  conclude by Lemma~\ref{ct}. We just need to check the conditions of
  Lemma~\ref{al1}. 

{\em Step 1.}  Given the first cookies on
$\bint{(k-1)m^{1/4},km^{1/4}}$, we get by the properties of
conditional expectation and \eqref{exp} that
  \begin{multline*}
    \Big|E\left[ \tvmmk - \tvmmkk \mid \tvmmkk\right]-\sum_{j=\fl{(k-1)m^{1/4}}+1}^{\fl{km^{1/4}}}(r^+(R^j_1)+1)\Big|\\ 
\le  \sum_{j=\fl{(k-1)m^{1/4}}+1}^{\fl{km^{1/4}}}E\left[\left|E\left[\tvmj-\tvmjj\,|\,\tvmjj\right]-(r^+(R^j_1)+1)\right|\,|\,\tvmmkk\right]\le \epsilon_mm^{1/4}.
  \end{multline*}
Recalling the meaning of the condition that the first cookie environment is $(m^{1/4},\rho)$-good we see that for all $m$ and $k$
\begin{equation}
  \label{1e}
  \left|\frac{1}{m^{1/4}}E\left[\tvmmk-\tvmmkk\,|\,\tvmmkk\right]-(\rho+1)\right|\le \frac{1}{\ln m}+\epsilon_m.
\end{equation}

{\em Step 2.} Our next task is to deal with conditional variance over
intervals $\bint{(k-1)m^{1/4},km^{1/4}}$ for $k\le Tm^{3/4}\wedge \tau_{rm}$ with
arbitrary fixed $T,r>0$. We want to show that
\begin{multline}
  \label{1v}
  \max_{1\le k\le Tm^{3/4}\wedge \tau_{rm}}\left|\text{Var}( \tvmmk \mid \tvmmkk)-\nu \fl{m^{1/4}}(\tvmmkk\vee (N_mm^{1/4}))\right|\\=o(N_mm^{1/2}),
  \end{multline}
  where $\tau_{rm}$ is the first time the process
  $\tilde{V}^+_{m,\fl{km^{1/4}}}, k\ge 0$, enters $(rm,\infty)$.

  Fix an arbitrary $m\in\N$ and $k, 1\le k\le Tm^{3/4}\wedge
  \tau_{rm}$. To simplify the notation, we shall use $V_j$ instead of
  $\tilde{V}^+_{m,\fl{(k-1)m^{1/4}+j}}$ and $V_{j+}$ instead of
  $V_j\vee N_mm^{1/4}$ for $j\in\bint{0,m^{1/4}}$. We shall also
  write $E_0[\cdot]$ and $\text{Var}_0(\cdot)$ instead of
  $E[\cdot\,|\,V_0]$ and $\text{Var}(\cdot\,|\,V_0)$.

  With this notation, the $k$-th term in \eqref{1v} can be estimated as follows:
  \begin{equation}\label{tel}
    \left|\text{Var}_0(V_{\fl{m^{1/4}}})-\nu \fl{m^{1/4}}V_{0+}\right|\le \sum_{j=1}^{\fl{m^{1/4}}}\left|\text{Var}_0(V_j)-\text{Var}_0(V_{j-1})-\nu V_{0+}\right|.
  \end{equation}
  We shall show that for $N_m$ such that $N_m/m^{3/5}\to\infty$
  (retaining the property that $N_m=o(m^{3/4})$) each term in the
  above sum is $o(N_mm^{1/4})$ as $m\to\infty$.

  First we apply the conditional variance formula (conditioning on
  ${\cal F}_{j-1}$ and using the Markov property to replace
  ${\cal F}_{j-1}$ with $V_{j-1}$) and get that
\begin{multline}\label{cv}
  \left|\text{Var}_0(V_j)-\text{Var}_0(V_{j-1})-\nu V_{0+}\right|\ =
  \left|E_0[\text{Var}(V_j\,|\,V_{j-1})]+\text{Var}_0(E(V_j\,|\,V_{j-1}))-\text{Var}_0(V_{j-1})-\nu V_{0+}\right| \\ 
\le |E_0\left[ \text{Var}(V_j\,|\,V_{j-1})-\nu V_{j-1+} \right] |+\nu |E_0(V_{j-1+}-V_{0+})|\\
+\left|\text{Var}_0\left((E[V_j\,|\,V_{j-1}]-V_{j-1})+V_{j-1}\right)-\text{Var}_0\left(V_{j-1}\right)\right|.
\end{multline}
We know from \eqref{exp} that $|E[V_j\,|\,V_{j-1}]-V_{j-1}|\le \alpha$
for some constant $\alpha$. Note that if $|Y|\le \alpha$ then
$\text{Var}(Y)\le \alpha^2$ and
\[|\text{Var}(X+Y)-\text{Var}(X)|\le \alpha^2+2\alpha \sqrt{\text{Var}(X)}.\]
Applying this inequality with $X=V_{j-1}$ and
$Y=E[V_j\,|\,V_{j-1}]-V_{j-1}$ to the last term of
\eqref{cv} and using \eqref{var} to estimate the first term we obtain for some constant $\Cl{aux}>0$
\begin{equation*}
  \left|\text{Var}_0(V_j)-\text{Var}_0(V_{j-1})-\nu V_{0+}\right|\le
  \Cr{aux}+ \nu |E_0(V_{j-1+}-V_{0+})|+2\alpha
  \sqrt{\text{Var}_0(V_{j-1})}.
\end{equation*}
Let
\begin{equation}
  \label{bk}
  B_k=\{\forall j\in\bint{1,m^{1/4}}, |V_j-V_{j-1}|\le m^{3/5}\}.
\end{equation}
 Since we are considering only $k\le Tm^{3/4}\wedge(\tau_{rm}+1)$, we can
assume that $V_0\le rm$. Then by Lemma A.1 from \cite{kpERWMCS} there are
$c,C > 0$ such that
\[P(B_k^c)\le Cm^{1/4}e^{-cm^{1/5}/(16r)}.\] Recall that $N_m/m^{3/5}\to\infty$ and $N_m=o(m^{3/4})$ as $m\to\infty$. If $V_0\ge N_mm^{1/4}$ then on the set $B_k$ \[|V_{j+}-V_{0+}|=|V_{j+}-V_0|\le m^{1/4}m^{3/5}=o(N_mm^{1/4}),\]  and if $V_0<N_mm^{1/4}$ then on 
$B_k$
\[|V_{j+}-V_{0+}|=|V_{j+}-\fl{N_mm^{1/4}}|\le
  m^{1/4}m^{3/5}\ind{V_{j+}>N_mm^{1/4}}=o(N_mm^{1/4}).\]
Using these estimates we get 
\begin{align*}
  &\left|\text{Var}_0(V_j)-\text{Var}_0(V_{j-1})-\nu V_{0+}\right|\\ &\le
  \Cr{aux}+ \nu |E_0[(V_{j-1+}-V_{0+})\ind{B_k}]|+\nu
  |E_0[(V_{j-1+}-V_{0+})\ind{B_k^c}]|+ 2\alpha
  \sqrt{\text{Var}_0(V_{j-1})}\\ &\le o(N_mm^{1/4})+\nu\sqrt{E_0[(V_{j-1+}-V_{0+})^2]P(B_k^c)}+2\alpha\sqrt{E_0[(V_{j-1}-V_0)^2]}.
\end{align*}
Now we observe that
\[(V_{j-1+}-V_{0+})^2\le 3((V_{j-1+}-V_{j-1})^2+(V_{j-1}-V_0)^2+(V_0-V_{0+})^2),\] where
$0\le V_{i+}-V_i\le N_mm^{1/4}$ for all $i$. Taking into account
a stretched exponential decay of $P(B_k^c)$  we arrive at the inequality
\[\left|\text{Var}_0(V_j)-\text{Var}_0(V_{j-1})-\nu V_{0+}\right|\le
  o(N_mm^{1/4})+2(\nu+\alpha)\sqrt{E_0[(V_{j-1}-V_0)^2]}.\] To bound the last term, we let $j\in\bint{1,m^{1/4}}$ and use \eqref{exp}, \eqref{var} to obtain
\begin{multline}\label{sm}
  E_0[(V_j-V_0)^2] \le j\sum_{i=1}^jE_0\left[(V_i-V_{i-1})^2\right] = j\sum_{i=1}^jE_0\left[E\left[(V_i-V_{i-1})^2|\,V_{i-1}\right]\right] \\ \le j\sum_{i=1}^jE_0\left|\text{Var}(V_i-V_{i-1}\,|\,V_{i-1})-\nu V_{i-1+}\right|+j\sum_{i=1}^jE_0\left[\left(E\left[V_i-V_{i-1}\,|\,V_{i-1}\right]\right)^2\right]+j\nu\sum_{i=1}^jE_0[V_{i-1+}]\\ \le \Cl{15} m^{1/2}+j\nu\sum_{i=1}^j(E_0[V_{i-1+}-V_j]+E_0[V_j-V_0])+\nu m^{1/2}V_0=O(m^{3/2}),
\end{multline}
where $\Cr{15}$ is some fixed constant appropriately larger than $c_{14}$.
This implies that the right hand side of \eqref{tel} is $o(N_mm^{1/4})$
and, thus, completes the proof of \eqref{1v}.

{\em Step 3.} We need to show that
\begin{equation}
  \label{ii}
  E\left(\max_{1\le k\le Tm^{3/4}\wedge
      (\tau_{rm}+1)}(\tvmmk-\tvmmkk)^2\right)=o(m^2).
\end{equation}
Let $B_k$ be defined as in \eqref{bk}. Then the right hand side of the above expression is equal to
\begin{align*}
  E&\left[ \max_{1\le k\le Tm^{3/4}\wedge
      (\tau_{rm}+1)}\left\{ (\tvmmk-\tvmmkk)^2\left(\ind{B_k}+\ind{B_k^c}\right)\right\} \right] \\ 
&\le (m^{3/5+1/4})^2 +Tm^{3/4}\max_{1\le k\le
    Tm^{3/4}\wedge
    (\tau_{rm}+1)}E\left[\left(\tvmmk-\tvmmkk\right)^2\ind{B^c_k}\right]\\ 
&\le o(m^2) +Tm^{3/4}\max_{1\le
    k\le Tm^{3/4}\wedge
    (\tau_{rm}+1)}\left(E\left[\left(\tvmmk-\tvmmkk\right)^4\right]\right)^{1/2}\left(P\left(B^c_k\right)\right)^{1/2}.
\end{align*}
Given the stretched exponential decay of the last probability, any
polynomial in $m$ bound  on the 4-th moment above will suffice.

Fix an arbitrary $k,\ 1\le k\le Tm^{3/4}\wedge (\tau_{rm}+1)$ and
recall our shortcut notation from the previous step. For each
$j\in\bint{1,m^{1/4}}$, 
using the representation in \eqref{tVkSumG} together with Lemma A.3 from \cite{kpERWMCS} we can obtain that 
\begin{align*}
  E\left[\left(V_j-V_{j-1}\right)^4\right]
&=E\left[E\left[\left(V_j-V_{j-1}\right)^4\big|\,V_{j-1}\right]\right] \\
&\leq \Cl{16} E\left[ ((V_{j-1}+1)\vee N_mm^{1/4})^2 \right] 
\leq \Cr{16} E\left[ V_{j-1}^2 \right] + o(m^2). 
\end{align*}

Finally, by
\eqref{sm},
\[E[V_{j-1}^2|\,V_0]\le 2E[(V_{j-1}-V_0)^2|\,V_0]+2V_0^2\le
  O(m^{3/2})+2(rm)^2.\] Collecting all these estimates we get a
desired polynomial bound, and we are done.

{\em Step 4.} Estimates \eqref{1e}, \eqref{1v}, and \eqref{ii} imply that the process $Z_k^n=m^{-1/4}\tvmmk$ with $n=\fl{m^{3/4}}$ satisfies the conditions of Lemma~\ref{al1} with $D=1+\rho$. An application of Lemma~\ref{al1} and Lemma~\ref{ct} completes the proof.  
\end{proof}

\subsection{Other results needed}

In the proof of Lemma \ref{time}, we need some large deviation estimates for the supremum of a concatenation of BLPs. 
We show this below as a corollary of an analogous result for concatenation of BESQ processes. 

\begin{lemma}
  \label{lem0n}
  Let $(Y(t))_{t\ge 0}$ be a solution of
  \[dY(t)=D(t)\,dt+\sqrt{\nu (Y(t))_+}dB(t),\quad 0\le t\le T,\quad
    Y(0)=y\in(0,T],\] where $\nu > 0$ and  $D:[0,T]\to\R$ is a piecewise constant
  non-random function bounded above by some $d>0$. Then
  there exist $\Cl{sd},\Cl{d}>0$ (which depend on $d$ and $\nu$ but
  not on $y$ and $T$) such that
  \[
P \left( \sup _ {t\le T} Y(t) \ge  x T \right)  \le \Cr{d} e ^{- \Cr{sd} x}\quad\text{for all }x\ge 0.
\]

%Let $Y_. $ be the concatenation of $2$ or $3$ $BESQ$ processes, defined on $[\delta , 2 \epsilon] $ for $ 0 \leq \delta < \epsilon ]$ and taking initial value, $Y( \delta) = \delta  < \epsilon $.  Let $d$ be the maximal value of the $2$ or $3$ drifts of the component $BESQ$ processes.

%Then $\exists \Cr{sd} , \Cr{d} \ \in \ (0, \infty ) $ (not depending on $\delta, \delta ' $ or $\epsilon $) so that $\forall x \geq 0$,
%\[
%P \left( \sup _ {\delta < s < 2 \epsilon} Y(s) \geq  x \epsilon \right)  \leq \Cr{d} e ^{- \Cr{sd} x}
%\]
\end{lemma}

\begin{proof}
Without loss of generality we can assume that $x \geq 2$. 
By the
comparison theorem for one-dimensional SDEs the process $4Y/\nu$ is
stochastically dominated by a BESQ$^{\ceil{4d/\nu}}(4y/\nu)$ process. The last
process is just $4y/\nu$ plus the sum of squares of $\ceil{4d/\nu}$
independent one-dimensional Brownian motions. Therefore, the probability in
question does not exceed
  \[P\left(\max_{t\le T}\sum_{i=1}^{\ceil{4d/\nu}}B_i^2(t)\ge
      \frac{4(T x-y)}{\nu}\right)\le \ceil{4d/\nu}P\left(\max_{t\le
        T}|B(s)|\ge \sqrt{\frac{2T
          x}{\nu\ceil{4d/\nu}}}\right)\le \Cr{d}e^{- \Cr{sd}x}.\] 
% by the
%comparison theorem for one-dimensional SDEs.  
%the process $4Y/\nu$ is
%stochastically dominated by a BESQ$^{\ceil{4d/\nu}}(4y/\nu)$ process. 
%In this case $Y(.) $ may be written as 
%\[
%Y(t) \ = \ \sum_{i=1}^d B_i (t) ^ 2
%\]
%where the $B_i $ are i.i.d. speed $\frac{\nu}{4} $ one-dimensional Brownian motions starting at value $y/d$.
 
%Therefore, the probability in
%question does not exceed
%  \[
 % P\left( \max_{t\le T}\sum_{i=1}^{d} B_i^2(t) \ge
   %  (T x-y) \right)  \ \le \ d P\left( \max_{t\le T}
   %     |B_1(t)|\ge \sqrt{\frac{Tx-y}{d} } 
    % ^    \right) \le \Cr{d}e^{- \Cr{sd}x}.
   %      \]
\end{proof}

%This immediately begets
\begin{cor}
\label{C0}
For $m\in\N$ let $\{Z^m_j\}_{j \geq 0} $ be a BLP starting from
$0$ that is the concatenation of $V^+ $ and then two $U^+$ processes
on $3$ intervals $I_1, I_2 $ and $I_3 $ where
$I_1 \cup I_2 \cup I_3 = \bint{ 0, 2\epsilon m} $ and assume that
the first cookie environment on $I_1$ is
$(m^{1/4}, \frac{\nu}{2}-1)$-good, the first cookie environment on
$I_2$ is $(m^{1/4}, 0)$-good) and the first cookie environment on
$I_3$ is i.i.d.\ with distribution $\eta$.

Then  for $\Cr{sd},\Cr{d} $ as in Lemma \ref{lem0n} we have that for every $K <\infty ,$ there exists $m_0 (K)< \infty $ such that 
 \[
 P \Big( \sup_{ j \leq 2 \epsilon m} Z^m_j\ge 2\epsilon m x\Big) \leq 2 \Cr{d} e ^{-\Cr{sd}x}\quad  \text{for all $m\ge m_0 (K) $  and $x \leq K$.}
\]
 \end{cor}

 \begin{proof} We fix $K\in(0,\infty)$.  Though the interest in the
   corollary is for BLPs starting at value $0$, by monotonicity of
   these processes, it is enough to show the desired result for BLPs
   satisfying $Z^m_0 = \fl{ \epsilon m } $.  We argue by
     contradiction and suppose that the result is not true.  This
     implies the existence of a sequence $\{m_k\}_{k \ge 0}$,
     intervals $I_1^{m_k}$, $I_2^{m_k}$, and $I_3^{m_k}$ partitioning
     $ \bint{ 0, 2\epsilon m} $ and suitable $m_k$ indexed
     environments satisfying the stated hypotheses on these intervals
     so that the stated probability bound is violated for all $k$.  Taking
     a subsequence if needed we may suppose that, in the obvious
     sense, that the intervals $I_j^{m_k}$ divided by $\epsilon m_k $
     converge to intervals $I_j$ for $j=1,2 $ and $3$.  In the
     following, to avoid a burdensome notation, we write $m_k$ as $m$.
     It is sufficient to show that under these conditions the claimed
     probability bounds hold.

By Theorem~\ref{da0.25} , Corollary \ref{da0.25dead} and then Theorem \ref{DA0}, the
   processes $\{m^{-1}Z^m_{\fl{ms}}\}_{s \geq 0}$ converge weakly % in $J_1 $
   % topology
   to a concatenation of a $ \frac{\nu}{4}$\,BESQ$^2$ process starting
   at value $\epsilon$ (on interval $I_1$) with a
   $\frac{\nu}{4}$\,BESQ$^{0}$ process on $I_2$ and then a
   $\frac{\nu}{4}$\,BESQ$^{2 \theta_+}$ process on $I_3$. Note that for the
 interval $I_1$, Theorem \ref{da0.25} suffices since
   a BESQ$^2$ process starting at $\epsilon$ never hits zero.  Lemma
   \ref{lem0n} is applicable to this limit process, and we get that for every
   $x \geq 0$, $\limsup _{m\to\infty} P( \sup_{ j \leq 2 \epsilon m}
   Z^m_j \ge 2\epsilon m x ) \leq \Cr{d} e ^{-\Cr{sd}x }$.
%We write $\tau \ = \ \inf \{k : Z_k \geq m \epsilon / 2 \}$.  We first note that if $\tau \ > \ 2 \epsilon m $, then trivially $\sup _k Z_k \ \le \ \epsilon m /2$ so this case can be discounted.
%If $ \tau > \frac{3}{2} \epsilon m $, we extend $J_3 $ by $\epsilon m /2 $.

%Next we note that the event $\{Z_\tau \ \ge \ m \epsilon,\, \tau \leq 2 \epsilon m\} $ is contained in the union of events $\cup_{i=1} ^ {2 \epsilon m - 1} \{Z_{i-1} \ < \  \epsilon m /2 , Z_i \ \ge \ \epsilon m  \} $.
%By Lemma A.1 of \cite{kpERWMCS} this latter event has probability bounded by $2 \epsilon m e^{-c \epsilon m} $ for some nontrivial $c$ and all $m$.  Again this event may be discarded. We now apply Lemma \ref{lem0n} with 
%$y \ = \ \epsilon m $ and $T \ = \ 2 \epsilon - \tau / m \wedge \epsilon / 2$ (using the product environment on the extension of $J_3 $ if necessary).  Lemma \ref{dead} and Propositions \ref{da0.25} and \ref{DA0} imply that $\frac{Z_{\tau + \fl{sm}} }{m}$ converges in distribution on $[0,T] $ to a concatenation of BESQ
%processes.  
%These limit processes define a bound $d$ and we can apply Lemma \ref{lem0n}, given the weak convergence to obtain that $\forall \ x \ \in \ [\frac{1}{2} ,K] $
%\[
%P \left( \sup _ {k\le 2 \epsilon m }Z_k \ge  2x m \epsilon \right)  \  < \  P \left( \sup _ {k \le mT} Z_{\tau +k } \ge  x T m \right)  \le \  \frac{4}{3}\Crd} e ^{- \Cr{sd} x}
%\]
%for $m$ large enough.  Here $\Cr{sd}, \ \Cr{d} $ are constants that correspond to $T, \ y = 0 $ and $d$.
   To complete the proof we % now simply increase $\Cr{d} $ if necessary so that $\frac{4}{3}\Cr{d} e ^{- \Cr{sd} } \ > \ 1$ and
   take $0=x_0 < x _1 < \ldots < x_r= K$ so that
   $\forall i,\ x_i- x_{i-1} < \delta $ where
   $ e^ {- \Cr{sd} \delta } < 3/2$.  For $m$ sufficiently
   large and all $x_i$, $i\in\bint{0,r}$, we have % to obtain that for $m$
   % sufficiently large
   $P( \sup_{ j \leq 2 \epsilon m} Z^m_j \ge 2\epsilon m x_i ) \leq 
   \frac{4}{3}\, \Cr{d} e ^{-\Cr{sd}x_i }$ and so for such $m$ by
   monotonicity
\[
\forall x \le K, \ P \left( \sup_{j \le 2 \epsilon m }Z^m_j \ge  2\epsilon m x\right)   \le  \frac{4}{3}\,\Cr{d}\,e ^{- \Cr{sd} (x-\delta)}\le 2\Cr{d} e ^{- \Cr{sd} x}.\qedhere
\]
 \end{proof}

Finally, we need the following general lemma about couplings which is used in the proof of Lemma \ref{lemH}.
For this, recall the definition of the family of probability measures $\mathcal{H}_{\delta,\epsilon}$ in Definition \ref{Hde}.

\begin{lemma}\label{triple}
    For every $\lambda\in \cal{H}_{\delta, \epsilon}$ there is a
    coupling $\nu$ of probability measures $\lambda$ and $\lambda_0$
    such that $\nu(\{(x,y)\in\R^2:\,|x-y|>\delta\})< 8\epsilon^3$.
  \end{lemma}

  \begin{proof}
    We shall construct a random vector
    $(\zeta,\zeta^{(0)},\zeta^{(1)})$ with respective marginal
    distributions $\lambda,\lambda_0,\lambda_1$ so that
    $P(|\zeta-\zeta^{(0)}|>\delta)< 8\epsilon^3$. Then $\nu$ is the
    joint distribution of $(\zeta,\zeta^{(0)})$.

Recall that $\lambda \in \cal{H}_{\delta,\epsilon}$ can be represented as $\lambda = \int K(z, \cdot)\, \lambda_1(dz)$ 
with $K$ and $\lambda_1$ satisfying the conditions in Definition \ref{Hde}. 
    Let $\nu_0$ be a maximal coupling of $\lambda_0$ and $\lambda_1$
and
 $(\zeta^{(0)},\zeta^{(1)})$ be a random vector with
    distribution $\nu_0$. Then
    \[\nu_0(\{(y,z)\in\R:\,y\ne
      z\})=P(\zeta^{(0)}\ne
      \zeta^{(1)})=\|\zeta^{(0)}-\zeta^{(1)}\|_{TV}< 8\epsilon^3.\]
    Denote the regular conditional probability distribution of
    $\zeta^{(0)}$ given $\zeta^{(1)}=z$ by $K_0(z,\cdot)$. We
    construct $(\zeta,\zeta^{(0)},\zeta^{(1)})$ as follows.
    \begin{itemize}
    \item draw $\zeta^{(1)}$ according to $\lambda_1$;
    \item given $\zeta^{(1)}=z$, draw $\zeta$ from $K(z,\cdot)$
    and $\zeta^{(0)}$ from $K_0(y,\cdot)$ independently from each other.
  \end{itemize}
  We have
  \begin{multline*}
    P(|\zeta-\zeta^{(0)}|>\delta)=P(|\zeta-\zeta^{(1)}|>\delta,\ \zeta^{(0)}=\zeta^{(1)})+P(|\zeta-\zeta^{(0)}|>\delta,\ \zeta^{(0)}\ne\zeta^{(1)})\\ \le P(|\zeta-\zeta^{(1)}|>\delta)+P(\zeta^{(0)}\ne\zeta^{(1)})= \int K(z,[z-\delta,z+\delta]^c)\lambda_1(dz)+\nu_0(\{(y,z)\in\R:\,y\ne
      z\})< 8\epsilon^3.
  \end{multline*}
\end{proof}

\bibliographystyle{alpha}
\bibliography{CookieRW}
\end{document}